\documentclass[a4paper, 11pt]{article}
\usepackage{fullpage}
\newcommand{\DRAFT}[1]{#1}
\newcommand{\FINAL}[1]{ }

\usepackage{mathrsfs}
\usepackage{bm,color}
\usepackage{dblfloatfix} 
\usepackage{algorithm}
\usepackage{amssymb,amsmath,amsthm,amsfonts,graphicx,epsfig}
\usepackage{xspace}
\usepackage{array}
\usepackage{color}
\usepackage{float}
\usepackage{epstopdf}
\usepackage[table]{xcolor}
\usepackage{multicol}
\usepackage{multirow}
\usepackage{rotating}
\usepackage{cite}
\usepackage{subfig}
\usepackage{caption}
\usepackage{lettrine}
\usepackage{hhline}
\newcommand{\Prob}{\mathsf{p}}

\newtheorem{myprop}{Proposition}[section]
\newtheorem{myassum}{Assumption}[section]

\newcommand{\eR}{\mathbb{R}}

\def \y{\mathbf{y}}
\def \x{\mathbf{x}}
\def \w{\mathbf{w}}
\def \u{\mathbf{u}}

\def \bLambda{\bm{\Lambda}}

\def \H{\mathbf{H}}

\def \KL{\mathcal{KL}}
\def \argmin{\arg \min}
\def \argmax{\arg \max}

\def \dd {\mathrm{d}}

\def \D{\mathbf{D}}

\def \bSigma{\bm{\Sigma}}
\def \m{\mathbf{m}}
\def \Diag{\text{Diag}}

\def \EE{\mathbb{E}}

\renewcommand{\leq}{\ensuremath{\leqslant}}
\renewcommand{\argmin}[2]{\ensuremath{\underset{\substack{{#1}}}%
{\text{\rm argmin}}\;\;#2 }}
\renewcommand{\argmax}[2]{\ensuremath{\underset{\substack{{#1}}}%
{\text{\rm argmax}}\;\;#2 }}

\newcommand{\compresslist}{
  \setlength{\itemsep}{1pt}
  \setlength{\parskip}{0pt}
  \setlength{\parsep}{0pt}}

\newcommand{\f}[1]{{\color{black}#1}}

\renewcommand{\eR}{\mathbb{R}}
\newcommand{\X}{\mathbf{X}}
\newcommand{\G}{{\Gamma}}

\title{A Variational Bayesian Approach for Image Restoration. Application to Image Deblurring with Poisson-Gaussian Noise.}

\author{Yosra~Marnissi, Yuling~Zheng, Emilie~Chouzenoux and~Jean-Christophe~Pesquet
\thanks{Y. Marnissi and E. Chouzenoux are with University Paris-Est, LIGM, UMR CNRS 8049, 77454 Marne-la-Vall\' ee, 
France (e-mail: first.last@u-pem.fr). Y. Zheng is with IBM Research China. Building 19, Zhongguancun Software Park, 8 Dongbeiwang Western Road, Haidian District, Beijing, 11 100193, CN, 
China (e-mail: ylzh@cn.ibm.com). J.-C. Pesquet is with Center for Visual Computing, CentraleSupelec, University Paris-Saclay, Grande Voie des Vignes,
92295 Ch\^atenay-Malabry, France (e-mail: jean-christophe@pesquet.eu). This work was supported by the CNRS Imag'in project under
grant 2015OPTIMISME, by the CNRS MASTODONS project under grant 2016TABASCO, and by the Agence Nationale de la
Recherche under grant ANR-14-CE27-0001 GRAPHSIP.}
}
 
\begin{document}
 
\maketitle

\begin{abstract}
In this paper, a methodology is investigated for signal recovery in the presence of non-Gaussian noise. In contrast with regularized minimization approaches often adopted in the literature, in our algorithm the regularization parameter is reliably estimated from the observations.  As the posterior density of the unknown parameters is analytically intractable, the estimation problem is derived in a variational Bayesian framework where the goal is to provide a good approximation to the posterior distribution in order to compute posterior mean estimates. 
Moreover, a majorization technique is employed to circumvent the difficulties raised by the intricate forms of the non-Gaussian likelihood and of the prior density. We demonstrate the potential of the proposed approach through comparisons with state-of-the-art techniques that are specifically tailored to signal recovery in the presence of mixed Poisson-Gaussian noise. Results show that the proposed approach is efficient and achieves performance comparable with other methods where the regularization parameter is manually tuned from the ground truth.
\end{abstract}

{\small \emph{Keywords}:
inverse problems, restoration, Variational Bayesian methods, parameter estimation, Poisson-Gaussian noise, majorization, minimization.
}

\section{Introduction} \label{Sect:Intro}
One of the most challenging tasks in signal processing is restoration, where one aims at providing an accurate estimate of the original signal from degraded observations. These degradations may arise due to various phenomena which are often unavoidable in practical situations. Undesirable blurring may be introduced by the atmosphere or may also stem from the intrinsic limitations of the acquisition system characterized by its point spread function. Furthermore, data can be perturbed by noise which can be viewed as a parasite signal added to the information of interest, hence altering the extraction of this information. Noise may originate from various sources. On the one hand, sensors generally suffer from internal fluctuations referred to electrical or thermal noise. This type of noise is additive, independent of the signal of interest, and it can be modeled by a Gaussian distribution. On the other hand, it has been experimentally proven that in many situations, the signal of interest may suffer from noise with more complex characteristics. In fact, many devices lead to measurements distorted by heteroscedastic noise i.e. the characteristics of the noise depend on the characteristics of the unknown signal \cite{healey1994radiometric, tian2001analysis,janesick2007photon,azzari2014gaussian,clar:Liu_2014,Chakrabarti,boubchir2014undecimated}. For example, to better reflect the physical properties of optical communication, the noise remains additive Gaussian but its variance is assumed to be dependent on the unknown signal \cite{moser2012capacity}. Signals can also be corrupted with multiplicative noise \cite{aubert2008,buades2005non,dong2013convex,huang2009new} such as the speckle noise which commonly affects synthetic aperture radar (SAR), medical ultrasound and optical coherence tomography images \cite{parrilli2012nonlocal}, as well as with impulsive noise \cite{cai2010fast}. A mixture of Gaussian and impulsive noise has also been studied in \cite{xiao2011restoration,yan2013restoration}. Furthermore, in applications such as astronomy, medicine, and fluorescence microscopy where signals are acquired via photon counting devices, like CMOS and CCD cameras, the number of collected photons is related to some non-additive counting errors resulting in a shot noise \cite{boulanger:hal-01274259,healey1994radiometric,tian2001analysis,janesick2007photon,petropulu}. The latter is non-additive, signal-dependent and it can be modeled by a Poisson distribution \cite{salmon2014poisson,setzer2010deblurring,jeong2013frame,harizanov2013epigraphical,dupe2012deconvolution,chaux2009nested
,bonettini2015new,anthoine2011some,bonettini2011alternating,Altmann2016, rond2015poisson, lefkimmiatis2013poisson, harmany2009sparse, danielyan2011deblurring, pustelnik2011parallel}. In this case, when the noise is assumed to be Poisson distributed, the implicit assumption is that Poisson noise dominates over all other noise kinds. Otherwise, the involved noise is a combination of Poisson and Gaussian (PG) components \cite{Julliand2015,Kingsbury2014,foi2014,Makitalo_M_2012_picassp_Poisson_Fdueuigat,foi2013,Luisier2010a,clar:Li_2014,clar:Lanteri_2005,clar:Jezierska_2012,clar:Jezierska_2011,clar:Foi_2008,clar:Chouzenoux_2015,kittisuwan2016medical,Calatroni2016}. Most of the existing denoising methods only consider this noise as independent Gaussian, mainly because of the difficulties raised in handling other noise sources than the Gaussian one. In this paper, we focus on signal recovery beyond the standard additive independent Gaussian noise assumption. 
\subsection{State-of-the-art}
Existing strategies for solving inverse problems often define the estimate as a minimizer of an appropriate cost function. The latter is composed of two terms: the so-called data fidelity term whose role is to make the solution consistent with the observation and the regularization term that incorporates prior information about the target signal so as to ensure the stability of the solution \cite{demoment1989image}. Several algorithms have been proposed to tackle the problem of restoration for signals corrupted with non-Gaussian noise by using minimization approaches. 
For example, in \cite{clar:Repetti_2012}, a method is proposed to restore signals
degraded by a linear operator and corrupted with an additive Gaussian noise having a signal-dependent variance. An early work in~\cite{clar:Snyder_1993} and more recent developments in~\cite{clar:Lanteri_2005,clar:Benvenuto_2008,clar:Jezierska_2012,clar:Li_2014,Kingsbury2014,clar:Chouzenoux_2015,baji2016blind}
have proposed to restore signals corrupted with mixed PG noise using different approximations of the PG data fidelity term. In all these approaches, the regularization parameter allows a tradeoff  to be performed
between fidelity to the observations and the prior information. Too small values of this parameter may lead to noisy estimates while too large values yield oversmoothed solutions. Consequently, the problem of setting a proper value of the regularization parameter should be addressed carefully and may depend on both the properties of the observations and the statistics of the target signal. When ground truth is available, one can choose  the value of the regularization parameter that 
gives the minimal residual error evaluated through some suitable metric. However, in real applications where no ground truth is available, the problem of selecting the regularization parameter remains an open issue especially in situations where the images are acquired under poor conditions i.e. when the noise level is very high. Among existing approaches dealing with
regularization parameter estimation, the works in \cite{clar:Ramani_2008,clar:Eldar_2009,PesquetBenazzaChaux09,deledalle2014stein,clar:Giryes_2011,clar:Almeida_2013,clar:Hansen_2006}
have to be mentioned.
However, most of the mentioned methods were developed under the assumption of a Gaussian noise and their extension to the context of non-Gaussian noise is not easy. One can however cite the works in \cite{Luisier2010a,Makitalo_M_2012_picassp_Poisson_Fdueuigat} 
proposing efficient estimators in the context of denoising
i.e. problems that do not involve linear degradation. Other approaches can be found in \cite{bertero2010discrepancy, zanni2015numerical} proposing efficient estimates in the specific case of a Poisson likelihood.

To address the shortcomings of these methods, one can adopt the Bayesian framework. In particular, Bayesian estimation methods based on Markov Chain Monte Carlo (MCMC) sampling algorithms have been recently extended to inverse problems involving non-Gaussian noise \cite{ying2012blind,Altmann2015,murphy2011joint,chaari}. However, despite good estimation performance that has been obtained, such methods remain computationally expensive for large scale problems. Another alternative approach is to rely on variational Bayesian approximation (VBA) \cite{vsmidl2006variational,zhengTIP2014,babacan2011variational,chen2014variational,dremeau2012boltzmann}. Instead of simulating from the true posterior distribution, VBA approaches aim at approximating the intractable true posterior distribution with a tractable one from which the posterior mean can be easily computed. These methods can lead generally to a relatively low computational complexity when compared with sampling based algorithms. 

\subsection{Contributions}
In this paper, we propose a VBA estimation approach for signals degraded by an arbitrary linear operator and corrupted with non-Gaussian noise. One of the main advantages of the proposed method is that it allows us to jointly estimate the original signal and the required regularization parameter from the observed data by providing good approximations of the Minimum Mean Square Estimator (MMSE) for the problem of interest. While using VBA, the main difficulty arising in the non-Gaussian case is that the involved likelihood and the prior density may have a complicated form and are not necessarily conjugate. To address this problem, a majorization technique is adopted providing a tractable VBA solution for non-conjugate distributions. 
Our approach allows us to employ a wide class of a priori distributions accounting for the possible sparsity of the target signal after some appropriate linear transformation. Moreover, it can be easily applied to several non Gaussian likelihoods that have been widely used\cite{clar:Chouzenoux_2015,Kingsbury2014, clar:Murtagh_1995,clar:Starck_1998}.
In particular, experiments in the case of images corrupted by PG noise showcase the good performance of our approach compared with methods using the discrepancy principle for estimating the regularization parameter \cite{bardsley2009regularization}. Moreover, we propose variants of our method leading to a significant reduction of the computational cost while maintaining a satisfactory restoration quality.   
\subsection{Outline}
This paper is organized as follows. In Section \ref{sec:Pro_state}, we formulate the considered signal recovery problem in the Bayesian framework and we present a short overview on the variational Bayesian principle. In Section \ref{sec:Algo}, we present our proposed estimation method based on VBA. In Section \ref{Sect:results}, we provide simulation results together with comparisons with state-of-the-art methods in terms of image restoration performance and computation time. Finally, some conclusions are drawn in Section~\ref{sec:Conclu}.
\section{Problem statement}
\label{sec:Pro_state}
\subsection{Bayesian formulation}
In this paper, we consider a wide range of applications where the degradation model can be formulated as an inverse possibly ill-posed problem as follows:
\begin{equation}
\left( \forall i \in \{1, \ldots, M \} \right), \quad y_i= \mathcal{D}(\left[\mathbf{H} \overline{\x}\right]_i),
\end{equation}
where $\overline{\x} \in \mathbb{R}^N$ is the signal of interest, $\mathbf{H} \in \mathbb{R}^{M\times N}$ is the linear operator typically modeling a blur, a projection, or a combination of both degradations, $[\mathbf{H}\overline{\x}]_i$ denotes the $i$-th component of $\mathbf{H}\overline{\x}$ , $\y=(y_i)_{1 \leqslant i\leqslant M } \in \mathbb{R}^M$ is the measured data and $\mathcal{D}$ is the noise model that may depend on the target data. The objective is to find an estimator $\hat{\x}$ of $\overline{\x}$ from $\mathbf{H}$ and $\y$. The neg-log-likelihood  $\phi$ of the observations reads 
\begin{equation}
(\forall \mathbf{x} \in \mathbb{R}^N)\quad
\phi(\mathbf{x}; \y)=-\ln \Prob(\y|\x) = \sum\limits_{i=1}^M \phi_i(\left[ \mathbf{H} \mathbf{x} \right]_i;y_i). 
\label{e:loglikelihood}
\end{equation}
Depending on the noise statistical model $\mathcal{D}$, $\phi_i$ may take various forms \cite{janesick2007photon,Makitalo_M_2012_picassp_Poisson_Fdueuigat,cai2010fast,yan2013restoration}. In particular, it reduces to a least squares function for additive Gaussian noise.
In the Bayesian framework, we apply regularization  by assigning a prior distribution to the data $\x$ to be recovered. In this paper, we adopt the following flexible expression of the prior density of $\overline{\mathbf{x}}$:
\begin{align}
\Prob(\mathbf{x}\mid \gamma)= \tau \gamma^{\frac{N}{2\kappa} } \exp \Big(-\gamma \sum\limits_{j=1}^J\Vert \mathbf{D}_j\mathbf{x} \Vert^{2 \kappa} \Big),
\label{prior_sig}
\end{align}
where $\kappa$ is a constant in $(0,1]$, $\Vert \cdot \Vert$ denotes the $\ell_2$-norm and $(\mathbf{D}_j)_{1\leq j \leq J} \in (\mathbb{R}^{S\times N})^J$ where $\mathbf{D} = [\mathbf{D}_1^\top,\ldots,\mathbf{D}_J^\top]^\top$ is \f{a linear operator. For instance, $\mathbf{D}$ may be a matrix computing the horizontal and vertical discrete difference between neighboring pixels so that $J = N$ and $S = 2$. A sparsity prior in
an analysis frame can also be modeled by setting $J=1$ and $\mathbf{D}$ equals to a frame operator with decomposition size $S \geq N$~\cite{chapitre_livre_wiley}. Other examples will be given in Section \ref{Sect:results}. Note that the constant $\gamma\in (0,+\infty)$ can be viewed as a regularization parameter that plays a prominent role in the restoration process and $\tau\in (0,+\infty)$ is a constant independent of $\gamma$. The form of the partition function for such a prior distribution, i.e. the normalizing factor $\tau \gamma^{N/2k}$, follows from the fact that the associated potential is $2\kappa$-homogeneous \cite{Pereyra_2015eusipco}.}

In this paper, the noise shape parameter $\kappa$ is chosen to be fixed through empirical methods and we aim at estimating parameter $\gamma$ together with $\x$. To this end, we choose a Gamma prior for $\gamma$, i.e. $\Prob(\gamma) \propto \gamma^{\alpha-1} \exp(-\beta \gamma)$ where $\alpha$ and $\beta$ are positive constants (set in practice to small values to ensure a weakly informative prior). 

Using the Bayes' rule, we can obtain the posterior distribution of the set of unknown variables $\mathbf{\Theta}=(\mathbf{x}, \gamma)$ given the vector of observations $\mathbf{y}$:
\begin{equation}
\Prob(\mathbf{\Theta} \mid \y)\propto \Prob(\y \mid\x) \Prob(\x \mid \gamma) \Prob(\gamma).
\end{equation}
However, this distribution has an intricate form. In particular, its normalization constant does not have a closed form expression. To cope with this problem, we resort to the variational Bayesian framework. The rationale of this work is to find a simple approximation to the true posterior distortion, leading to a tractable computation of the posterior mean estimate.
\subsection{Variational Bayes principle}
The variational Bayes approach has been first introduced in physics \cite{parisi1988statistical}. The idea behind it is to approximate the posterior distribution 
$\Prob(\mathbf{\Theta} \mid \y)$ with another
 distribution 
 denoted by $q(\mathbf{\Theta})$ which is as close as possible to $\Prob(\mathbf{\Theta}\mid \y)$, by minimizing the Kullback-Leibler divergence between them \cite{vsmidl2006variational,chapitre_livre_wiley,tutoriel_JSTSP}:
\begin{equation}
q^{\rm opt} = \argmin{q} \KL\big(q(\mathbf{\Theta}) \| \Prob(\mathbf{\Theta}\mid \y)\big),
\label{OptProb}
\end{equation}
where 
\begin{equation}
\KL\big(q(\mathbf{\Theta}) \| \Prob(\mathbf{\Theta}\mid\y)\big) = \int q(\mathbf{\Theta}) \ln \frac{q(\mathbf{\Theta})}{\Prob(\mathbf{\Theta}\mid \y)} \dd \mathbf{\Theta}.
\end{equation}
This minimization becomes tractable 
if 
 a suitable factorization
 structure 
 of $q(\mathbf{\Theta})$ is assumed. 
 In particular, we assume that \f{$q(\mathbf{\Theta}) = \prod_{r=1}^R q_r(\Theta_r)$}. Hence, the optimal density approximation \f{$q^{\rm opt}_r(\Theta_r)$} for each variable $\Theta_r$, is obtained by minimizing the $ \KL$ divergence while holding the remaining densities for the rest of variables fixed. In this case, there exists an optimal solution to the optimization problem \eqref{OptProb} for each density \f{$(q_r)_{1\leq r \leq R}$}, given by the exponential of the expectation of the joint density with respect to the distribution of
 all the unknown parameters except the one of interest 
 i.e. \f{(see \cite{vsmidl2006variational, choudrey2002variational} for details of calculus)}
\begin{align}
(\forall r \! \in \! \{1,\ldots,R\})\;\; \f{q^{\rm opt}_r}(\Theta_r) \!\propto \! 
\exp \left(\langle \ln \Prob(\y,\Theta)\rangle_{\prod_{i\neq r}
\f{q^{\rm opt}_i}(\Theta_{i})}\right) 
\label{SoluVBA}
\end{align}
where $\langle \,\cdot\, \rangle_{\prod_{i\neq r}\f{q_i}(\Theta_{i})}= \int\cdot \prod_{i\neq r}\f{q_i}(\Theta_{i}) \dd\Theta_{i}$.
Due to the implicit relations existing between $\left(\f{q^{\rm opt}_r}(\Theta_r)\right)_{1\leq r \leq R}$, an analytical expression of $q^{\rm opt}(\Theta)$
generally 
does not exist. Usually, these distributions are determined in an iterative way, by updating one of the separable components $\left(\f{q_r}(\Theta_r)\right)_{1\leq r \leq R}$ while fixing the others \cite{vsmidl2006variational}. Applications of classical VBA approaches can be found in \cite{dremeau2012boltzmann,babacan2011variational,chen2014variational,marnissi2016fast,dremeau_bis} while improved VBA algorithms have been proposed in \cite{fraysse2014measure,zhengTIP2014}. Once the approximate distributions are computed, the unknown parameters are then estimated by the means of the obtained distributions.
\section{Proposed approach}
\label{sec:Algo}
\begin{center}
\begin{table*}[htb]
\resizebox{\textwidth}{!}{
\renewcommand{\arraystretch}{2}
\begin{tabular}{| l || l | l | l | l | l |}
\hline 
Name & $\phi_i(v ; y_i)$ & $\phi_i'(v ;y_i)$& $\beta_i(y_i)$& Domain of validity &Noise model\\
\hline
\hline
Gaussian& $\dfrac{1}{2 \sigma^2}\left( v-y_i\right)^2$  &$ \dfrac{1}{ \sigma^2} \left(v-y_i\right)$& $ \dfrac{1}{ \sigma^2} $& $y_i \in \mathbb{R}$, \f{$\sigma>0$} &Gaussian\\
\hline
 Cauchy $\phantom{\dfrac{\sqrt{\frac{3}{8}}}{\sqrt{\frac{3}{8}}}}$ & $ \ln \left( 1+\dfrac{(v-y_i)^2}{\sigma^2}\right)$  &$\dfrac{2(v-y_i)}{ \sigma^2+(v-y_i)^2 }$& $ \dfrac{2}{\sigma^2} $&$y_i \in \mathbb{R}$, \f{$\sigma>0$}&  Cauchy\\
\hline
\shortstack{Anscombe  transform } & $2 \left( \sqrt{y_i+\frac{3}{8}}-\sqrt{v+{\frac{3}{8}}}\right)^2$&$2-\dfrac{2 \sqrt{y_i+\frac{3}{8}}}{\sqrt{v+{\frac{3}{8}}}}$& $\left(\dfrac{3}{8}\right)^{-3/2}\sqrt{y_i+\frac{3}{8}} $&$y_i \geqslant -\dfrac{3}{8}$ & Poisson\\
\hline
\shortstack{Generalized Anscombe \\ transform } & $2 \left( \sqrt{y_i+\sigma^2+\frac{3}{8}}-\sqrt{v+\sigma^2+\frac{3}{8}}\right)^2$& $2-\dfrac{2\ \sqrt{y_i+\frac{3}{8}+\sigma^2}}{\sqrt{v+\frac{3}{8}+\sigma^2}}$& $\left(\frac{3}{8}+\sigma^2\right)^{-3/2}\sqrt{y_i+\frac{3}{8}+\sigma^2} $&$y_i \geqslant -\dfrac{3}{8}-\sigma^2$ & Poisson-Gaussian\\
\hline
\shortstack{Shifted Poisson }&$(v+\sigma^2)-(y_i+\sigma^2)\ln (v+\sigma^2)$&  $1-\dfrac{y_i+\sigma^2}{v+\sigma^2}$& $\dfrac{y_i+\sigma^2}{\sigma^4}$&$y_i \geqslant -\sigma^2$, \f{$\sigma>0$} &Poisson-Gaussian\\
%
\hline
\shortstack{Weighted least squares }&$\dfrac{(y_i-v)^2}{2(\sigma^2+v)}+\dfrac{1}{2}\ln (\sigma^2+v)$& $\dfrac{1}{2}-\dfrac{(y_i+\sigma^2)^2}{2(v+\sigma^2)^2}+\dfrac{1}{2(\sigma^2+v)}$&$\max  \left\lbrace  \dfrac{(y_i+\sigma^2)^2}{\sigma^6} -\dfrac{1}{2 \sigma^4}, \dfrac{1}{54 (y_i+\sigma^2)^4}\right\rbrace  $&$y_i \in \mathbb{R} \backslash\lbrace{-\sigma^2}\rbrace $, \f{$\sigma>0$} &Poisson-Gaussian\\
\hline 
\end{tabular}
}
\caption{\small Examples of differentiable functions satisfying Assumption~\ref{assum1}. \f{ The Anscombe transform provides a differentiable
approximation of the exact Poisson data fidelity term, while the three last functions can be employed to approximate the exact mixed Poisson-Gaussian log-likelihood. Note that alternative expressions for the Anscombe-based approaches can be found in \cite{makitalo2011optimal,makitalo2013optimal}.} $\phi_i'$ denotes the first derivative of function $\phi_i$ and $\beta_i(y_i)$ is the Lipschitz constant of $\phi_i'$ (for functions in lines 3-6, we assume that $\phi_i$ is replaced on $\mathbb{R}_-$ by its quadratic extension \eqref{eq:quadext}.) The expression for the Lipschitz constant of the gradient of the weighted least squares likelihood was established in \cite[Chap. IV]{repetti2015algorithmes}.} \label{t:1}
\end{table*}
\end{center} 
In this work, we assume the following separable form for~$q$:
\begin{equation}
q(\mathbf{\Theta})=\f{q_{\X}(\x)q_{\G}(\gamma)}. 
\end{equation}
Unfortunately, by using directly (\ref{SoluVBA}), we cannot obtain an explicit expression of $\f{q_{\X}(\x)}$ due to the intricate form of both the prior distribution and the likelihood when the statistics of the noise are no longer Gaussian. In this paper, we propose to use deterministic methods to construct quadratic upper bounds for the negative logarithms of both the likelihood and the prior density \cite{seeger2012fast}. 
This allows us to derive an upper bound of the desired cost function in \eqref{OptProb} as will be described in the following.
\subsection{Construction of the majorizing approximation}
\subsubsection{Likelihood}
One popular approach in signal recovery is the half-quadratic formulation \cite{geman1995nonlinear}. Under some mild assumptions and by introducing some auxiliary variables, a complicated criterion can be written as the infimum of a surrogate half-quadratic function i.e. the latter is quadratic with respect to the original variables and the auxiliary variables appear decoupled. This half-quadratic criterion can be then efficiently minimized using classical optimization algorithms. 
This formulation has been widely used in energy-minimization approaches \cite{Idier01,Nikolova05,champagnat2004connection} where the initial optimization problem is replaced by the minimization of the constructed surrogate function. Furthermore, this technique has been recently extended to sampling algorithms \cite{marnissiSSP16}. The initial intractable posterior distribution to sample from is replaced by the conditional distribution of the target signal given the auxiliary variables. The obtained distribution has been shown to be much simpler to explore by using standard sampling algorithms. In this paper, we propose to use half-quadratic approaches to construct an upper bound for the objective function in \eqref{OptProb}.

We assume that the likelihood satisfies the following property:
\begin{myassum}
For every $i \in \{1, \ldots, M\}$, $\phi_i$ is differentiable on $\mathbb{R}$ and there exists  $\mu_i(y_i)>0$ such that the function defined by $v \mapsto \frac{v^2}{2}-\frac{\phi_i(v;y_i)}{\mu_i(y_i)}$ is convex on $\mathbb{R}$.
\label{assum1}
\end{myassum}  
In particular, this assumption is satisfied when,  \f{for every $i \in \{1, \ldots, M\}$, $\phi_i$ is $\beta_i(y_i)$-Lipschitz differentiable on $\mathbb{R}$, i.e., 
\begin{equation}
\left( \forall u \in \mathbb{R} \right)\left( \forall v \in \mathbb{R} \right) \quad  \vert \phi_i^{'}( v; y_i) -  \phi_i^{'}( u; y_i)\vert \leqslant \beta_i(y_i) \vert v-u\vert, 
\end{equation}
as soon as $\mu_i(y_i) \geqslant\beta_i(y_i)$.
}

Table~\ref{t:1} shows some examples of useful functions satisfying the desired property (up to an additive constant). Note that, since the functions in lines 3-6 of Table \ref{t:1} are $\beta_i(y_i)$-Lipschitz differentiable only on $\mathbb{R}_+$, we propose to use on $\mathbb{R}_-$ a quadratic extension of them defined as follows:
\f{
\begin{equation} 
(\forall v \in \mathbb{R}_-) \quad  \phi_i(v;y_i)=\phi_i(0;y_i)+ \phi_i^{'}(0;y_i) v + \frac{1}{2} \beta_i(y_i) v^2, 
\label{eq:quadext}
\end{equation}
}
so that the extended version of $\phi_i(.;y_i)$ is now differentiable on $\mathbb{R}$ with $\beta_i(y_i)$-Lipschitzian gradient.

For every $i \in \{1, \ldots, M\}$ and $v \in \mathbb{R}$, let us define the following function:
\begin{equation}
 \varsigma_i(v;y_i)= \underset{t \in \mathbb{R}}{\textnormal{sup}} \left( -\frac{1}{2} \left( v-t\right)^2 + \dfrac{\phi_i(t;y_i)}{\mu_i(y_i)}\right). \label{eq:varsigma}
\end{equation}
Then, the following property holds:
\begin{myprop}
For every $i \in \{1, \ldots, M\}$,
\begin{equation}
\left(\forall v \in \mathbb{R} \right) \quad \phi_i(v;y_i) = \underset{ {w}_i \in \mathbb{R}}{\inf} \quad {T}_i( v,{w}_i; y_i). \label{eq:GY}
\end{equation}
where, for every $v\in \mathbb{R}$,
\begin{align}
{T}_i(v,w_i; y_i)= \mu_i(y_i) \left(\frac{1}{2}(v - {w}_i )^2 + \varsigma_i({w}_i;y_i)\right).
\label{e:defTi}
\end{align}
Moreover, the unique minimizer of $w_i \mapsto {T}_i( v,{w}_i; y_i)$ reads  
\begin{equation}
\widehat{w}_i(v) = v - \frac{1}{\mu_i(y_i)} \phi_i'(v ; y_i).
\end{equation}
\label{prop_GY}
\end{myprop}
\begin{proof}
See Appendix \ref{appA}.
\end{proof}

It follows from this result that
\begin{equation}
\left(\forall \mathbf{x} \in \mathbb{R}^N \right) \quad \phi(\mathbf{x};\y) = \underset{ {\mathbf{w}} \in \mathbb{R}^M}{\inf} {T}( \mathbf{x}, {\mathbf{w}};\mathbf{y}), 
\end{equation}
where ${T}(\mathbf{x},{\mathbf{w}};\mathbf{y})=\sum\limits_{i=1}^M {T}_i(\left[\mathbf{H}\mathbf{x}\right]_i, {w}_i; y_i)$.

Note that \eqref{eq:GY} shows that, for every $i\in \{1,\ldots,M\}$,  $\phi_i(\cdot;y_i)$ is a so-called Moreau envelope of the function
$\mu_i(y_i)\varsigma_i(\cdot;y_i)$. A more direct proof of Proposition \ref{prop_GY} can thus be derived from the properties of the proximity operator 
\cite{article_de_synthese_plc_jcp} when the functions $(\phi_i)_{1\le i \le M}$ are convex. The proof we provide in the appendix however does not make such a restrictive assumption.

\subsubsection{Prior}
 Similarly, we construct a surrogate function for the prior distribution. More precisely, we follow the same idea as in \cite{chen2014variational} and we use the following convexity inequality to derive a majorant for the $\ell_{\kappa}$-norm \f{with $\kappa \in (0,1]$}:
$$\f{( \forall \nu > 0)(\forall \upsilon \geqslant 0)} \quad 
\upsilon^{\kappa} \le (1-\kappa)\nu^{\kappa} + \kappa \nu^{\kappa-1} \upsilon.
$$
Hence, we obtain the following majorant function for the negative logarithm of the prior distribution: 
\begin{align}
\gamma \sum\limits_{j=1}^J 
 \Vert \mathbf{D}_j\mathbf{x} \Vert^{2\kappa} &\leqslant \gamma \sum\limits_{j=1}^J
\frac{\kappa  \Vert \mathbf{D}_j \mathbf{x} \Vert^2 + (1-\kappa)\lambda_j}{\lambda_j^{1-\kappa}}.
%
\end{align}
where $(\lambda_j)_{1\leq j \leq J}$ are positive variables. In the following, we will denote by $Q(\mathbf{x},\boldsymbol{\lambda};\gamma)\!=\!\sum\limits_{j=1}^J Q_j(\mathbf{D}_j \mathbf{x}, \lambda_j;\gamma)$, the function in the right-hand side of the above inequality where, for every $j \!\in\! \left\lbrace 1,\ldots, J\right\rbrace$,
\begin{equation}
Q_j(\mathbf{D}_j\mathbf{x}, \lambda_j;\gamma)= \gamma\frac{\kappa  \Vert \mathbf{D}_j \mathbf{x} \Vert^2 + (1-\kappa)\lambda_j}{\lambda_j^{1-\kappa}}.
\label{majPrior}
\end{equation} 
\subsubsection{Proposed majorant}
Thus, we can derive the following lower bound for the posterior distribution:
\begin{equation}
\Prob(\mathbf{\Theta}\mid\mathbf{y}) \geqslant  L(\mathbf{\Theta}| \mathbf{y}; \w, \boldsymbol{\lambda}),
\label{minorantposte}
\end{equation}
where function $L$ is defined as
$$ L(\mathbf{\Theta}| \mathbf{y}; \w, \boldsymbol{\lambda})= C(\mathbf{y})  \exp \left[-T( \mathbf{x},\mathbf{w};\mathbf{y})-Q(\mathbf{x},\boldsymbol{\lambda};\gamma)\right] \Prob(\gamma)$$ with $C(\mathbf{y}) = \Prob(\mathbf{y})^{-1} (2\pi)^{-M/2}\tau \gamma^{\frac{N}{2\kappa} }  $. 
The minorization of the distribution leads to an upper bound for the $\KL$ divergence:
\begin{equation}\label{e:minKL}
\KL(q(\mathbf{\Theta}) \| \Prob(\mathbf{\Theta}\mid \mathbf{y})) \leq \KL (q(\mathbf{\Theta}) \| L(\mathbf{\Theta}|\mathbf{y}; \w, \boldsymbol{\lambda})).
\end{equation} 
Note that, although the constructed lower bound in \eqref{minorantposte} is tangent to the posterior distribution i.e. $$ \Prob(\mathbf{\Theta}\mid\mathbf{y}) =\underset{\w \in \mathbb{R}^M, \boldsymbol{\lambda} \in \mathbb{R}^J}{\textnormal{sup}} L(\mathbf{\Theta}| \mathbf{y}; \w, \boldsymbol{\lambda}),$$ the tangency property  may not be generally satisfied in \eqref{e:minKL}. Thus, the tightness of the constructed majorant of the $\KL$ divergence may have a significant impact on the accuracy of the method. By minimizing the constructed bound  \eqref{e:minKL} with respect to $\w$ and $\boldsymbol{\lambda}$, we make this bound as tight as possible. Note that,  for every $ i \in \{ 1, \ldots, M \}$ and $ j \in \{ 1, \ldots, J \}$, $\lambda_j \mapsto  \KL (q(\mathbf{\Theta}) \| L(\mathbf{\Theta}|\mathbf{y}; \w, \boldsymbol{\lambda}))$ and $w_i \mapsto \KL (q(\mathbf{\Theta}) \| L(\mathbf{\Theta}|\mathbf{y}; \w, \boldsymbol{\lambda}))$ can be minimized separately. Hence, Problem \eqref{OptProb} can be solved by the following four-step alternating optimization scheme:
\begin{itemize}
\item 
Minimizing 
the upper bound in \eqref{e:minKL}
w.r.t. $\f{q_{\X}(\x)}$;
\item Updating the auxiliary variables $w_i$ in order to minimize $\KL(q(\mathbf{\Theta})  \| L(\mathbf{\Theta}|\mathbf{y}; \w, \boldsymbol{\lambda}))$, for every $i \in \{ 1, \ldots, M \}$;
\item Updating the auxiliary variable $\lambda_j$ in order to minimize $\KL(q(\mathbf{\Theta}) \| L(\mathbf{\Theta}|\mathbf{y}; \w, \boldsymbol{\lambda}))$, for every $j \in \{ 1, \ldots, J \}$;
\item Mimimizing the upper bound in \eqref{e:minKL} 
w.r.t. $\f{q_{\G}(\gamma)}$.
\end{itemize}
The main benefit of this majorization strategy is to guarantee that the optimal approximate posterior distribution for $\mathbf{x}$ belongs to the Gaussian family and the optimal approximate posterior distribution for $\gamma$ belongs to the Gamma one, i.e. 
\begin{align*}
\f{q_{\X}(\mathbf{x})}&\equiv\mathcal{N}(\mathbf{m}, \mathbf{\Sigma}),\quad
\f{q_{\G}(\mathbf{\gamma})}  \equiv \mathcal{G}(a,b).
\end{align*}
Therefore, the distribution updates can be performed by updating their parameters, namely $\mathbf{m}$, $\mathbf{\Sigma}$, $a$, and $b$.

\section{Iterative algorithm}
Subsequently, at a given iteration $k$ of the proposed algorithm, the corresponding estimated variables will be indexed by ~$k$.
\subsection{Updating \f{$q_{\X}(\x)$}}
Because of the majorization step, we need to minimize the upper bound on the $\KL$ divergence. The standard solution (\ref{SoluVBA}) can still be used by replacing the joint distribution by a lower bound $ L(\mathbf{\Theta},\mathbf{y}; {\mathbf{w}}, \boldsymbol{\lambda})$
chosen proportional to $L(\mathbf{\Theta}|\mathbf{y}; {\mathbf{w}}, \boldsymbol{\lambda})$: 
\FINAL{
\begin{align}
\f{q_{\X}^{k+1}(\x)} \propto &  \exp \left( \left\langle\ln  L(\x,\gamma,\mathbf{y}; {\mathbf{w}}^k, \boldsymbol{\lambda}^k)\right\rangle_{\f{q_{\G}^k}(\gamma)}\right) \nonumber\\
\propto  & \exp \left( \int \ln  L(\x,\gamma,\mathbf{y}; {\mathbf{w}}^k, \boldsymbol{\lambda}^k)\f{q_{\G}^{k}}(\gamma)\dd \gamma \right) \nonumber\\
	     \propto  & \exp \Bigg( -\sum_{i=1}^M \frac{1}{2} \mu_i(y_i) \left( \left[\mathbf{H}\mathbf{x}\right]_i-{w}_i^k \right)^2 \nonumber \\
	      & \qquad\quad- \dfrac{a_k}{b_k} \sum_{j=1}^J \frac{\kappa  \Vert \mathbf{D}_j \mathbf{x} \Vert^2 + (1-\kappa)\lambda_j^k}{(\lambda_j^k)^{1-\kappa}}
	       \Bigg). 
\end{align}
}
\DRAFT{
\begin{align}
\f{q_{\X}^{k+1}(\x)} \propto &  \exp \left( \left\langle\ln  L(\x,\gamma,\mathbf{y}; {\mathbf{w}}^k, \boldsymbol{\lambda}^k)\right\rangle_{\f{q_{\G}^k}(\gamma)}\right) \nonumber\\
\propto  & \exp \left( \int \ln  L(\x,\gamma,\mathbf{y}; {\mathbf{w}}^k, \boldsymbol{\lambda}^k)\f{q_{\G}^{k}}(\gamma)\dd \gamma \right) \nonumber\\
	     \propto  & \exp \Bigg( -\sum_{i=1}^M \frac{1}{2} \mu_i(y_i) \left( \left[\mathbf{H}\mathbf{x}\right]_i-{w}_i^k \right)^2 
 - \dfrac{a_k}{b_k} \sum_{j=1}^J \frac{\kappa  \Vert \mathbf{D}_j \mathbf{x} \Vert^2 + (1-\kappa)\lambda_j^k}{(\lambda_j^k)^{1-\kappa}}
	       \Bigg). 
\end{align}
}


The above distribution  can be identified as a multivariate Gaussian distribution whose covariance matrix and mean parameter are given by
\begin{align}
\bSigma_{k+1}^{-1} &= \H^\top \Diag(\boldsymbol{\mu}(\y)) \H + 2\dfrac{a_k}{b_k}  \D^\top \bLambda^k \D,\label{CovX}\\
\m_{k+1} & = \bSigma_{k+1} \H^\top \u,  \label{MeanX}
\end{align} 
where $\boldsymbol{\mu}(\y)=\left[\mu_1(y_1), \ldots, \mu_M(y_M)\right]^\top$, $\u$ is a $M \times 1$ vector whose $i$-th component is given by $u_i = \mu_i(y_i) {w}_i^k $ and $\bLambda$ is the diagonal matrix whose diagonal elements are $\left( \kappa (\lambda_j^k)^{\kappa-1}\mathbf{I}_{S}\right)_{1 \le j \le J}$. 


\subsection{Updating ${\mathbf{w}}$}
The auxiliary variable ${\mathbf{w}}$ is determined by minimizing the upper bound of $\KL$ divergence 
with respect to this variable: 
\begin{align}
{\mathbf{w}}^{k+1} = &  \,\argmin{{\mathbf{w}}} \int \f{q_{\X}^{k+1}}(\mathbf{\x})\f{q^{k}_{\G}}(\mathbf{\gamma}) \ln \frac{\f{q_{\X}^{k+1}}(\x)\f{q^{k}_{\G}}(\gamma)}{{L}(\mathbf{\Theta}|\mathbf{y}; \w, \boldsymbol{\lambda}^k)} \dd \x \dd \gamma \nonumber \\
	      = &  \,\argmin{{\mathbf{w}}} \int \f{q_{\X}^{k+1}}(\mathbf{\x})\f{q^{k}_{\G}}(\mathbf{\gamma})\left( -\ln {L}(\mathbf{\Theta}|\mathbf{y}; {\mathbf{w}}, \boldsymbol{\lambda}^k) \right) \dd \mathbf{\x} \dd \gamma\nonumber \\
	      = &  \,\argmin{{\mathbf{w}}} \int \f{q_{\X}^{k+1}}(\x) \sum_{i=1}^M {T}_i(\left[\mathbf{H}\mathbf{x}\right]_i, {w}_i;y_i) \dd \x \label{eq:22} \\
	      = &  \,\argmin{{\mathbf{w}}}  \sum_{i=1}^M {T}_i( \left[\mathbf{H}\m_{k+1}\right]_i, {w}_i;y_i),
\end{align}
where the equality in \eqref{eq:22} follows from the expression in \eqref{e:defTi}. 
Interestingly, it follows from Property \ref{prop_GY} that
\begin{align}
{w}_i^{k+1} &= \argmin{w_i}  {T}_i( \left[\mathbf{H}\m_{k+1}\right]_i, w_i;y_i) \nonumber \\
&=\left[\mathbf{H}\m_{k+1}\right]_i - \frac{1}{\mu_i(y_i)} \phi_i^{'}(\left[\mathbf{H}\m_{k+1}\right]_i;y_i).
\label{eqw}
\end{align}

\subsection{Updating $\boldsymbol{\lambda}$}
The variable $\boldsymbol{\lambda}$ is determined in a similar way: for every $j\in \{1,\ldots,J\}$,
\FINAL{
\begin{align}
\lambda_j^{k+1} &=  \argmin{\lambda_j \in [0,+\infty)} \KL(\f{q^{k+1}_{\X}}\big(\mathbf{\x})\f{q^{k}_{\G}}(\mathbf{\gamma})\| L(\mathbf{\Theta}|\mathbf{y}; \w^{k+1}, \boldsymbol{\lambda})\big) \nonumber \\
&= \argmin{\lambda_j \in [0,+\infty)} \sum_{i=1}^Q \int \f{q^{k+1}_{\X}}(\mathbf{\x})\f{q^{k}_{\G}}(\mathbf{\gamma})  Q_i(\mathbf{D}_i \mathbf{x}, \lambda_i; \gamma) \dd \x \dd \gamma \nonumber \\
&= \argmin{\lambda_j \in [0,+\infty)} \int \f{q^{k+1}_{\X}}(\mathbf{\x})\f{q^{k}_{\G}}(\mathbf{\gamma}) Q_j(\mathbf{D}_j \mathbf{x}, \lambda_j; \gamma) \dd \x \dd \gamma  \nonumber \\
&=\argmin{\lambda_j \in [0,+\infty)} \int \f{q^{k+1}_{\X}}(\mathbf{\x})\f{q^{k}_{\G}}(\mathbf{\gamma}) \nonumber \\ & \hspace{2cm }\times \gamma \ \frac{\kappa  \Vert \mathbf{D}_j \mathbf{x} \Vert^2 + (1-\kappa)\lambda_j}{\lambda_j^{1-\kappa}} \dd \x \dd \gamma \nonumber  \\
&=\argmin{\lambda_j \in [0,+\infty)} \frac{\kappa\, \EE_{\f{q^{k+1}_{\X}}(\x)}\left[\Vert\mathbf{D}_j\x \Vert^2 \right] + (1-\kappa)\lambda_j}{\lambda_j^{1-\kappa}} .
\end{align}
}
\DRAFT{
\begin{align}
\lambda_j^{k+1} &=  \argmin{\lambda_j \in [0,+\infty)} \KL(\f{q^{k+1}_{\X}}\big(\mathbf{\x})\f{q^{k}_{\G}}(\mathbf{\gamma})\| L(\mathbf{\Theta}|\mathbf{y}; \w^{k+1}, \boldsymbol{\lambda})\big) \nonumber \\
&= \argmin{\lambda_j \in [0,+\infty)} \sum_{i=1}^Q \int \f{q^{k+1}_{\X}}(\mathbf{\x})\f{q^{k}_{\G}}(\mathbf{\gamma})  Q_i(\mathbf{D}_i \mathbf{x}, \lambda_i; \gamma) \dd \x \dd \gamma \nonumber \\
&= \argmin{\lambda_j \in [0,+\infty)} \int \f{q^{k+1}_{\X}}(\mathbf{\x})\f{q^{k}_{\G}}(\mathbf{\gamma}) Q_j(\mathbf{D}_j \mathbf{x}, \lambda_j; \gamma) \dd \x \dd \gamma  \nonumber \\
&=\argmin{\lambda_j \in [0,+\infty)} \int \f{q^{k+1}_{\X}}(\mathbf{\x})\f{q^{k}_{\G}}(\mathbf{\gamma}) 
\gamma \ \frac{\kappa  \Vert \mathbf{D}_j \mathbf{x} \Vert^2 + (1-\kappa)\lambda_j}{\lambda_j^{1-\kappa}} \dd \x \dd \gamma \nonumber  \\
&=\argmin{\lambda_j \in [0,+\infty)} \frac{\kappa\, \EE_{\f{q^{k+1}_{\X}}(\x)}\left[\Vert\mathbf{D}_j\x \Vert^2 \right] + (1-\kappa)\lambda_j}{\lambda_j^{1-\kappa}} .
\end{align}
}

The minimum is attained at 
\begin{align}
\lambda_j^{k+1}  = &\,\EE_{\f{q^{k+1}_{\X}}(\x)}\left[\Vert\mathbf{D}_j\x \Vert^2 \right]\nonumber\\
     = &\,\Vert \mathbf{D}_j\m_{k+1}\Vert^2+ \text{trace}\left[\mathbf{D}_j^\top\mathbf{D}_j  \bm{\Sigma}_{k+1}\right].
\label{eqlambdai}
\end{align}

\subsection{Updating \f{$q_{\G}(\gamma)$}}
Using (\ref{SoluVBA}) where the joint distribution is replaced by its lower bound function, we obtain
\FINAL{
\begin{align}
	  \f{q_{\G}(\gamma)}	   \propto &  \exp \left( \left\langle\ln  L(\mathbf{x},\gamma,\mathbf{y}; \mathbf{w}^{k+1}, \boldsymbol{\lambda}^{k+1})\right\rangle_{\f{q^{k+1}_{\X}}(\mathbf{x})}\right) \nonumber\\
\propto  & \exp \left( \int \ln  L(\mathbf{x},\gamma,\mathbf{y}; \mathbf{w}^{k+1}, \boldsymbol{\lambda}^{k+1})\f{q^{k+1}_{\X}}(\mathbf{x})d \mathbf{x} \right) \nonumber\\	    
    \propto  & \gamma^{\frac{N}{2\kappa} +\alpha-1} \exp (-\beta \gamma) \nonumber \\ \times \exp &\Bigg(  -  \gamma \sum_{j=1}^J 
	       \frac{\kappa \EE_{\f{q^{k+1}_{\X}}(\x)}\left[ \Vert \mathbf{D}_j \mathbf{x} \Vert^2\right] + (1-\kappa)\lambda_j^{k+1}}{(\lambda_j^{k+1})^{1-\kappa}} \Bigg)\nonumber  \\
\equiv&\mathcal{G}(a_{k+1},b_{k+1}).
\end{align}
}
\DRAFT{
\begin{align}
	  \f{q_{\G}(\gamma)}	   \propto &  \exp \left( \left\langle\ln  L(\mathbf{x},\gamma,\mathbf{y}; \mathbf{w}^{k+1}, \boldsymbol{\lambda}^{k+1})\right\rangle_{\f{q^{k+1}_{\X}}(\mathbf{x})}\right) \nonumber\\
\propto  & \exp \left( \int \ln  L(\mathbf{x},\gamma,\mathbf{y}; \mathbf{w}^{k+1}, \boldsymbol{\lambda}^{k+1})\f{q^{k+1}_{\X}}(\mathbf{x})d \mathbf{x} \right) \nonumber\\	    
    \propto  & \gamma^{\frac{N}{2\kappa} +\alpha-1} \exp (-\beta \gamma)
		\exp \Bigg(  -  \gamma \sum_{j=1}^J 
	       \frac{\kappa \EE_{\f{q^{k+1}_{\X}}(\x)}\left[ \Vert \mathbf{D}_j \mathbf{x} \Vert^2\right] + (1-\kappa)\lambda_j^{k+1}}{(\lambda_j^{k+1})^{1-\kappa}} \Bigg)\nonumber  \\
\equiv&\mathcal{G}(a_{k+1},b_{k+1}).
\end{align}
}

	       Using \eqref{eqlambdai}, 
one can recognize that the above distribution is a Gamma one with parameters
\begin{align}
a_{k+1}&= \frac{N}{2\kappa}+\alpha = a 
, \qquad
b_{k+1}=\sum_{j=1}^J (\lambda_j^{k+1})^\kappa + \beta \label{e:bk+1}.
\end{align}
\subsection{Resulting algorithm}
The proposed method is outlined in Algorithm \ref{Algo2}. It alternates between the update of the auxiliary variables and the distribution of the unknown parameters. 
\begin{algorithm}[h]
\caption{VBA approach for recovery of signals corrupted with non-Gaussian noise.}
\label{Algo1}
\begin{enumerate}
\compresslist
 \item Set initial values: $\w^0, \boldsymbol{\lambda}^0, b_0$. Compute $a$ with \eqref{e:bk+1}.
 \item For $k=0,1,\ldots$
 \begin{enumerate}
 \compresslist
  \item Update parameters $\bSigma_{k+1}$ and $\m_{k+1}$ of $\f{q^{k+1}_{\X}}(\x)$ using 
  \eqref{CovX} and \eqref{MeanX}.
  \item Update $\w^{k+1}$ using \eqref{eqw}.
  \item Update $\boldsymbol{\lambda}^{k+1}$ using \eqref{eqlambdai}.
  \item Update parameter $b_{k+1}$ of $\f{q^{k+1}_{\G}}(\gamma)$ using \eqref{e:bk+1}.
  \end{enumerate}
\end{enumerate}
\end{algorithm}
\subsection{Implementation issues}
\label{SecImplIss}
An additional difficulty arising in the implementation of Algorithm \ref{Algo2} is that the determination of $\bSigma_{k+1}$ requires inverting the matrix given by (\ref{CovX}), which is computationally expensive in high dimension. To bypass this operation, we propose to compare two approaches. 
The first one follows the idea in~\cite{babacan2011variational}: we make use of the linear conjugate gradient method to approximate $\m_{k+1}$ iteratively and in (\ref{eqlambdai}), where an explicit form of $\bSigma_{k+1}$ cannot be sidestepped, this matrix is approximated by a diagonal one whose diagonal entries are equal to the inverse of the diagonal elements of $\bSigma_{k+1}^{-1}$. 
The second technique uses Monte-Carlo sample averaging to approximate $\m_{k+1}$ and $\lambda_j^{k+1}$: specifically, we generate samples $(\mathbf{n}_s)_{1 \le s \le N_s}$ from Gaussian distribution with mean $\m_{k+1}$ and covariance matrix $\bSigma_{k+1}$ using \cite{high_Gauss_sample3},  as summarized in Algorithm \ref{Algo2}. This estimator has two desirable properties. First, its accuracy is independent of the problem size, its relative error only depends on the number of samples and it decays as $\sqrt{2 / N_s}$  (only $N_s = 2/\rho^2$ samples are required to reach a desired relative error $\rho$) \cite{PapandreouNIPS}. Second, 
for the simulation of $N_s$ independent Gaussian samples, one can take advantage of a multiprocessor architecture by resorting to parallel implementation allowing us to reduce the computation time. 
\begin{algorithm}[h]
\caption{Stochastic approach for computing the parameters of $q(\x)$.}
\label{Algo2}
\begin{enumerate}
\compresslist
 \item For $s=1,2,\ldots,N_s$
 \begin{enumerate}
 \compresslist
 \item \textbf{Perturbation} : Generate 
 \begin{align*}
   \boldsymbol{\nu}_s &\sim \mathcal{N}\left(\mathbf{u},  \left(\text{Diag}\big(\boldsymbol{\mu}(\y)\big)\right)^{1/2}\right)\\
   \boldsymbol{\eta}_s &\sim \mathcal{N}\left(0,\sqrt{2 \gamma^k} \mathbf{\Lambda}_k^{1/2}\right)
   \end{align*}
 with $\gamma^k = a_k/b_k$.
  \item \textbf{Optimization}: Compute $\mathbf{n}_s$ as the minimizer of 
  $\mathcal{J}(\mathbf{v})= \Vert  \boldsymbol{\nu}_s-  \text{Diag}\big(\boldsymbol{\mu}(\y)\big) \mathbf{H}\mathbf{v}\Vert^2_{\left(\text{Diag}(\boldsymbol{\mu}(\y))\right)^{-1}}+ \frac{1}{2 \gamma^k }\Vert \boldsymbol{\eta}_s - 2 \gamma^k \mathbf{\Lambda}_k \mathbf{D}\mathbf{v}\Vert^2_{ \mathbf{\Lambda}_k^{-1}}$, 
 which is equivalent to minimize $\tilde{\mathcal{J}}(\mathbf{v})=\mathbf{v}^\top \mathbf{\Sigma}_{k+1}^{-1} \mathbf{v}- 2 \mathbf{v}^\top \mathbf{z}_s   $
where $\mathbf{z}_s=\mathbf{H}^\top \boldsymbol{\nu}_s+   \mathbf{D}^\top \boldsymbol{\eta}_s.$
The minimizer is computed using the conjugate gradient algorithm.  
  \end{enumerate}
  \item  Update 
  \begin{align*}
  &\m_{k+1}= \frac{1}{N_s} \sum_{s=1}^{N_s}\mathbf{n}_s\\
 & \left( \forall j \in \lbrace 1, \ldots, J \rbrace \right) \quad \lambda_j^{k+1}= \frac{1}{N_s} \sum_{s=1}^{N_s} \Vert\mathbf{D}_j \mathbf{n}_s \Vert^2.
  \end{align*}
\end{enumerate}
\end{algorithm}

\section{Application to Poisson-Gaussian image restoration} 
\label{Sect:results}
\f{Let us now illustrate the usefulness of our algorithm via experiments in the context of image restoration when the noise follows a mixed PG model.
Recently, there has been a growing interest  for the PG noise model as it arises in many real imaging systems in astronomy \cite{clar:Benvenuto_2008, clar:Snyder_1993}, medicine \cite{nichols2002spatiotemporal}, photography \cite{Julliand2015}, and biology \cite{delpretti2008multiframe}. Numerous efficient restoration methods exist in the limit case when one neglects either the Poisson or the Gaussian component. However, such approximation may be rough, and lead to poor restoration results, especially in the context of low count imaging and/or high level electronic noise. On the opposite, restoration methods that specifically address mixed PG noise remain scarce, especially when the observation operator $\mathbf{H}$ differs from identity. The aim of this section is to show the applicability of the proposed VBA method in this context. 
}
\subsection{Problem formulation}
The vector of observations $\mathbf{y} \!=\! (y_i)_{1\le i \le M} \in \mathbb{R}^M$ is related to the original image $\overline{\mathbf{x}}$ through
\begin{align}
\mathbf{y}  = \mathbf{z}+\mathbf{b},
\end{align}
where $\mathbf{z}$ and $\mathbf{b}$
are assumed to be mutually independent random vectors and 
\begin{align*}
&\mathbf{z} \mid \overline{\mathbf{x}} \sim \mathcal{P}( \mathbf{H}\overline{\mathbf{x}}), \quad
\mathbf{b} \sim   \mathcal{N}(\mathbf{0}, \sigma^2 \mathbf{I}_M),
\end{align*}
$\mathcal{P}$ denoting \f{the independent Poisson distribution, and $\sigma>0$.  
The associated likelihood function reads} \cite{clar:Chouzenoux_2015}:
\begin{equation}\label{e:likelihood2}
\Prob(\mathbf{y}\mid\mathbf{x})= \prod\limits_{i=1}^M \left(\sum\limits_{n=1}^{+\infty}\dfrac{e^{-\left[\mathbf{H}\mathbf{x}\right]_i} \left( \left[\mathbf{H}\mathbf{x}\right]_i\right)^n}{n!}\dfrac{e^{-\frac{1}{2\sigma^2}(y_i-n)^2}}{\sqrt{2\pi \sigma^2}}\right).
\end{equation}
The expression of the PG likelihood \eqref{e:likelihood2} involves an infinite sum which makes its exact computation impossible in practice. In \cite{clar:Chouzenoux_2015}, the infinite sum was replaced by a finite summation with bounds depending on the current estimate of $\bar{\mathbf{x}}$. However, this strategy implies a higher computational burden in the reconstruction process when compared with other likelihoods proposed in the literature as accurate approximations of \eqref{e:likelihood2}.
In \cite{marnissi2016fast}, VBA inference techniques have been successfully applied to the restoration of data corrupted with PG noise using the generalized Anscombe transform (GAST) likelihood \cite{clar:Murtagh_1995,clar:Starck_1998,foi2013,Makitalo_M_2012_picassp_Poisson_Fdueuigat}. 
\f{Following these promising preliminary results, we will consider here the GAST approximation, as well as the shifted Poisson (SPoiss) \cite{Chakrabarti} and the weighted least squares (WL2) \cite{clar:Benvenuto_2008, clar:Li_2014, clar:Repetti_2012} approximations, defined respectively in lines 4, 5 and 6 of Table \ref{t:1}. In order to satisfy Assumption~\ref{assum1}, we will use $\mu_i(y_i) \equiv \max\left \lbrace\beta_i(y_i), \varepsilon\right\rbrace$ where $\varepsilon>0$ for the GAST and the SPoiss approximations. For the WL2 approximation, we set $\mu_i(y_i)=\max \left\lbrace (y_i+\sigma^2)^2/\sigma^6, \varepsilon \right\rbrace$. Note that in all our experiments, a data truncation is performed as a pre-processing step on the observed image $\mathbf{y}$ in order to satisfy the domain condition given in the fifth column of Table~\ref{t:1}.}

%
\subsection{Numerical results}
\begin{figure}[t]
\centering
\begin{tabular}{@{}c@{}c@{}}
\subfloat[][Image $\overline{\mathbf{x}}_1$ \\($256 \times 256$)]{\includegraphics[scale=0.45]{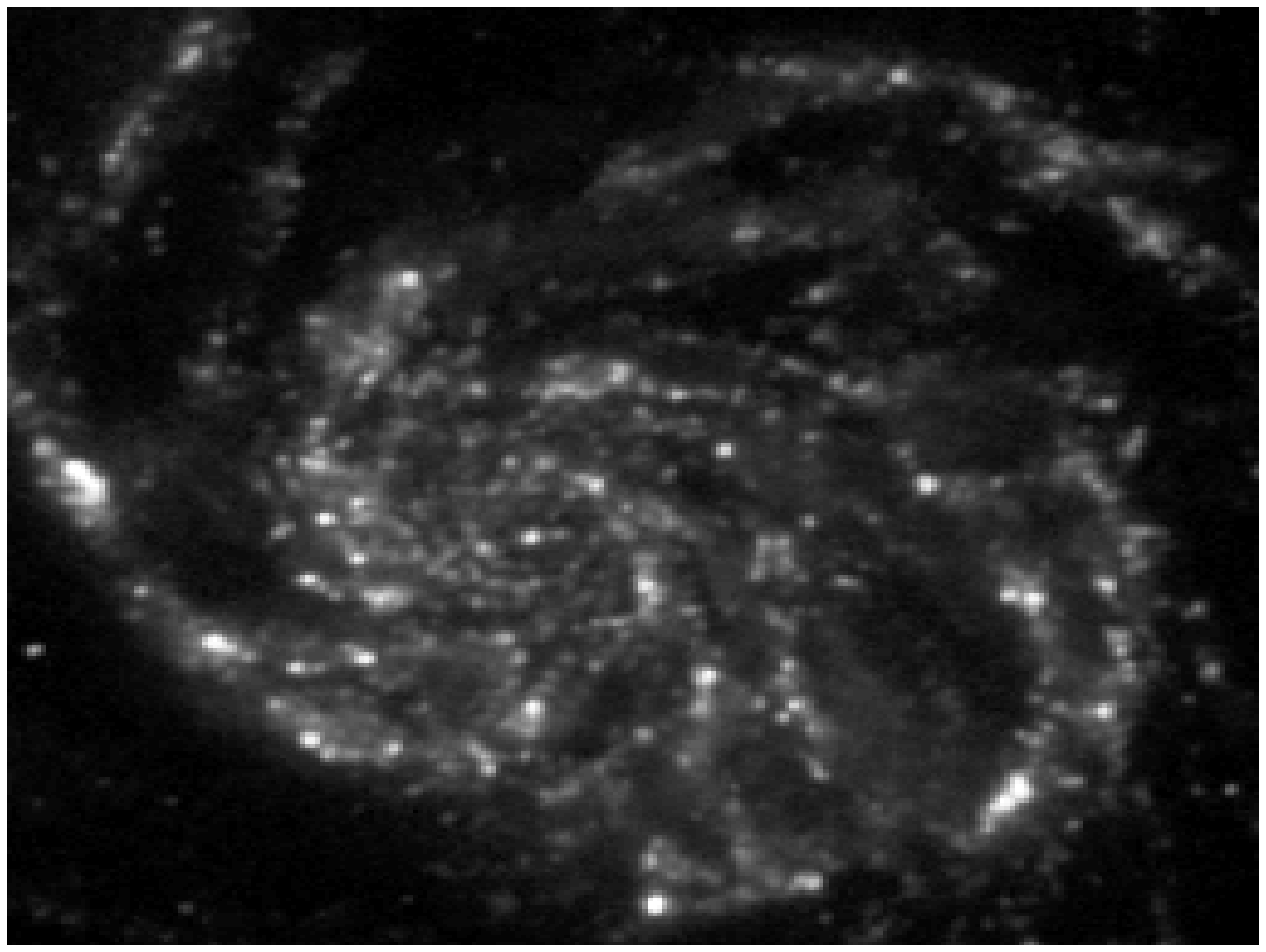} }&
\subfloat[][Image $\overline{\mathbf{x}}_2$ \\($190 \times 190$)]{\includegraphics[scale=0.45]{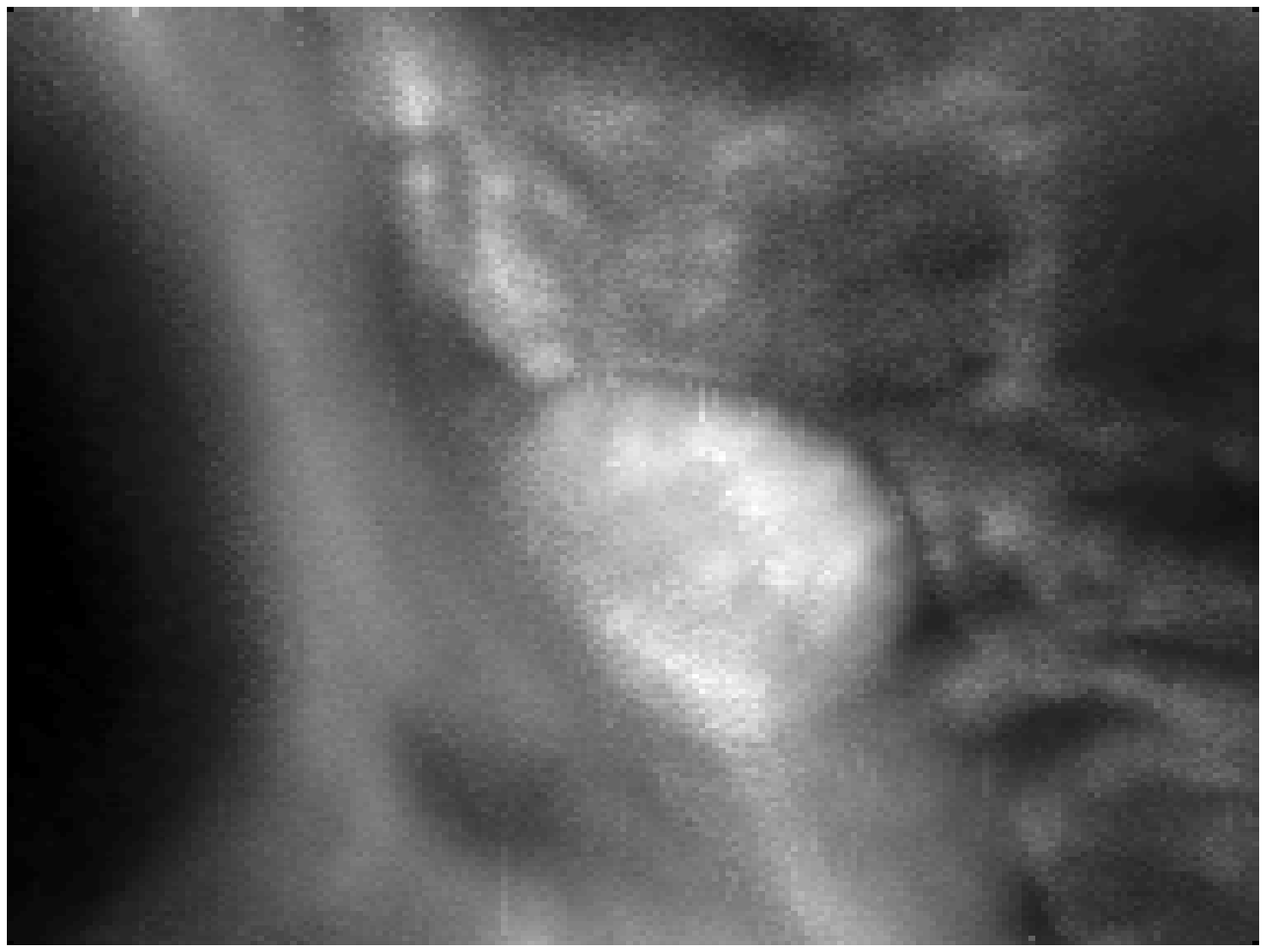} }\\
\subfloat[][Image $\overline{\mathbf{x}}_3$ \\($ 257\times 256 $)]{\includegraphics[scale=0.45]{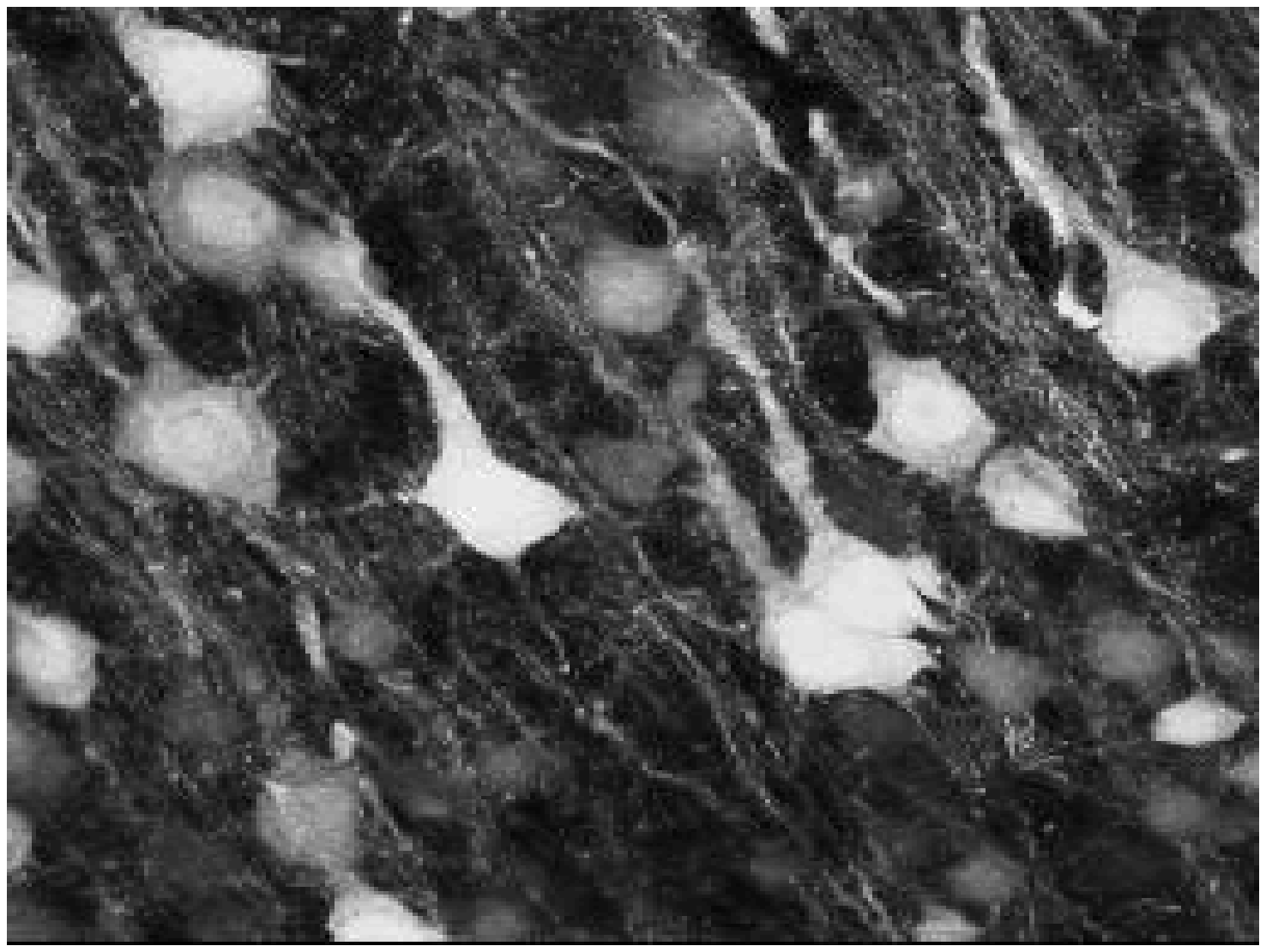} }
&
\subfloat[][Image $\overline{\mathbf{x}}_4$ \\($350 \times 350$)]{\includegraphics[scale=0.45]{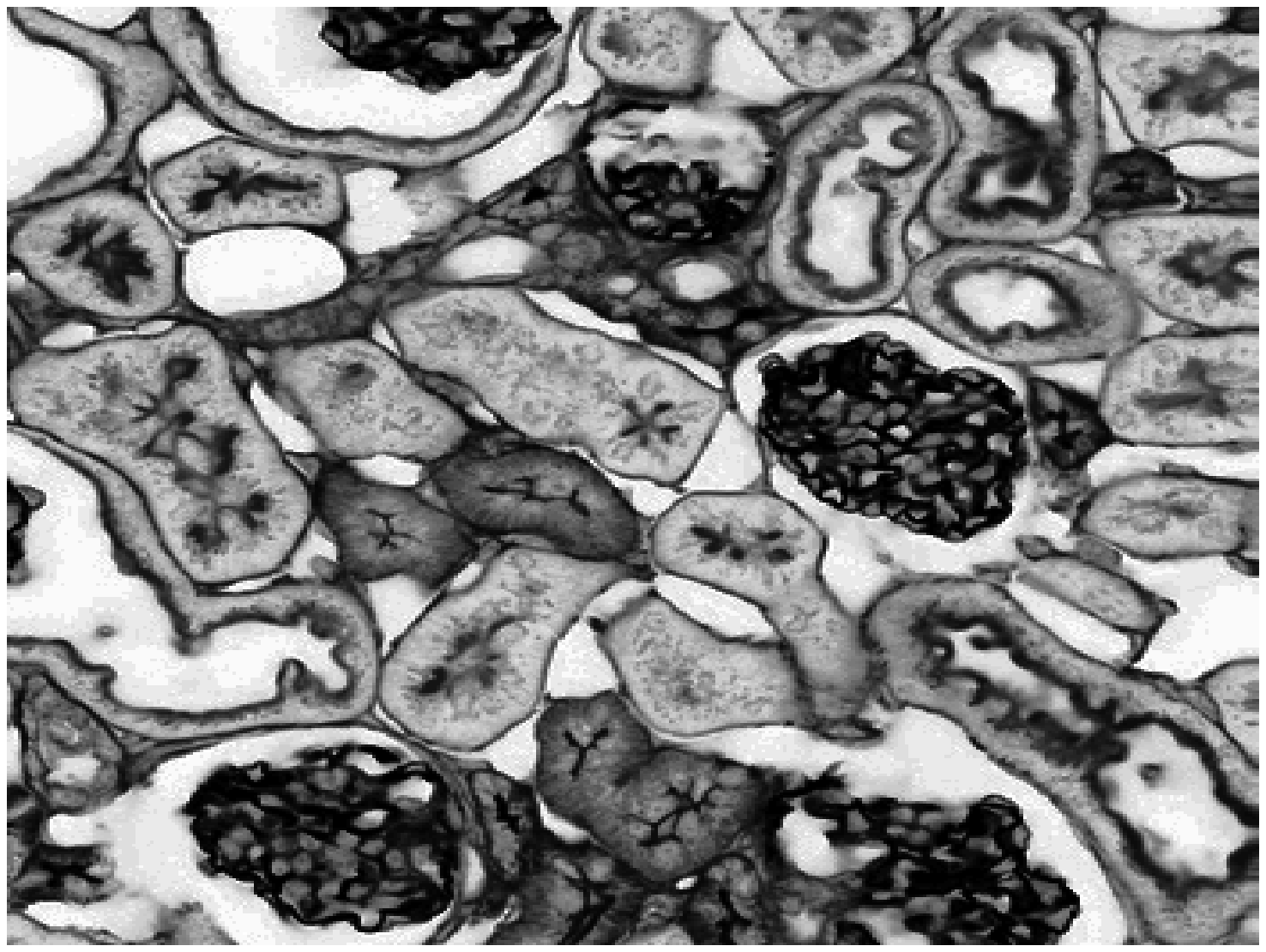} }\\
\subfloat[][Image $\overline{\mathbf{x}}_5$ \\($128 \times 128$)]{\includegraphics[scale=0.45]{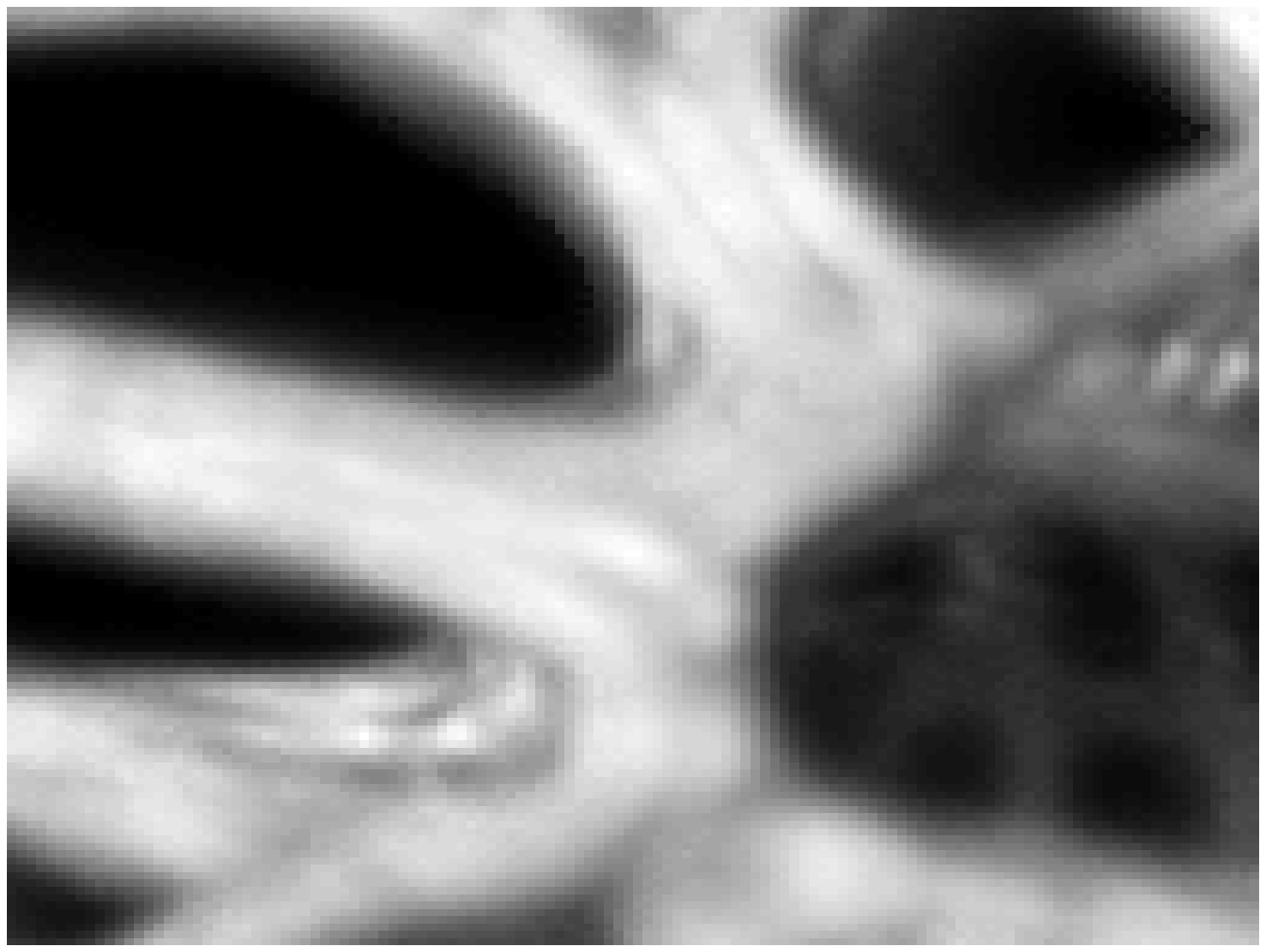} } &
\subfloat[][Image $\overline{\mathbf{x}}_6$ \\($256 \times 256$)]{\includegraphics[scale=0.45]{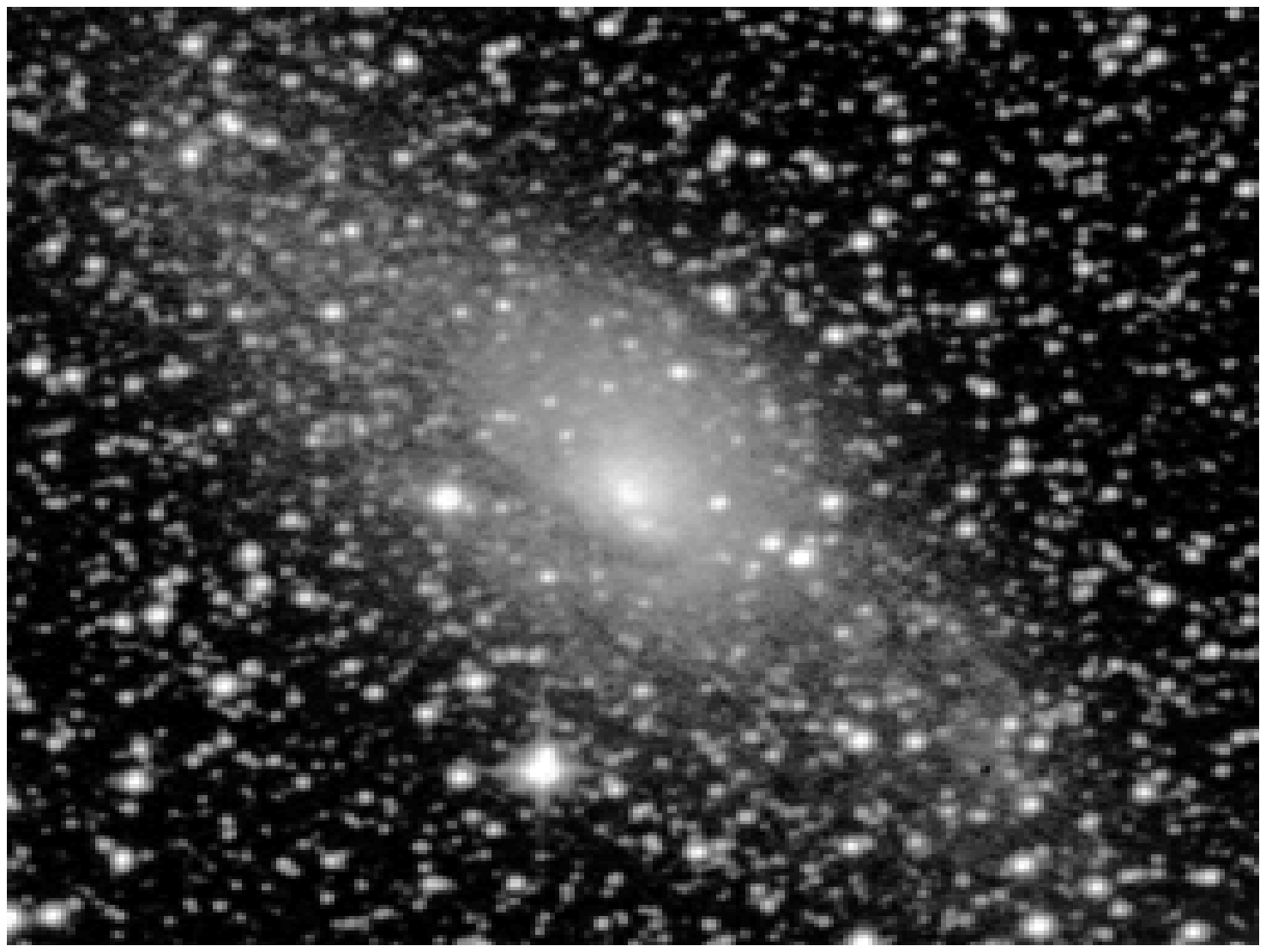} }
\end{tabular}
\caption{Original images.}
\label{fig:im}
\end{figure}
\begin{table}[t]
\tiny
\centering
\resizebox{0.9\textwidth}{!}{
\renewcommand{\arraystretch}{1.3}
\begin{tabular}{|l|l|l||l|l|l|}
\cline{4-6}
 \multicolumn{3}{c|}{ }&\shortstack{GAST }& SPoiss& WL2\\
\hhline{---|=|=|=|}
\multirow{9}{*}{\shortstack{ VBA }}&
\multirow{3}{*}{\shortstack{ Approx. 1 }}
&SNR & 8.13  & 9.36  & 9.90\\
\cline{3-6}
&&SSIM& 0.3987 & 0.4790 & 0.5140\\
\cline{3-6}
&&Time (s.)& 55 & 62 & 67 \\
\hhline{|~|-|----|}
& \multirow{3}{*}{\shortstack{ Approx. 2 \\
$N_s=160$ }}
&SNR & 9.57  & 10.17 &10.22  \\
\cline{3-6}
& &SSIM& 0.5260 & 0.6017&0.6058\\
\cline{3-6}
& &Time (s.)&688&601 &1011 \\
\hhline{|~|-|----|}
 & \multirow{3}{*}{\shortstack{ Approx. 2 \\
$N_s=640$ }}
&SNR & \textbf{9.61}&   \textbf{ 10.20} & \textbf{10.27}  \\
\cline{3-6}
&&SSIM& \textbf{0.5308} & \textbf{0.6088}& \textbf{0.6112}\\
\cline{3-6}
&&Time (s.)&\textbf{3606}&\textbf{3507}& \textbf{3510} \\
\hline
\hline
\multirow{6}{*}{\shortstack{ MAP }}&
\multirow{3}{*}{\shortstack{ Discrepancy \\ principle}}
&SNR & -1.13& 5.24 & 10.17 \\
\cline{3-6}
&&SSIM&0.0980 &   0.2961 & 0.6131 \\
\cline{3-6}
&&Time (s.)& 3326&  2215   & 3053 \\
\hhline{|~|-|----|}
&\multirow{3}{*}{\shortstack{ Best \\ parameter}}
&SNR & 9.46& 10.40 & 10.39 \\
\cline{3-6}
&&SSIM& 0.5078 & 0.6029   & 0.5920 \\
\cline{3-6}
&&Time (s.)& 4380&  2560   & 13740 \\
\hline 
\end{tabular}
}
\caption{ Restoration results for image $\overline{\mathbf{x}}_1$ with $x^+=10$ and $\sigma^2=4$. Uniform kernel with size $5 \times 5$. Initial SNR= -2.55~dB.}\label{t:2}
\centering
\resizebox{0.9\textwidth}{!}{
\renewcommand{\arraystretch}{1.3}
\begin{tabular}{|l|l|l||l|l|l|}
\cline{4-6}
 \multicolumn{3}{c|}{ }&\shortstack{GAST }& SPoiss& WL2\\
\hhline{---|=|=|=|}
\multirow{9}{*}{\shortstack{ VBA }}&
\multirow{3}{*}{\shortstack{ Approx. 1 }}
&SNR & 13.97 & 15.19  & 16.41  \\
\cline{3-6}
&&SSIM& 0.3544 & 0.4167 & 0.4959 \\
\cline{3-6}
&&Time (s.)& 58  & 64 & 70\\
\hhline{|~|-|----|}
&\multirow{3}{*}{\shortstack{ Approx. 2 \\
$N_s=160$ }}
&SNR& {18.05}  & {19.07} & {19.11}  \\
\cline{3-6}
&&SSIM& {0.6664}  & {0.6930}   & {0.7066} \\
\cline{3-6}
&&Time (s.)& {524}& {498 }  & {491}\\
\hhline{|~|-|----|}
&\multirow{3}{*}{\shortstack{ Approx. 2 \\
$N_s=640$ }}
&SNR& \textbf{18.11} & \textbf{19.12}   & \textbf{19.13} \\
\cline{3-6}
&&SSIM& \textbf{0.6778} & \textbf{0.7034}   & \textbf{0.7152} \\
\cline{3-6}
&&Time (s.)& \textbf{2048}& \textbf{1828 } & \textbf{1735}\\
\hline
\hline
\multirow{6}{*}{\shortstack{ MAP }}&
\multirow{3}{*}{\shortstack{ Discrepancy \\principle}}
&SNR & 16.52& 17.41 & 18.09 \\
\cline{3-6}
&&SSIM& 0.5484&  0.7570  &  0.6732  \\
\cline{3-6}
&&Time (s.)& 594&  583 & 2286\\
\hhline{|~|-|----|}
&\multirow{3}{*}{\shortstack{ Best \\parameter}}
&SNR & 17.83& 18.73 & 19.09 \\
\cline{3-6}
&&SSIM& 0.6519 & 0.6646   & 0.6702 \\
\cline{3-6}
&&Time (s.)& 674&  705   & 4164\\
\hline 
\end{tabular}
}
\caption{ Restoration results for  the image $\overline{\mathbf{x}}_2 $, $x^+=12$ and $\sigma^2=9$. Gaussian kernel with size $25 \times 25$, std 1.6. Initial SNR= 2.21 dB.}\label{t:3}
\end{table}
\begin{table}[t]
\tiny
\centering
\resizebox{0.9\textwidth}{!}{
\renewcommand{\arraystretch}{1.3}
\begin{tabular}{|l |l|l|| l|l|l|}
\cline{4-6}
 \multicolumn{3}{c|}{ }&\shortstack{GAST }& SPoiss& WL2\\
\hhline{---|=|=|=|}
\multirow{9}{*}{\shortstack{ VBA }}&
\multirow{3}{*}{\shortstack{ Approx. 1 }}
&SNR & 11.42 & 11.94    & 12.25  \\
\cline{3-6}
&&SSIM& 0.4184  & 0.4403   &0.4588\\
\cline{3-6}
&&Time (s.)& 45  & 47   & 53  \\
\hhline{|~|-|----|}
&\multirow{3}{*}{\shortstack{ Approx. 2 \\
$N_s=160$ }}
&SNR & 12.04    & 12.31  & 12.28 \\
\cline{3-6}
&&SSIM& 0.4555 & 0.4624  & 0.4627 \\
\cline{3-6}
&&Time (s.)& 328 &  332   &  396 \\
\hhline{|~|-|----|}
&\multirow{3}{*}{\shortstack{ Approx.2 \\
$N_s=640$ }}
&SNR & \textbf{12.09}  & \textbf{12.36}   & \textbf{12.33} \\
\cline{3-6}
&&SSIM& \textbf{0.4617} &\textbf{ 0.4684}  & \textbf{0.4683} \\
\cline{3-6}
&&Time (s.)& \textbf{1965}  & \textbf{ 2051}   &  \textbf{2019} \\
\hline
\hline
\multirow{6}{*}{\shortstack{ MAP }}&
\multirow{3}{*}{\shortstack{ Discrepancy \\principle}}
&SNR & 12.08& 12.38 & 12.08 \\
\cline{3-6}
&&SSIM&0.4523 &  0.4582  &  0.4314 \\
\cline{3-6}
&&Time (s.)& 6252&  3865 & 1929\\
\hhline{|~|-|----|}
&\multirow{3}{*}{\shortstack{ Best \\parameter}}
&SNR & 12.17& 12.45 & 12.37\\
\cline{3-6}
&&SSIM& 0.4531& 0.4576  & 0.4565 \\
\cline{3-6}
&&Time (s.)& 3348&  2441 & 2525\\
\hline 
\end{tabular}
}
\caption{ Restoration results for the image $\overline{\mathbf{x}}_3 $ with $x^+=15$ and $\sigma^2=9$. Uniform kernel with size $5 \times 5$. Initial SNR= 3.14~dB.}\label{t:4}
\centering
\resizebox{0.9\textwidth}{!}{
\renewcommand{\arraystretch}{1.3}
\begin{tabular}{|l|l |l|| l|l|l|}
\cline{4-6}
 \multicolumn{3}{c|}{ }&\shortstack{GAST }& SPoiss& WL2\\
\hhline{---|=|=|=|}
\multirow{9}{*}{\shortstack{ VBA }}&
\multirow{3}{*}{\shortstack{ Approx. 1 }}
&SNR & \textbf{13.80}  & \textbf{13.90}   &\textbf{ 13.66}  \\
\cline{3-6}
&&SSIM&  \textbf{0.5752}& \textbf{0.5769} &  \textbf{0.5582}\\
\cline{3-6}
&&Time (s.)& \textbf{29}  & \textbf{34}  & \textbf{88 } \\
\hhline{|~|-|----|} 
&\multirow{3}{*}{\shortstack{ Approx. 2 \\
$N_s=160$ }}
&SNR &13.72   & 13.76   & 13.56 \\
\cline{3-6}
&&SSIM& 0.5667 & 0.5641 & 0.5491\\
\cline{3-6}
&&Time (s.)& 555& 580  & 757\\
\hhline{|~|-|----|} 
&\multirow{3}{*}{\shortstack{ Approx. 2 \\
$N_s=640$ }}
&SNR &13.78  & 13.81   & 13.61 \\
\cline{3-6}
&&SSIM& 0.5715 &  0.5687 & 0.5534\\
\cline{3-6}
&&Time (s.)& 1897& 2170   & 2719\\
\hline
\hline
\multirow{6}{*}{\shortstack{ MAP }}&
\multirow{3}{*}{\shortstack{ Discrepancy \\principle}}
&SNR & 13.48& 13.60 & 13.39 \\
\cline{3-6}
&&SSIM&  0.5348 &  0.5393 &  0.5103 \\
\cline{3-6}
&&Time (s.)& 3049&  769 & 2644\\
\hhline{|~|-|----|}
&\multirow{3}{*}{\shortstack{ Best \\parameter}}
&SNR & 13.60& 13.71 & 13.75\\
\cline{3-6}
&&SSIM& 0.5568& 0.5605  & 0.5602 \\
\cline{3-6}
&&Time (s.)& 8390&  8477  & 2397\\
\hline 
\end{tabular}
}
\caption{ Restoration results for the image $\overline{\mathbf{x}}_4 $ with $x^+=20$ and $\sigma^2=9$. Uniform kernel with size $5 \times 5$. Initial SNR= 7.64 dB.}\label{t:5}
\end{table}
\begin{table}[t]
\centering
\tiny
\resizebox{0.9\textwidth}{!}{
\renewcommand{\arraystretch}{1.3}
\begin{tabular}{|l |l| l||l|l|l|}
\cline{4-6}
 \multicolumn{3}{c|}{ }&\shortstack{GAST }& SPoiss& WL2\\
\hhline{---|=|=|=|}
\multirow{9}{*}{\shortstack{ VBA }}&
\multirow{3}{*}{\shortstack{ Approx. 1 }}
&SNR &19.5  & 20.23   &\textbf{20.71} \\
\cline{3-6}
&&SSIM& 0.6649 & 0.7135 &\textbf{0.7793}\\
\cline{3-6}
&&Time (s.)&  16& 17  &\textbf{34 } \\
\hhline{|~|-|----|}
&\multirow{3}{*}{\shortstack{ Approx. 2 \\
$N_s=160$ }}
&SNR & 20.27  & 20.59  &20.56 \\
\cline{3-6}
&&SSIM& 0.7473 & 0.7660 & 0.7877 \\
\cline{3-6}
&&Time (s.)& 61 & 64   & 94\\
\hhline{|~|-|----|}
&\multirow{3}{*}{\shortstack{ Approx.2 \\
$N_s=640$ }}
&SNR & \textbf{20.35} & \textbf{20.67}   & 20.64  \\
\cline{3-6}
&&SSIM& \textbf{0.7563} &  \textbf{0.7798}  & 0.7989 \\
\cline{3-6}
&&Time (s.)& \textbf{195}& \textbf{197}  & 272  \\
\hline
\hline
\multirow{6}{*}{\shortstack{ MAP }}&
\multirow{3}{*}{\shortstack{ Discrepancy \\principle}}
&SNR & 19.39& 19.50 & 18.70 \\
\cline{3-6}
&&SSIM& 0.7458&  0.7550  &  0.7448  \\
\cline{3-6}
&&Time (s.)& 717& 1201 & 1087\\
\hhline{|~|-|----|}
&\multirow{3}{*}{\shortstack{ Best \\parameter}}
&SNR & 20.15& 20.41 & 20.44\\
\cline{3-6}
&&SSIM& 0.7535& 0.7594  & 0.7628 \\
\cline{3-6}
&&Time (s.)& 559&  125  & 253\\
\hline 
\end{tabular}
}
\caption{ Restoration results for image $\overline{\mathbf{x}}_5 $ with $x^+=20$ and  $\sigma^2=9$. Gaussian kernel with size $7 \times 7$, std 1. Initial SNR= 8.55 dB.}\label{t:6}
\centering
\resizebox{0.9\textwidth}{!}{
\renewcommand{\arraystretch}{1.3}
\begin{tabular}{|l|l |l|| l|l|l|}
\cline{4-6}
 \multicolumn{3}{c|}{ }&\shortstack{GAST }& SPoiss& WL2\\
\hhline{---|=|=|=|}
\multirow{9}{*}{\shortstack{ VBA }}&
\multirow{3}{*}{\shortstack{ Approx. 1 }}
&SNR &\textbf{ 14.17}  & {14.13}   & 13.90  \\
\cline{3-6}
&&SSIM& \textbf{0.7655} & {0.7647} & 0.7569\\
\cline{3-6}
&&Time (s.)&\textbf{9}   &  {8} &  26\\
\hhline{|~|-|----|}
&\multirow{3}{*}{\shortstack{ Approx.2 \\
$N_s=160$ }}
&SNR &14.1 &  14.13 & {14.09}  \\
\cline{3-6}
&&SSIM& 0.7605 &  0.7619 &{ 0.7620} \\
\cline{3-6}
&&Time (s.)& 104& 148   &{246} \\
\hhline{|~|-|----|}
&\multirow{3}{*}{\shortstack{ Approx. 2 \\
$N_s=640$ }}
&SNR &14.16  & \textbf{14.19} & \textbf{14.16}  \\
\cline{3-6}
&&SSIM& 0.7639 &\textbf{ 0.7650} &\textbf{ 0.7658 } \\
\cline{3-6}
&&Time (s.)& 332& \textbf{ 479} &\textbf{913}\\
\hline
\hline 
\multirow{6}{*}{\shortstack{ MAP }}&
\multirow{3}{*}{\shortstack{ Discrepancy \\principle}}
&SNR & 13.23& 13.29 & 13.32 \\
\cline{3-6}
&&SSIM& 0.7104&  0.7126  &  0.7117  \\
\cline{3-6}
&&Time (s.)& 2796&  4900 & 1045\\
\hhline{|~|-|----|}
&\multirow{3}{*}{\shortstack{ Best \\parameter}}
&SNR & 13.77& 13.79 & 13.84\\
\cline{3-6}
&&SSIM& 0.7565& 0.7570 & 0.7591\\
\cline{3-6}
&&Time (s.)& 10084&  10005 & 821\\
\hline 
\end{tabular}
}
\caption{ Restoration results for the image $\overline{\mathbf{x}}_6 $ with $x^+=100$ and $\sigma^2=36$. Uniform kernel with size $3 \times 3$. Initial SNR= 10.68 dB.}\label{t:9}
\end{table}
We evaluate the performance of the proposed approach for the restoration of images degraded by both blur and PG noise. We consider six test images, displayed in Figure \ref{fig:im}, whose intensities have been rescaled so that pixel values belong to a chosen interval $[0,x^+]$. Images $\overline{\mathbf{x}}_1$ and $\overline{\mathbf{x}}_6$ are HST astronomical images while images $\overline{\mathbf{x}}_2$, $\overline{\mathbf{x}}_3$, $\overline{\mathbf{x}}_4$ and $\overline{\mathbf{x}}_5$ correspond to the set of confocal microscopy images considered in \cite{clar:Chouzenoux_2015}. These images are then artificially degraded by an operator $\mathbf{H}$ modeling spatially invariant blur with point spread function $h$ and by PG noise with variance $\sigma^2$.

\subsubsection{Comparison with MAP approaches} 
\f{
In this first set of experiments, we choose a standard total variation prior, i.e.  $\kappa = 1/2$ and for every pixel $j \in \left\lbrace 1,\ldots, N \right\rbrace$, $\mathbf{D}_j\x ~=~ \big[[\nabla^h \x]_j, [\nabla^v \x]_j\big]^\top~\in~\mathbb{R}^2$
where $\nabla^h$ and $\nabla^v$ are the discrete gradients computed in the horizontal and vertical directions. As a result, $J =N$ and $S=2$.
}
The goal of our experiments is twofold. First, for each likelihood, we compare the accuracy of the two proposed approximations of the covariance matrix described in Section~\ref{SecImplIss} namely the diagonal approximation (denoted as approximation~1) and the Monte Carlo averaging strategy (designated as approximation 2) with different number of samples $N_s$, namely $N_s = 160$ or $640$. Second, the proposed method is compared with state-of-the-art algorithms that compute the MAP estimate \f{for the considered likelihoods}. More specifically, as GAST and SPoiss \f{data fidelity terms} are convex and Lipschitz differentiable, we use the method presented in \cite{clar:Chouzenoux_2015} where a primal-dual splitting algorithm was proposed to minimize convex penalized criteria \f{in the context of Poisson-Gaussian image restoration}. For the WL2 approximation, \f{the corresponding \f{data fidelity function} is not convex so the previous method could not be applied anymore. We thus }consider the variable metric forward-backward algorithm \f{proposed} in \cite{clar:Repetti_2012} \f{for the minimization of penalized WL2 functionals}. For \f{the aforementioned MAP approaches, it is necessary to set the regularization parameter $\gamma$ that balances the fidelity to the observation model and the considered prior. In this respect}, we test two variants. \f{In the first variant}, we estimate the regularization parameter using \f{an approach based on the} discrepancy principle \cite{bardsley2009regularization,bertero2010discrepancy, zanni2015numerical}. 
In the second variant, $\gamma$ is adjusted empirically to achieve the maximum Signal-to-Noise Ratio (SNR) value defined as SNR$=20 \log_{10} \left( \frac{\Vert \overline{\x} \Vert}{\Vert \overline{\x}-\hat{\x} \Vert} \right)$, which requires the availability of the true image. 
\begin{figure*}[h!]
\centering
\begin{tabular}{ccc}
\subfloat[][GAST - Approx. 1]{\includegraphics[scale=0.45]{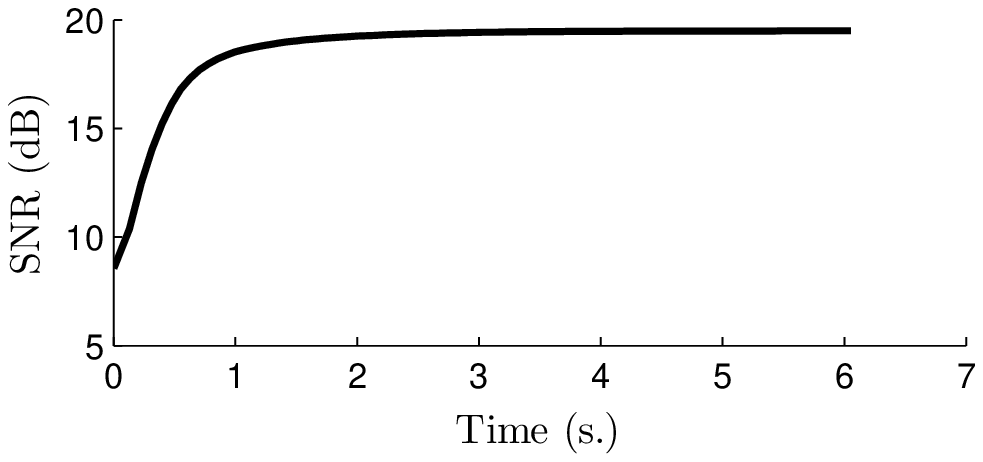} }&\subfloat[][GAST - Approx. 2 ($N_s=160$)]{\includegraphics[scale=0.45]{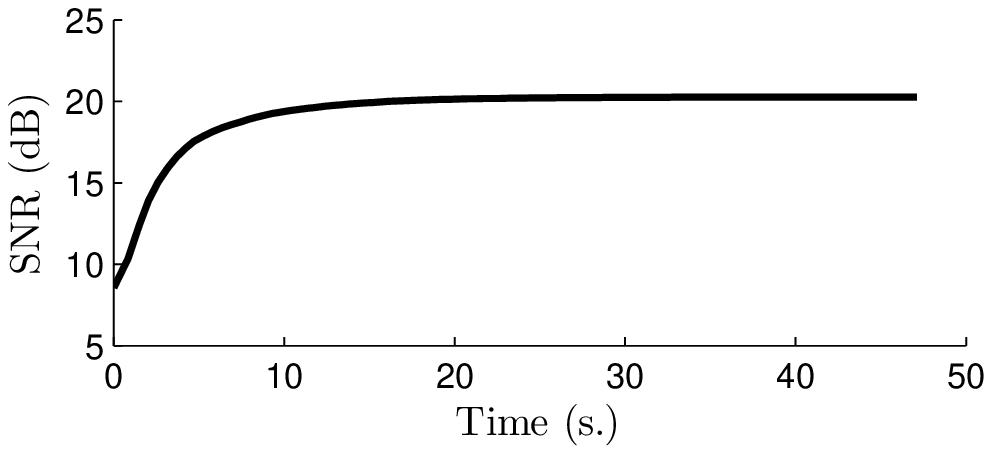}}&  \subfloat[][GAST - Approx. 2 ($N_s=640$)]{\includegraphics[scale=0.45]{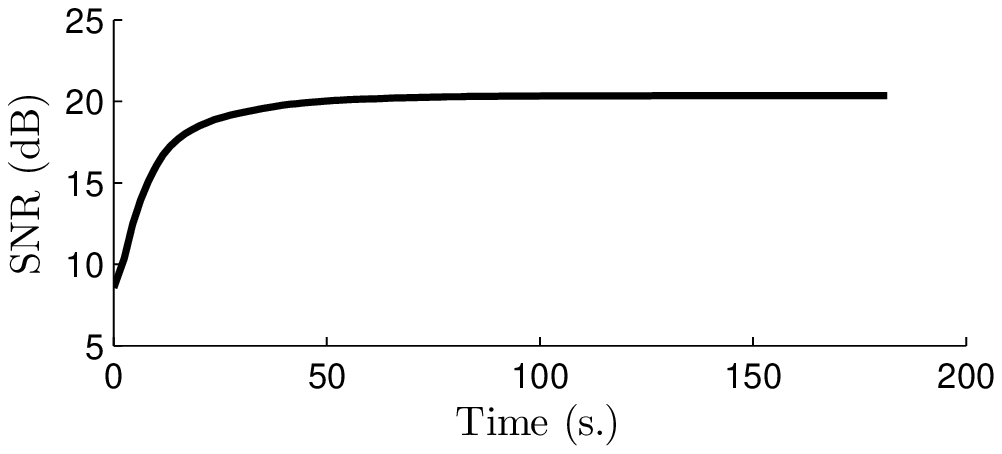}}\\
\subfloat[][SPoiss - Approx. 1]{\includegraphics[scale=0.45]{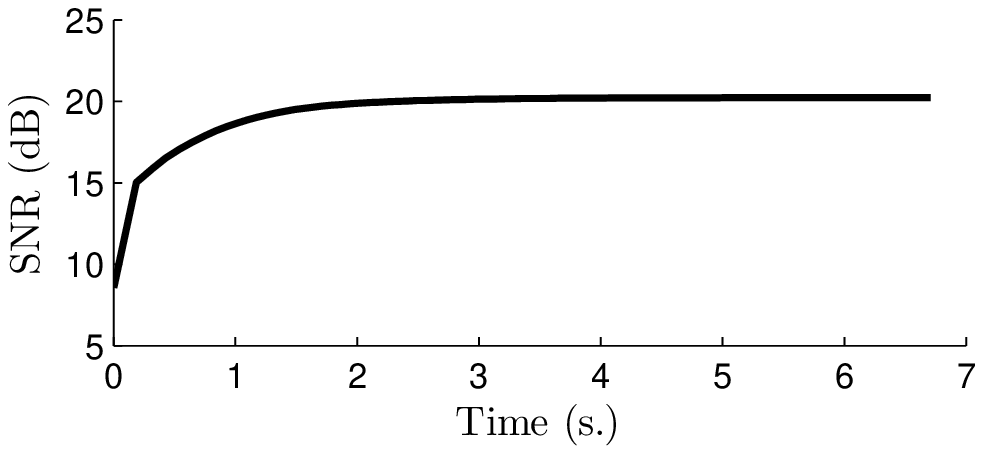} }& \subfloat[][SPoiss - Approx. 2 ($N_s=160$)]{\includegraphics[scale=0.45]{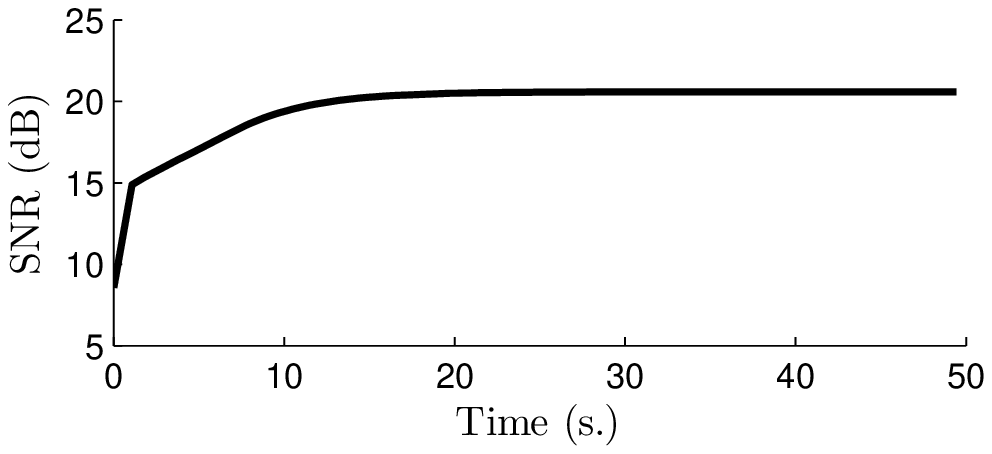}}& \subfloat[][SPoiss - Approx. 2 ($N_s=640$)]{\includegraphics[scale=0.45]{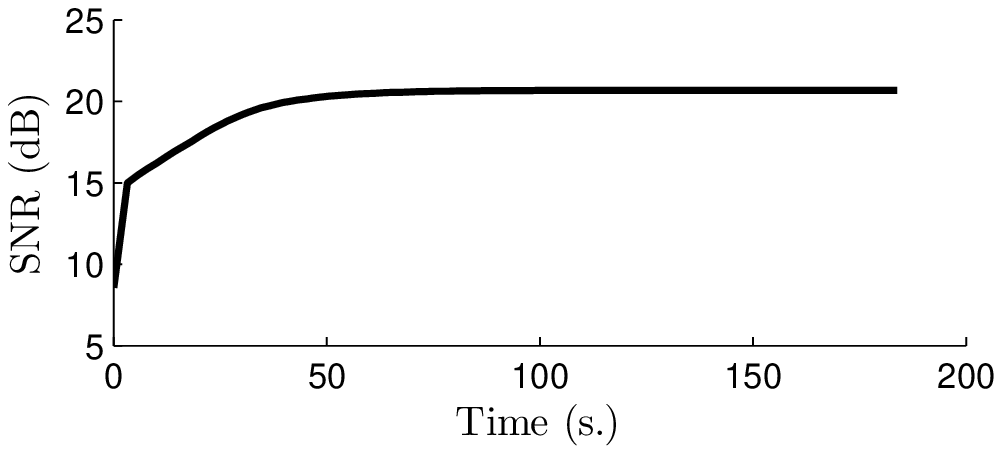}}\\
\subfloat[][WL2 - Approx. 1]{\includegraphics[scale=0.45]{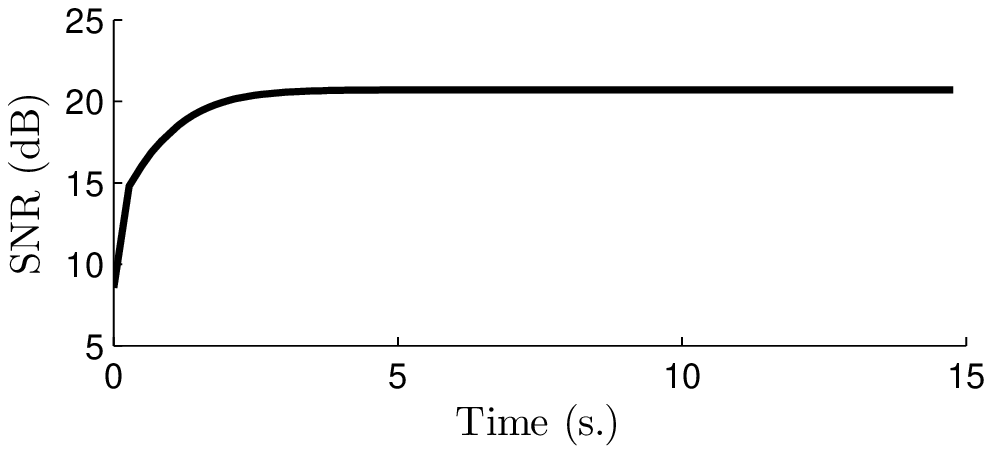} }& \subfloat[][WL2 - Approx. 2 ($N_s=160$)]{\includegraphics[scale=0.45]{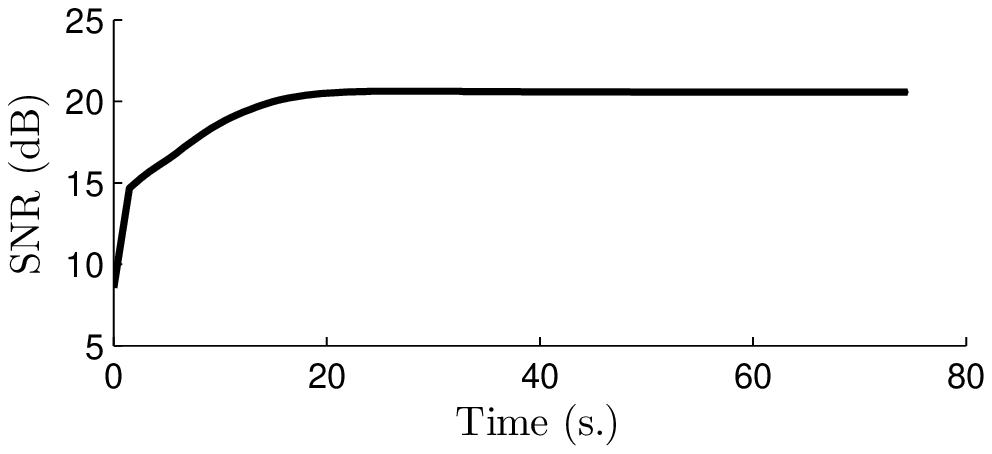}}& \subfloat[][WL2 - Approx. 2 ($N_s=640$)]{\includegraphics[scale=0.45]{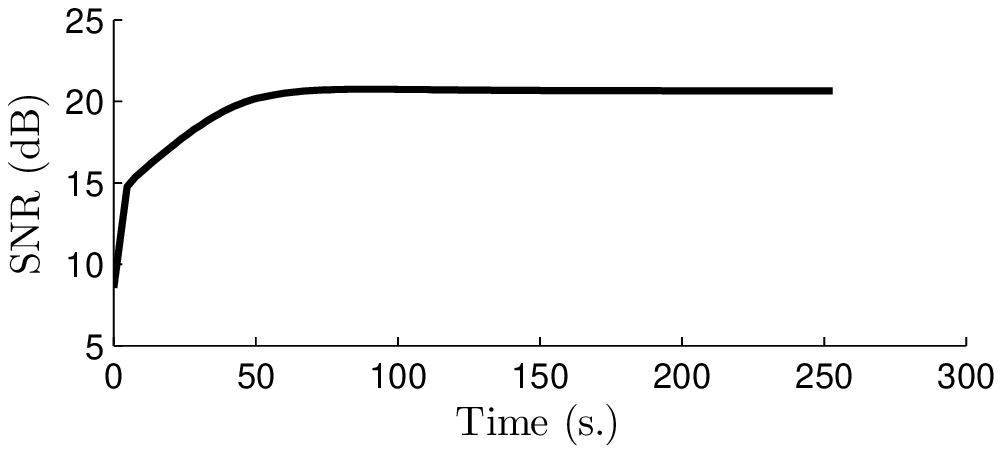}}
\end{tabular}
\caption{Evolution of SNR with respect to time for image $\bar{\mathbf{x}}_5$ using different data-fidelity terms and covariance approximations.}
\label{figcorr}
\end{figure*}

%
%
Tables \ref{t:2}-\ref{t:9} report the results obtained with the different images in terms of SNR, SSIM \cite{wang2004image}, and approximate computation time needed for convergence. For each likelihood, we emphasize in bold the approximation of the covariance matrix that achieves the best quantitative result in the shortest computational time. Simulations were performed on an Intel(R) Xeon(R) CPU E5-2630, @ 2.40~GHz, using a Matlab7 implementation. All tested methods were initialized with the degraded image. \f{Moreover, the initial value of the regularization parameter results from a maximum likelihood estimation performed on the degraded image}. The Monte Carlo averaging approximation was computed using parallel implementation with $16$ cores by means of the command \emph{PARFOR} of the Matlab\textsuperscript{\textregistered} Parallel Computing Toolbox\textsuperscript{TM}. \f{The iterations of VBA were run until the following stopping criterion is satisfied: $\dfrac{\Vert \mathbf{x}^{(t+1)}-\mathbf{x}^{(t)} \Vert }{\Vert \mathbf{x}^{(t)} \Vert }\leqslant \varepsilon$. We have set $\varepsilon=10^{-6}$, as it was observed to lead to a practical stabilization in terms of restoration quality. This can be checked by inspecting Figure \ref{figcorr} illustrating the evolution of the SNR of the restored image along time, until the achievement of the stopping criterion, in the test case from Table \ref{t:6}. For the MAP-based approaches, the computational time includes the search of the regularization parameter.} 

One can observe that in most studied situations (see~Tables~\ref{t:4}-\ref{t:9}), the diagonal approximation of the covariance matrix appears to give satisfactory qualitative results after a small computation time. However, in few other situations (see Tables \ref{t:2} and \ref{t:3}), it fails to capture the real qualitative structures of the covariance matrix leading to a poorer performance. The latter issue is well alleviated by using the Monte Carlo approximation where good results, in terms of image quality, are achieved within $N_s = 160$ samples which is equivalent to a relative approximation error equal to $11 \%$.  A few improvements are observed by decreasing the approximation error to $5 \%$ using $N_s = 640$ samples. 

We also notice that the GAST approximation does not seem to be suitable for very low count images (see Tables \ref{t:2} and \ref{t:3}), whereas, the other likelihoods lead to competitive results in all the experiments. The best tradeoff between restoration quality and small computational time seems to be achieved by the WL2 approximation. 

Finally, it can be observed that, in Tables \ref{t:2} and \ref{t:4}, our VBA method yields comparable performance in terms of SNR to the MAP estimate when the latter is computed with the optimal regularization parameter, while our approach requires less time to converge. In the other experiments, our approach leads to the best qualitative results. For instance, in Table \ref{t:3}, the gain in terms of SNR reaches up to 0.2 dB compared with the MAP estimator using the best regularization parameter, but our approach needs more time to converge. In Tables \ref{t:5}, \ref{t:6} and \ref{t:9}, we achieve both the best quantitative results and the smallest computational time. It should be noted that for most tested scenarii, discrepancy based approaches perform relatively poorly compared with the other methods, especially in the case of low count images (see Table \ref{t:2}). 

\begin{figure}[htb]
\centering
\begin{tabular}{cc}
\subfloat[][Degraded image with SNR= -2.55 dB (Uniform kernel $5 \times 5$, $x^+=10$ and $\sigma^2=4$). ]{\includegraphics[scale=0.45]{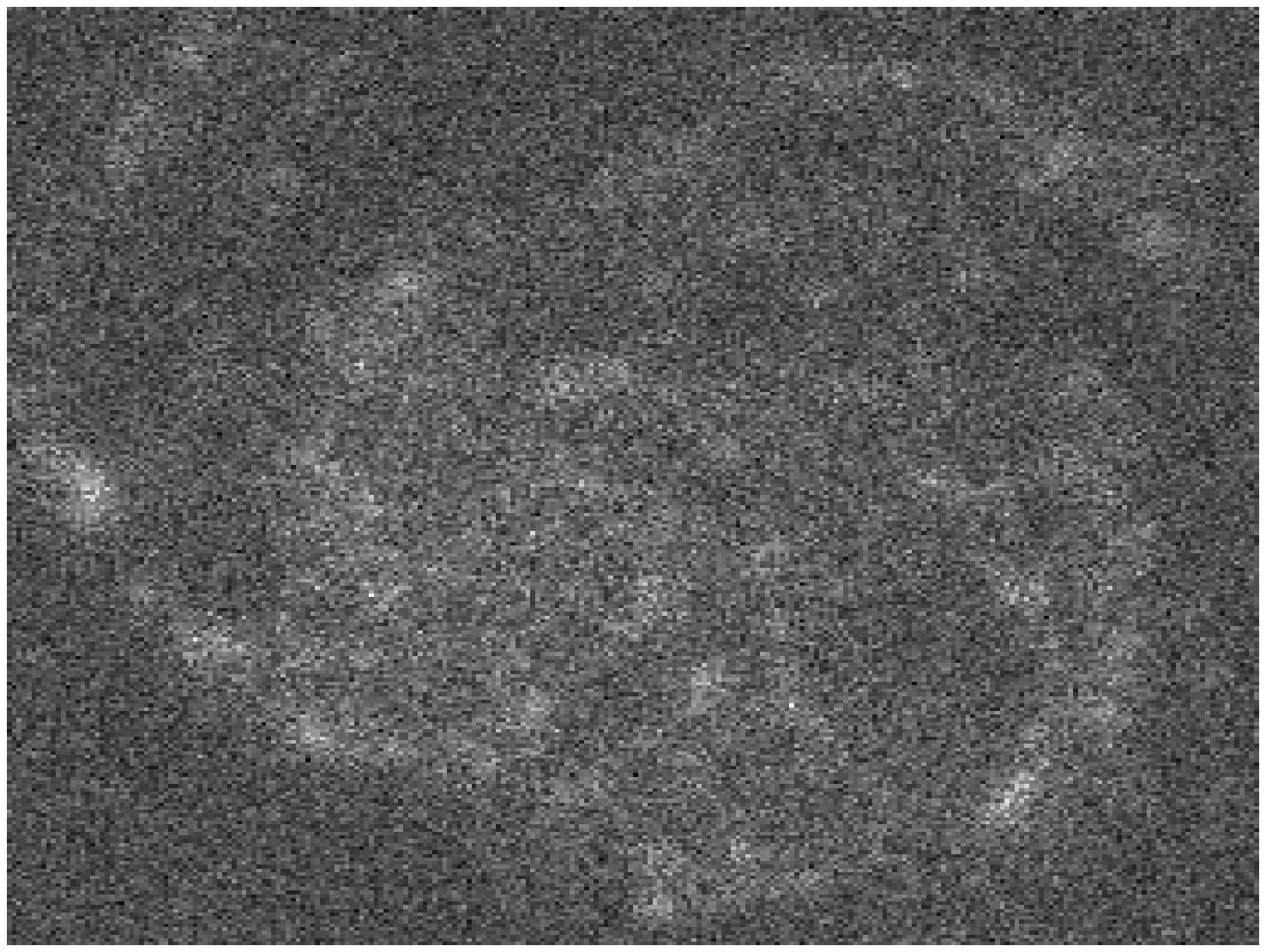} }&
 \subfloat[][Restored image with VBA approach using the Monte Carlo approximation with $640$ samples: SNR= 10.27 dB]{\includegraphics[scale=0.45]{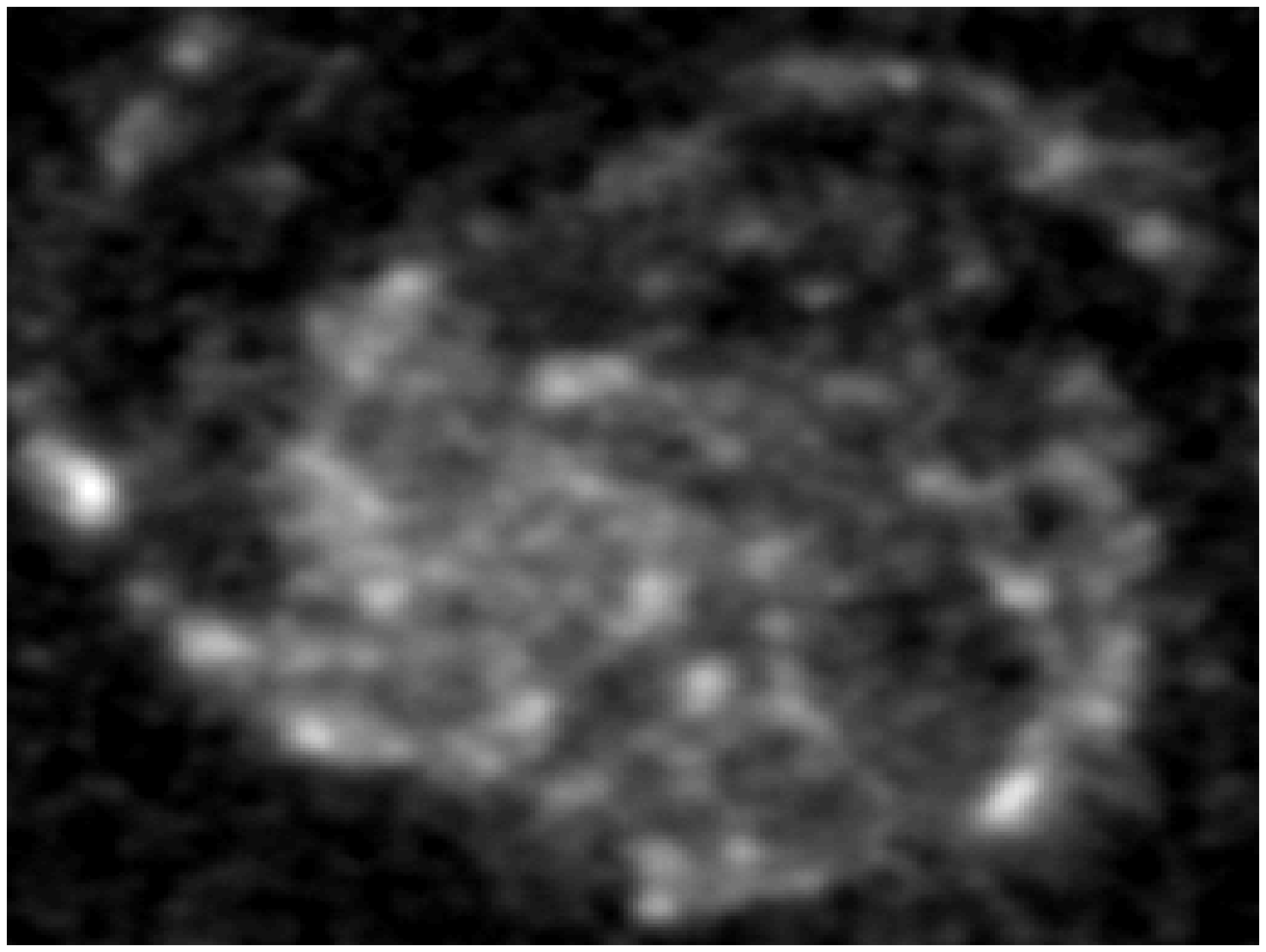} }
\\
  \subfloat[][Restored image with discrepancy principle: SNR= 10.17 dB]{\includegraphics[scale=0.45]{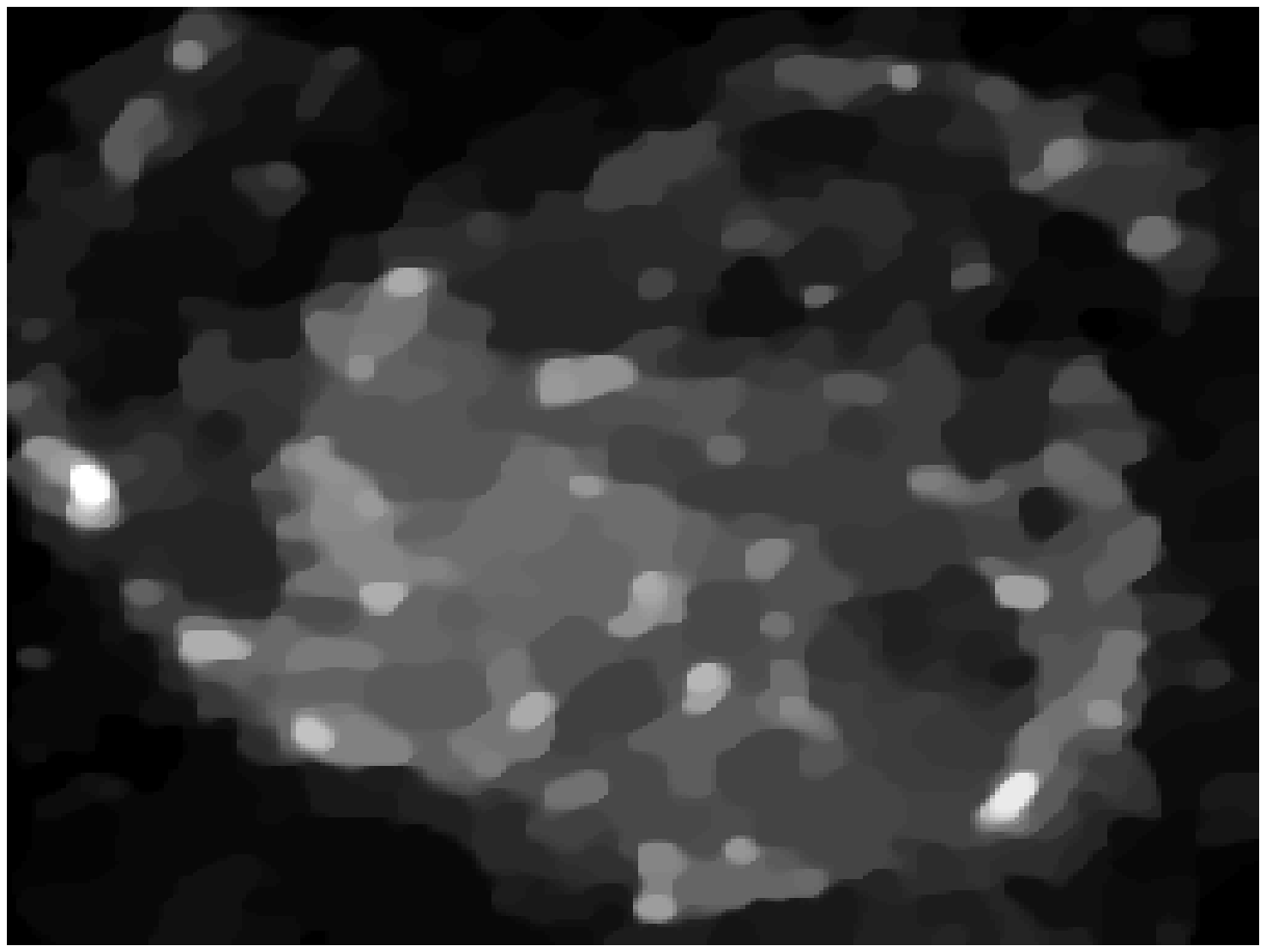} }&
   \subfloat[][Restored image with best parameter: SNR= 10.39 dB]{\includegraphics[scale=0.45]{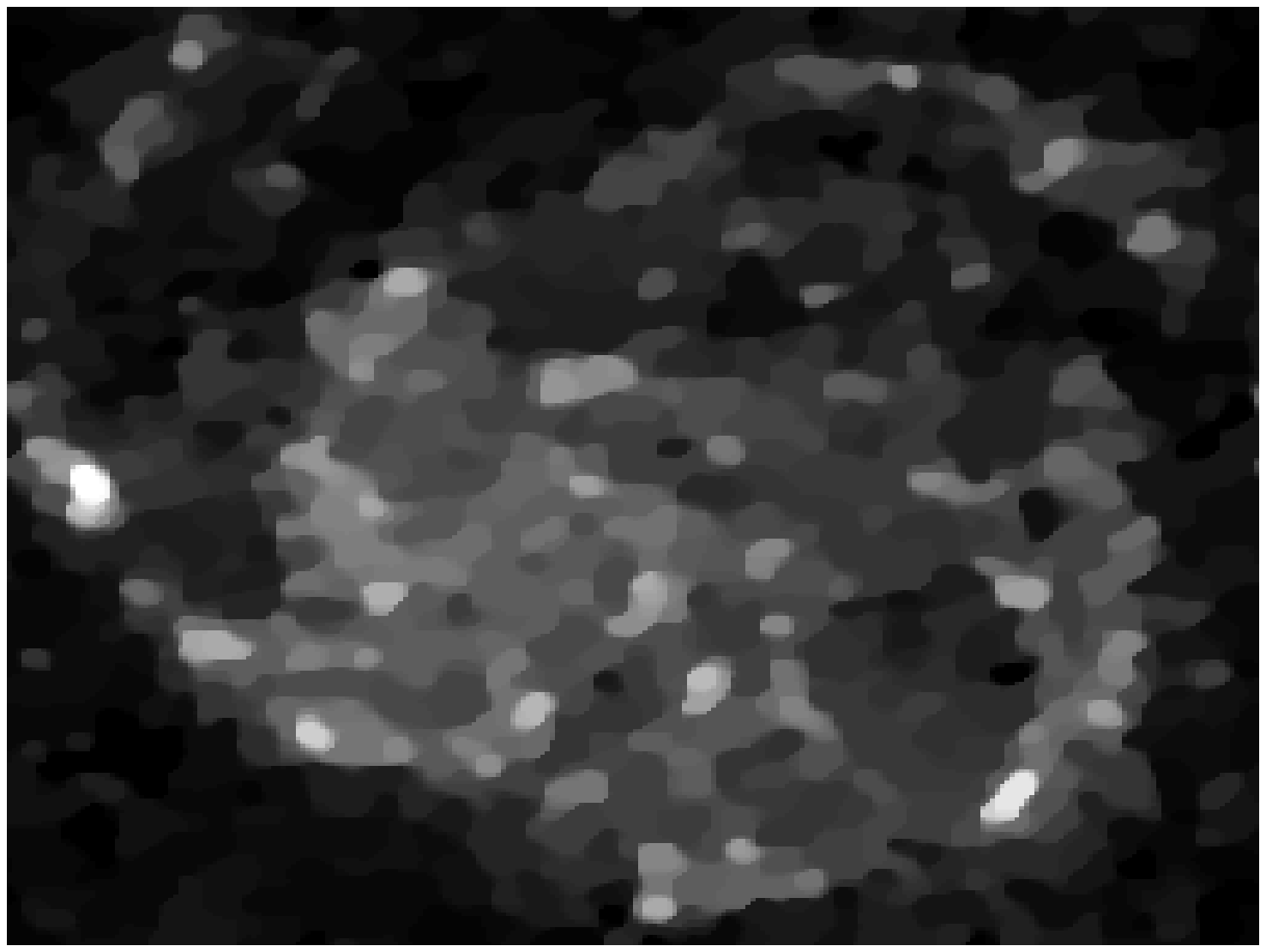} }
\end{tabular}
\caption{Restoration results for image $\bar{\mathbf{x}}_1$ using WL2 approximation.}
\label{fig1}
\end{figure}


\begin{figure}[htb]
\centering
\begin{tabular}{cc}
\subfloat[][Degraded image with SNR= 2.21 dB (Gaussian kernel $25 \times 25$, std 1.6, $x^+=12$ and $\sigma^2=9$).]{\includegraphics[scale=0.45]{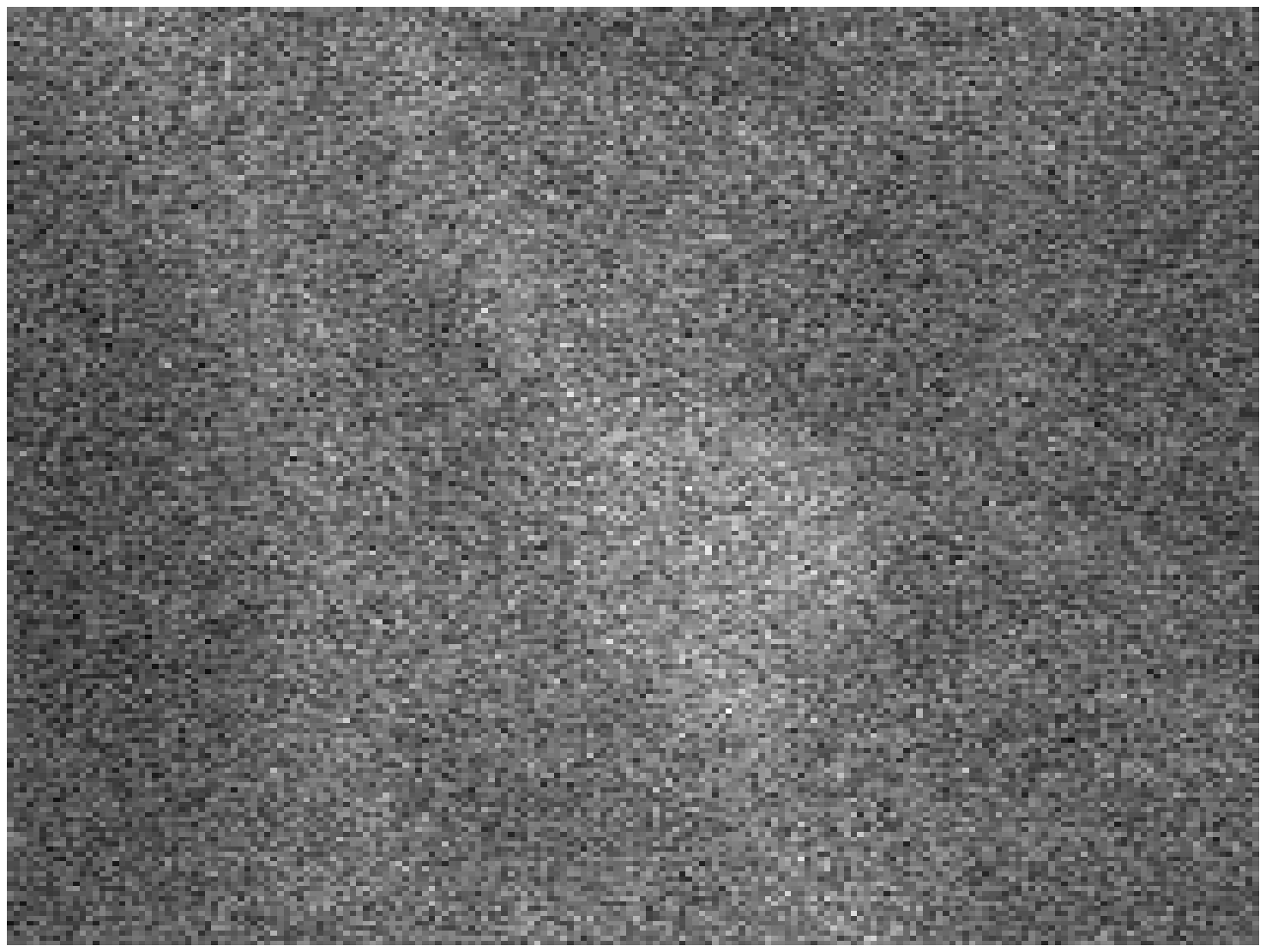} }&
 \subfloat[][Restored image with VBA approach using the Monte Carlo approximation with $640$ samples: SNR= 19.12 dB]{\includegraphics[scale=0.45]{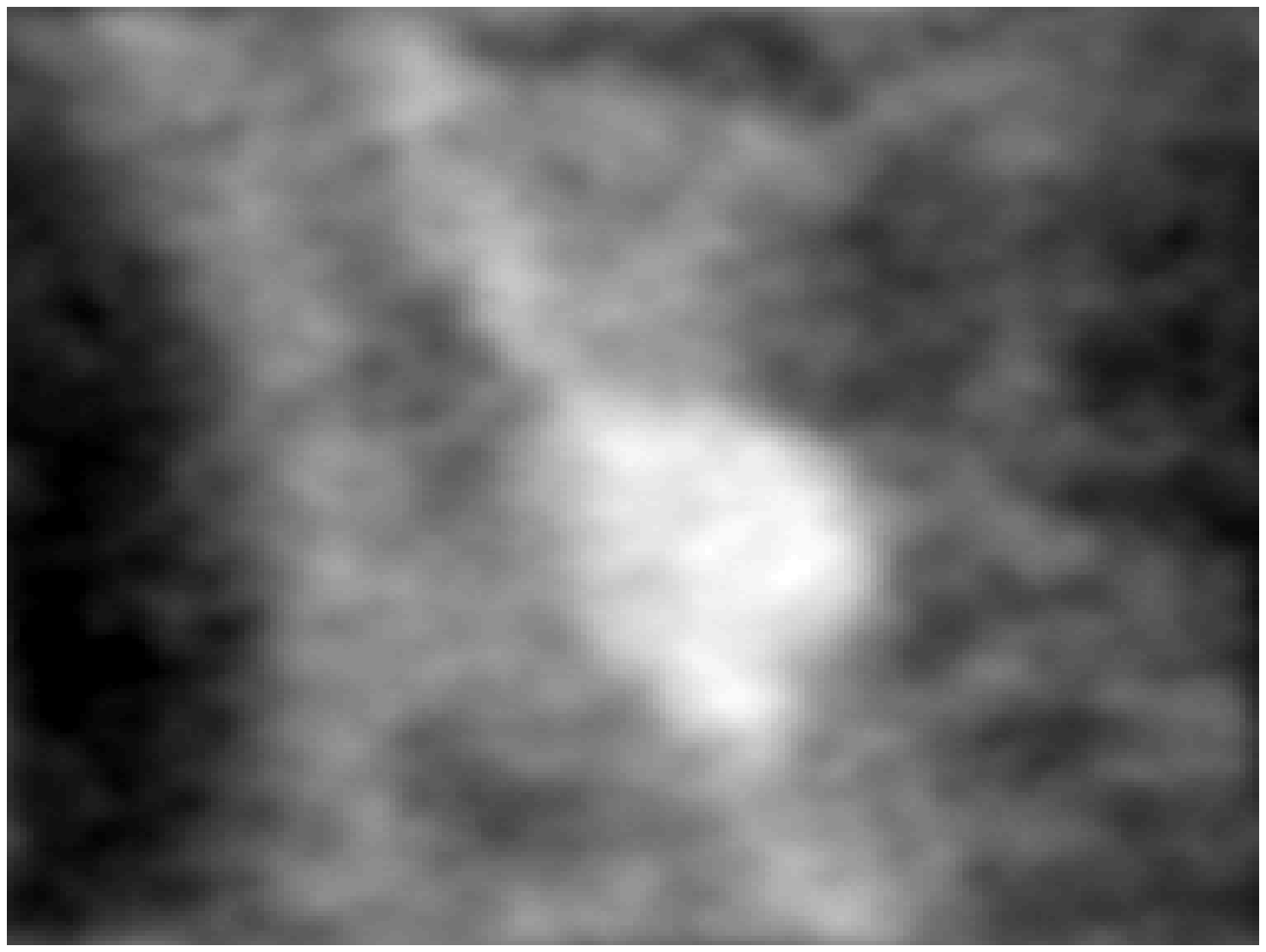} }\\
  \subfloat[][Restored image with discrepancy principle: SNR= 17.41 dB]{\includegraphics[scale=0.45]{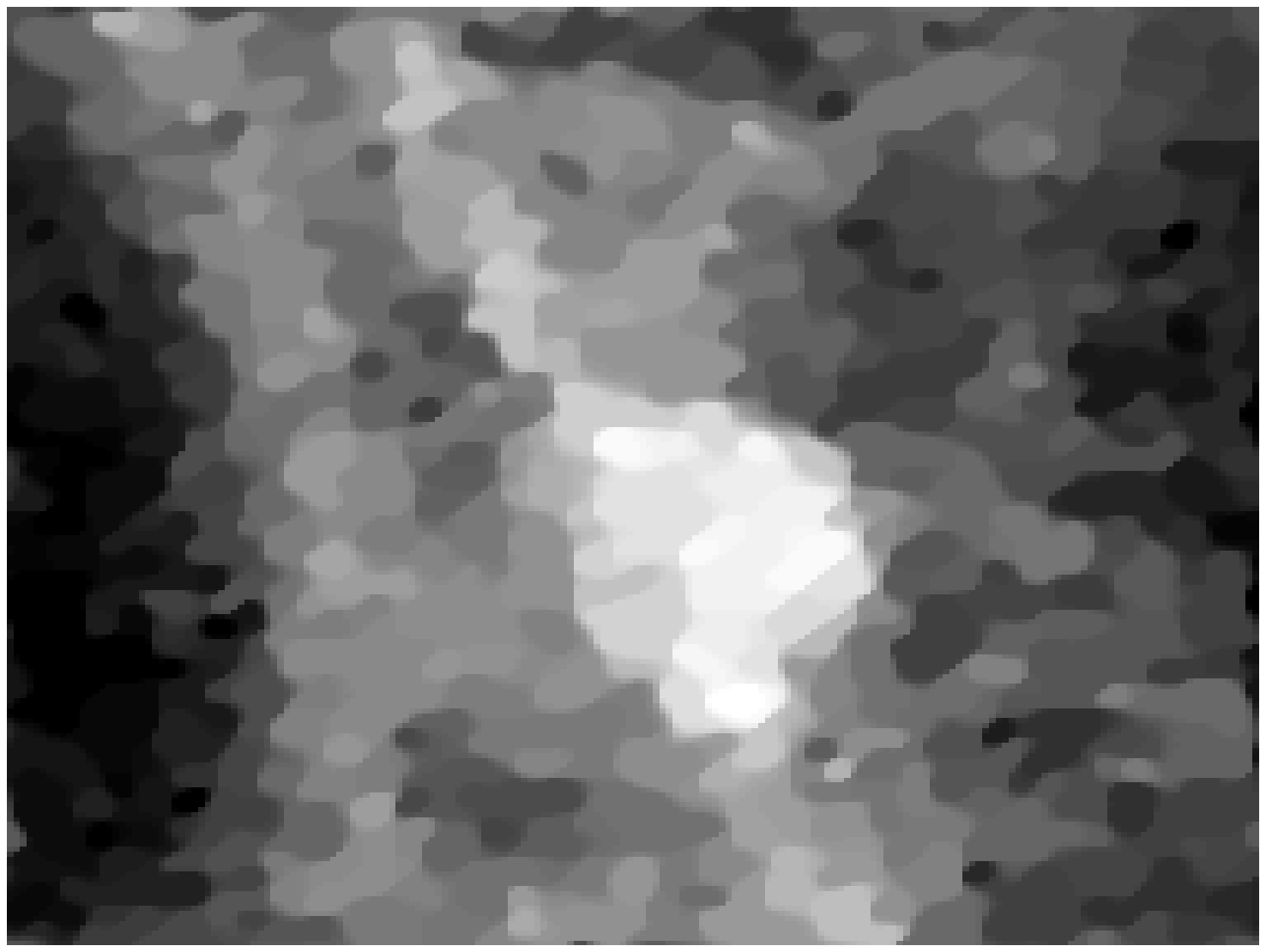} }&
   \subfloat[][Restored image with best parameter: SNR= 18.73 dB]{\includegraphics[scale=0.45]{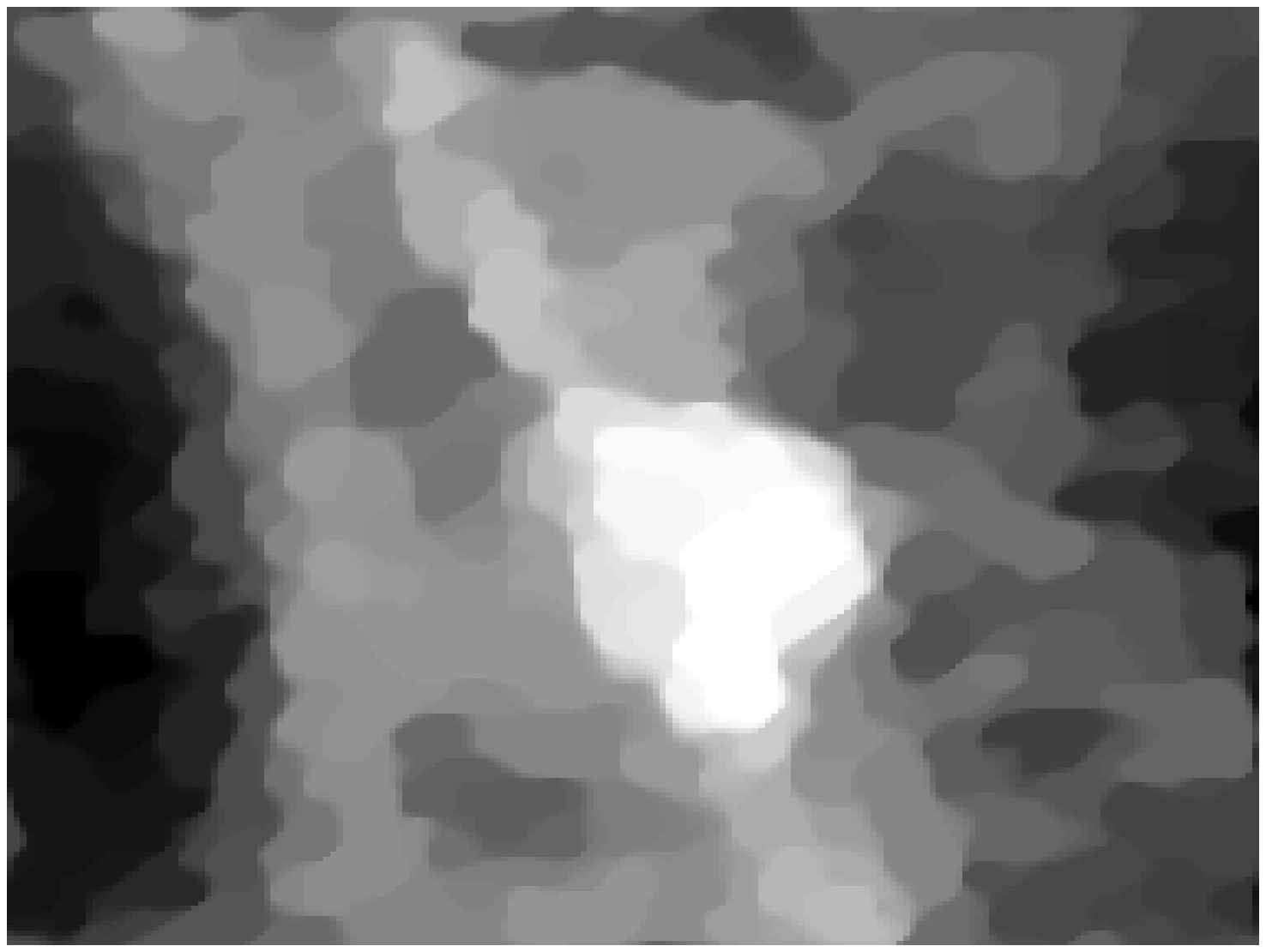} }
\end{tabular}
\caption{Restoration results for image $\bar{\mathbf{x}}_2$ using SPoiss approximation.}
\label{fig2}
\end{figure}


\begin{figure}[htb]
\centering
\begin{tabular}{cc}
\subfloat[][Degraded image with SNR= 3.14 dB (Uniform kernel $5 \times 5$, $x^+=15$ and $\sigma^2=9$).]{\includegraphics[scale=0.45]{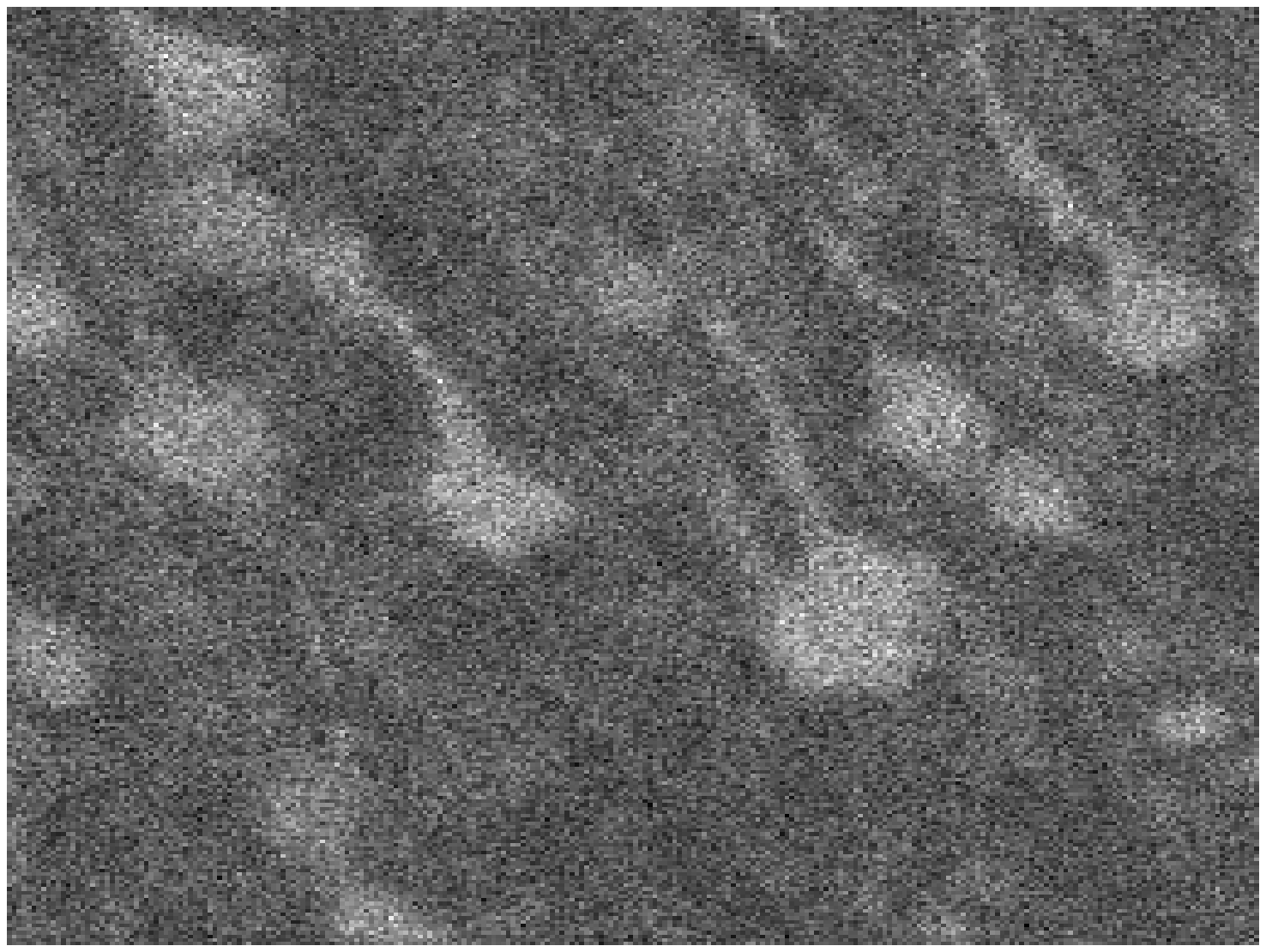} }&
 \subfloat[][Restored image with VBA approach using the Monte Carlo approximation with $640$ samples: SNR= 12.36 dB]{\includegraphics[scale=0.45]{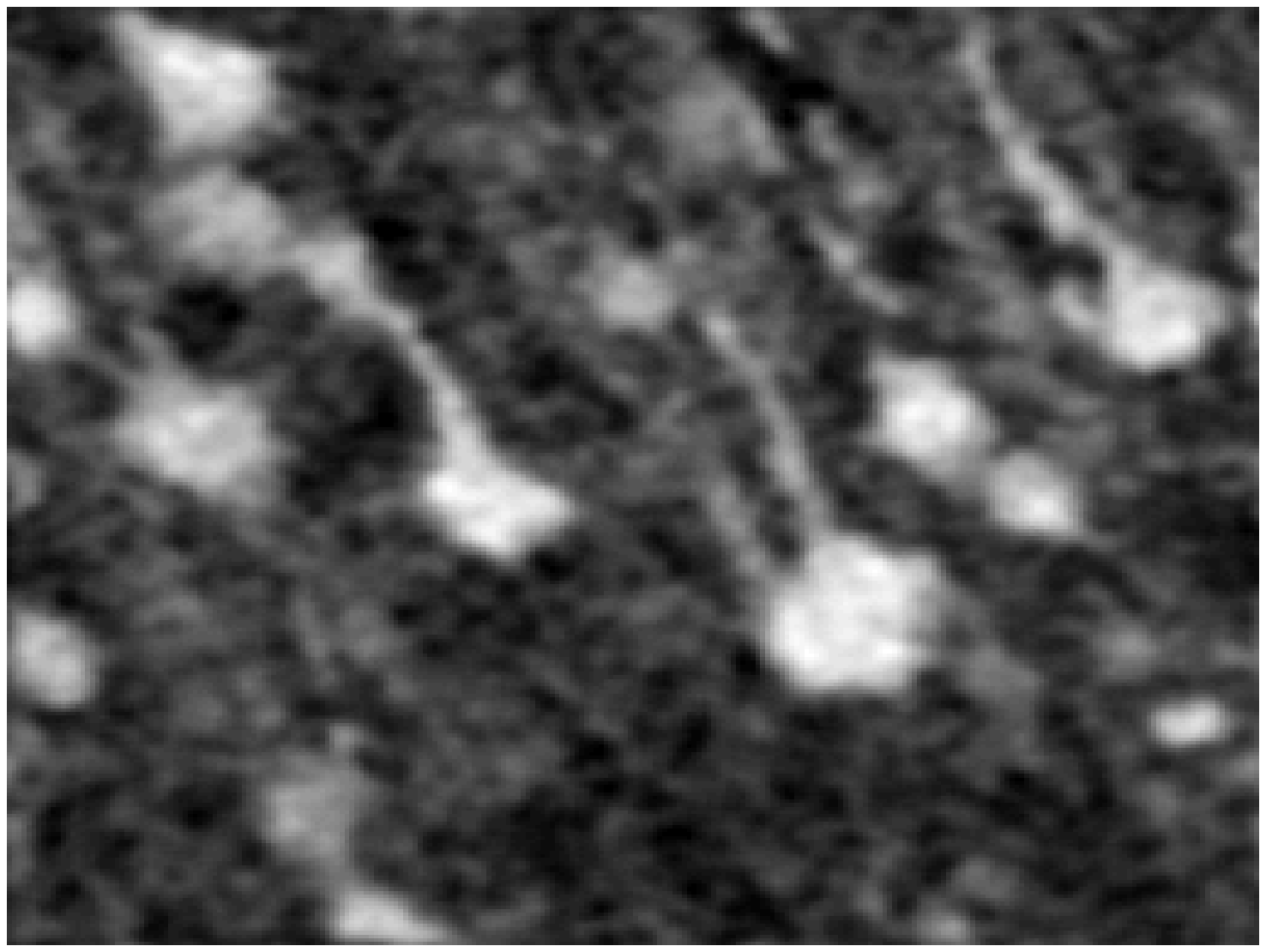} }\\
  \subfloat[][Restored image with discrepancy principle: SNR= 12.38 dB]{\includegraphics[scale=0.45]{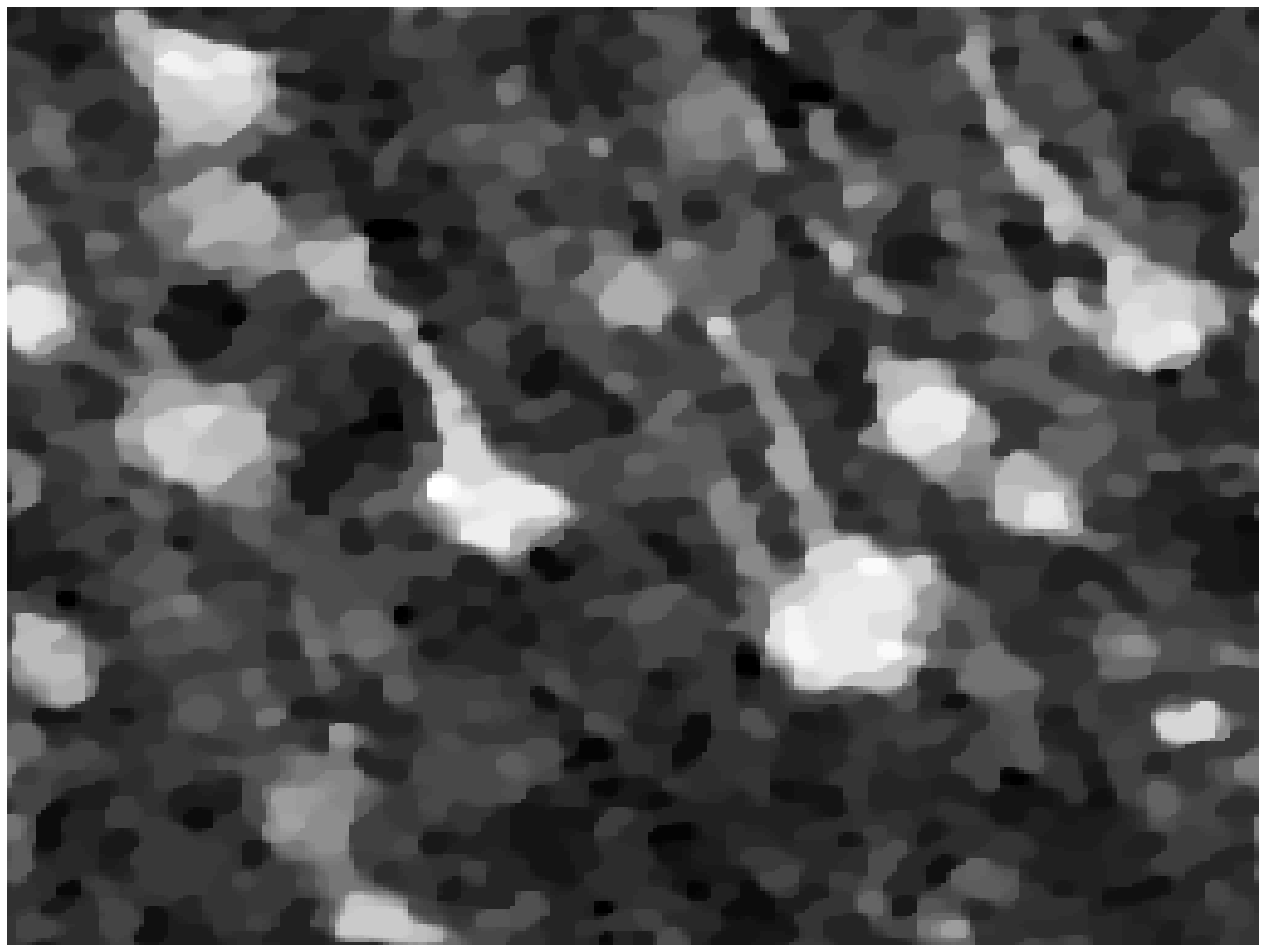} }&
   \subfloat[][Restored image with best parameter: SNR= 12.45 dB]{\includegraphics[scale=0.45]{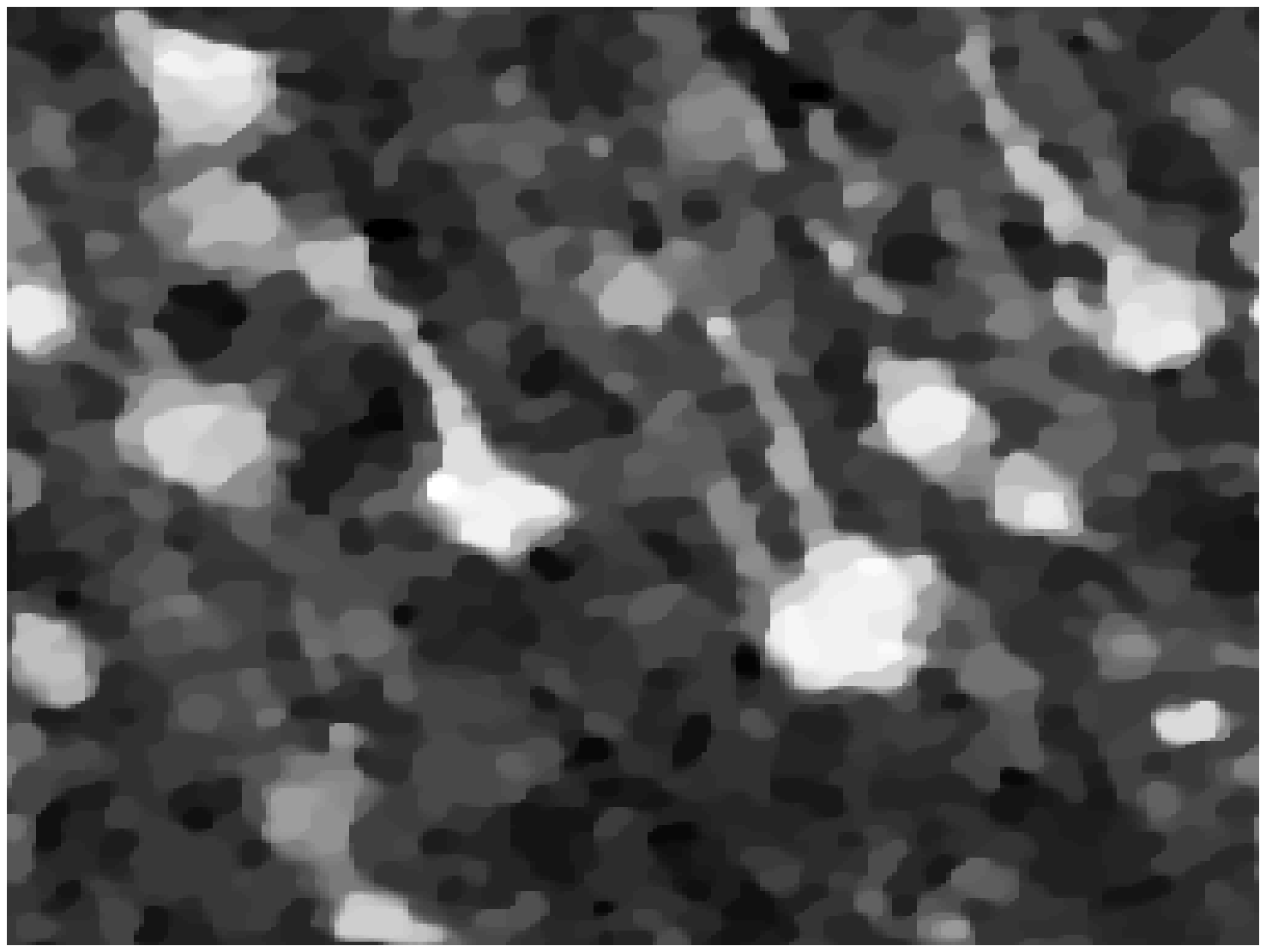} }
\end{tabular}
\caption{Restoration results for image $\bar{\mathbf{x}}_3$ using SPoiss approximation.}
\label{fig3}
\end{figure}


\begin{figure}[htb]
\centering
\begin{tabular}{cc}
\subfloat[][Degraded image with SNR=7.64 dB (Uniform kernel $5 \times 5$, $x^+=20$ and $\sigma^2=9$).]{\includegraphics[scale=0.45]{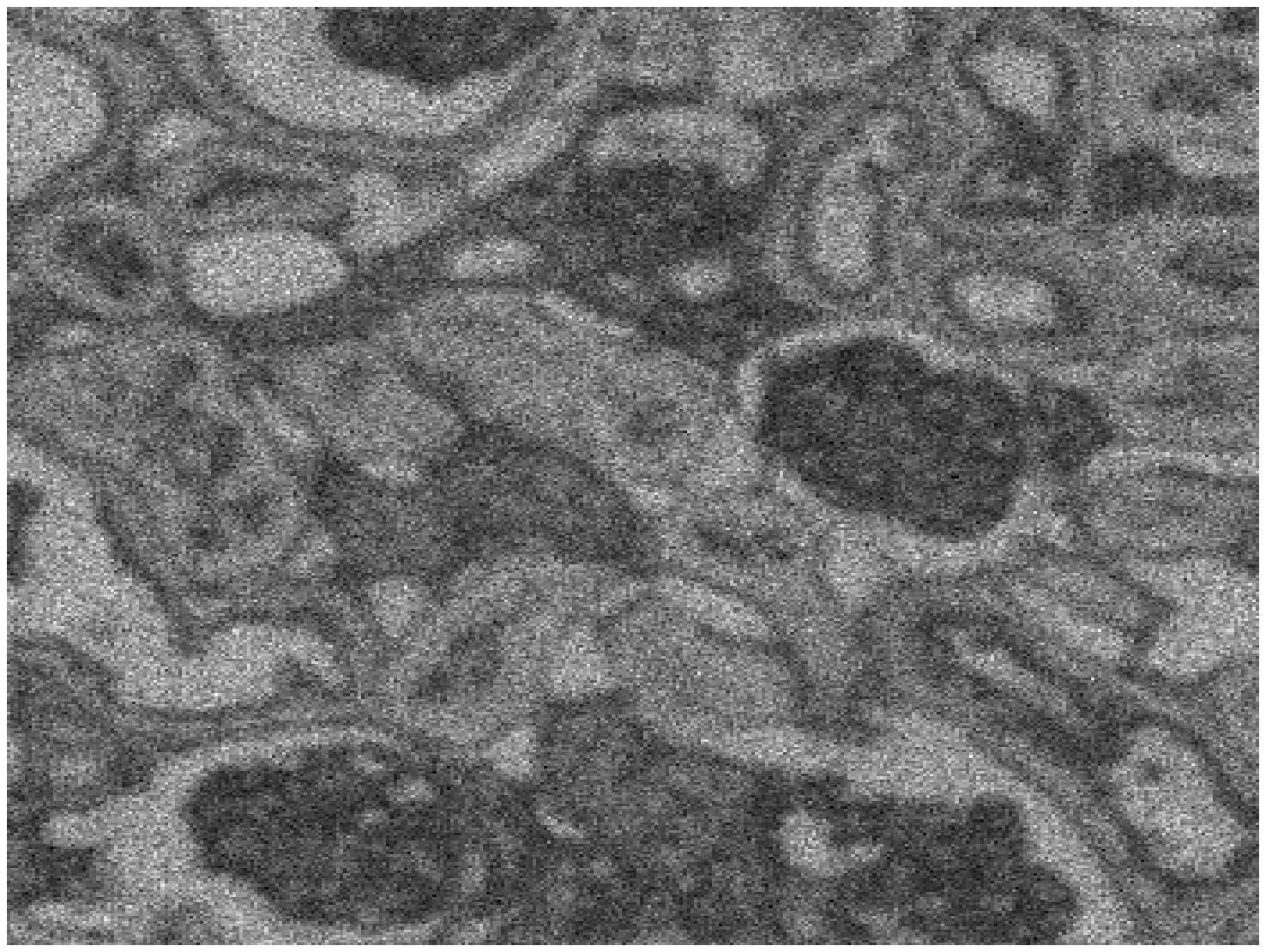} }&
 \subfloat[][Restored image with VBA approach using the diagonal approximation: SNR= 13.80 dB]{\includegraphics[scale=0.45]{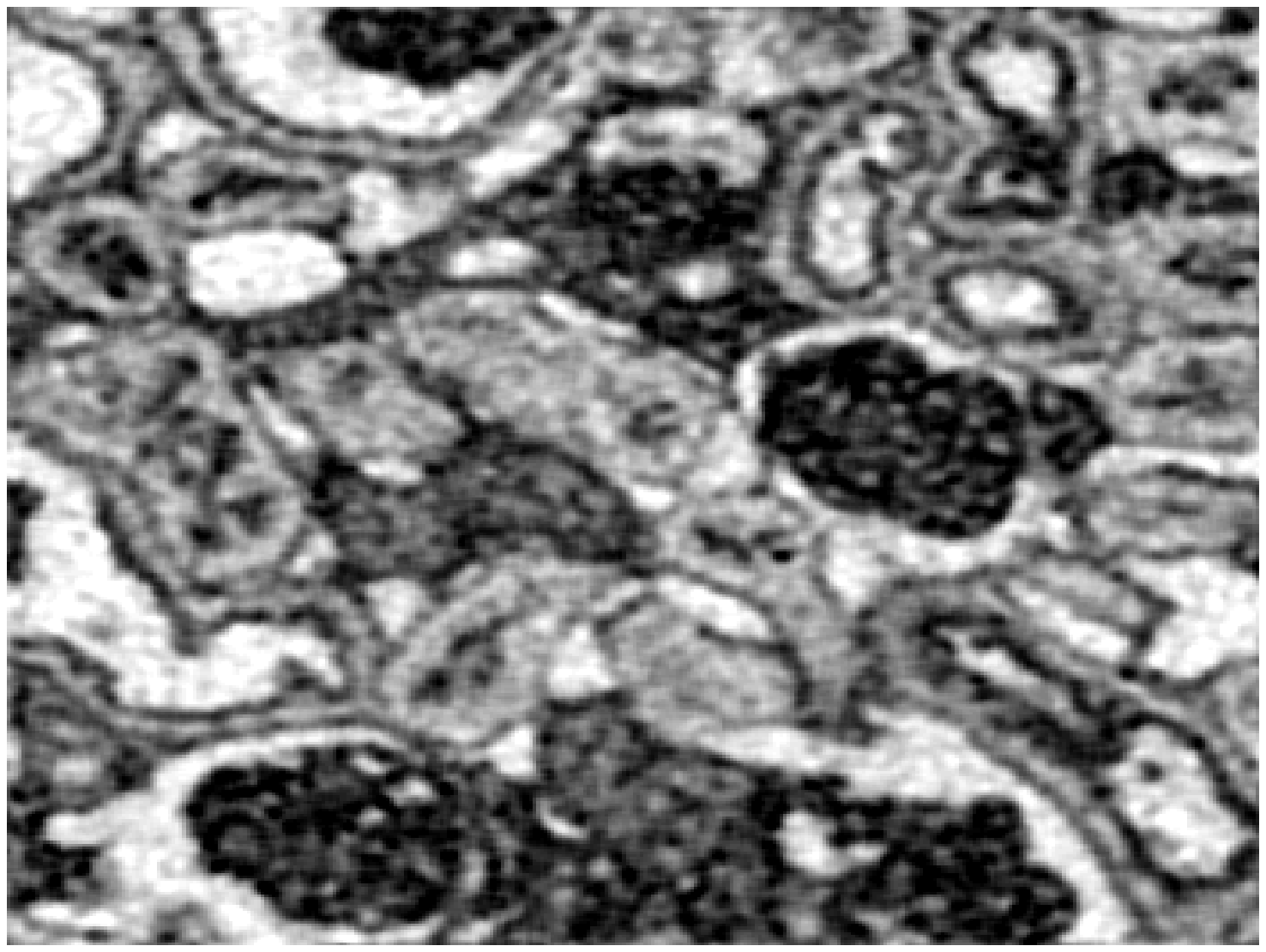} }\\
  \subfloat[][Restored image with discrepancy principle: SNR= 13.48 dB]{\includegraphics[scale=0.45]{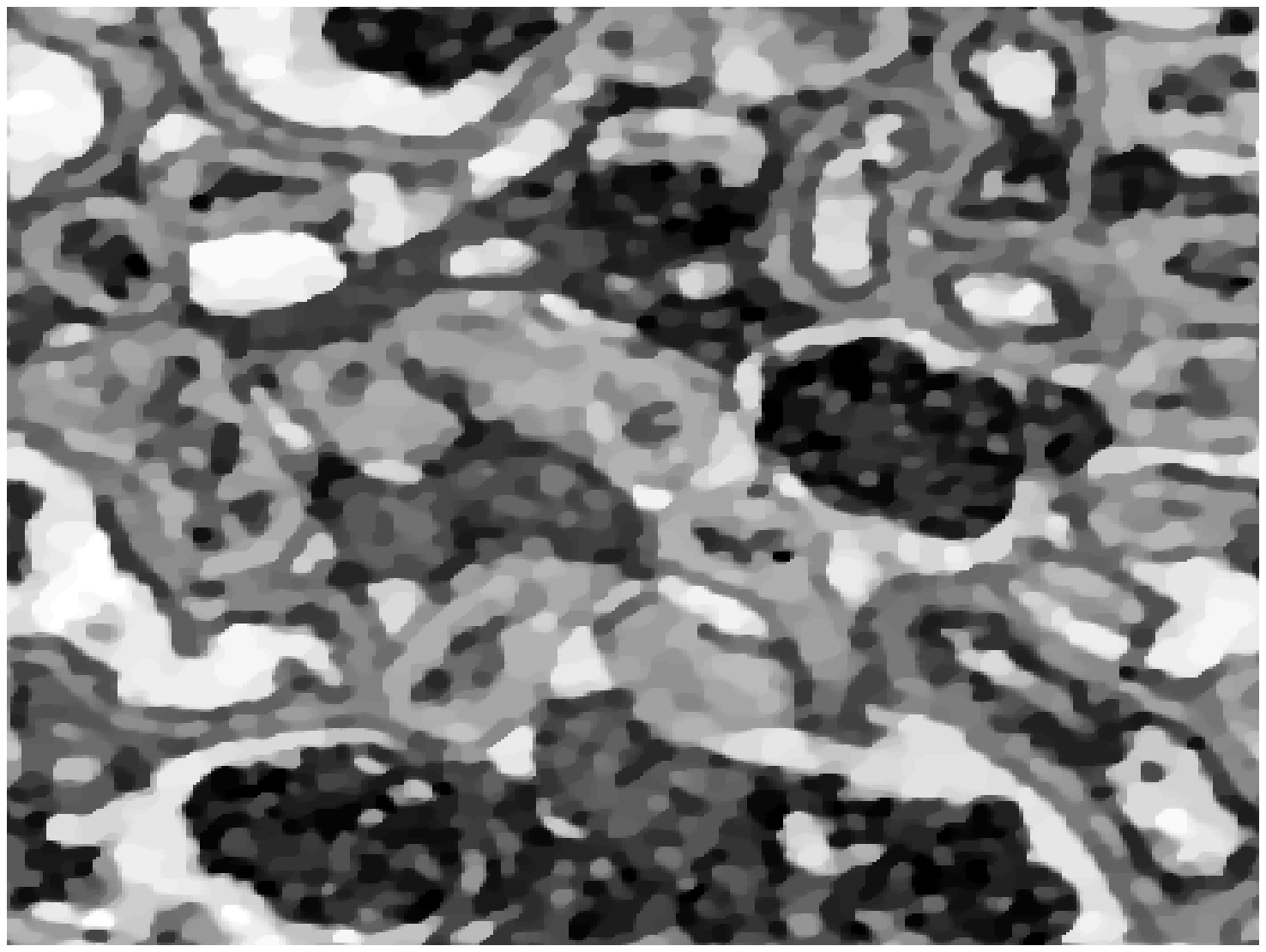} }&
   \subfloat[][Restored image with best parameter: SNR= 13.60 dB]{\includegraphics[scale=0.45]{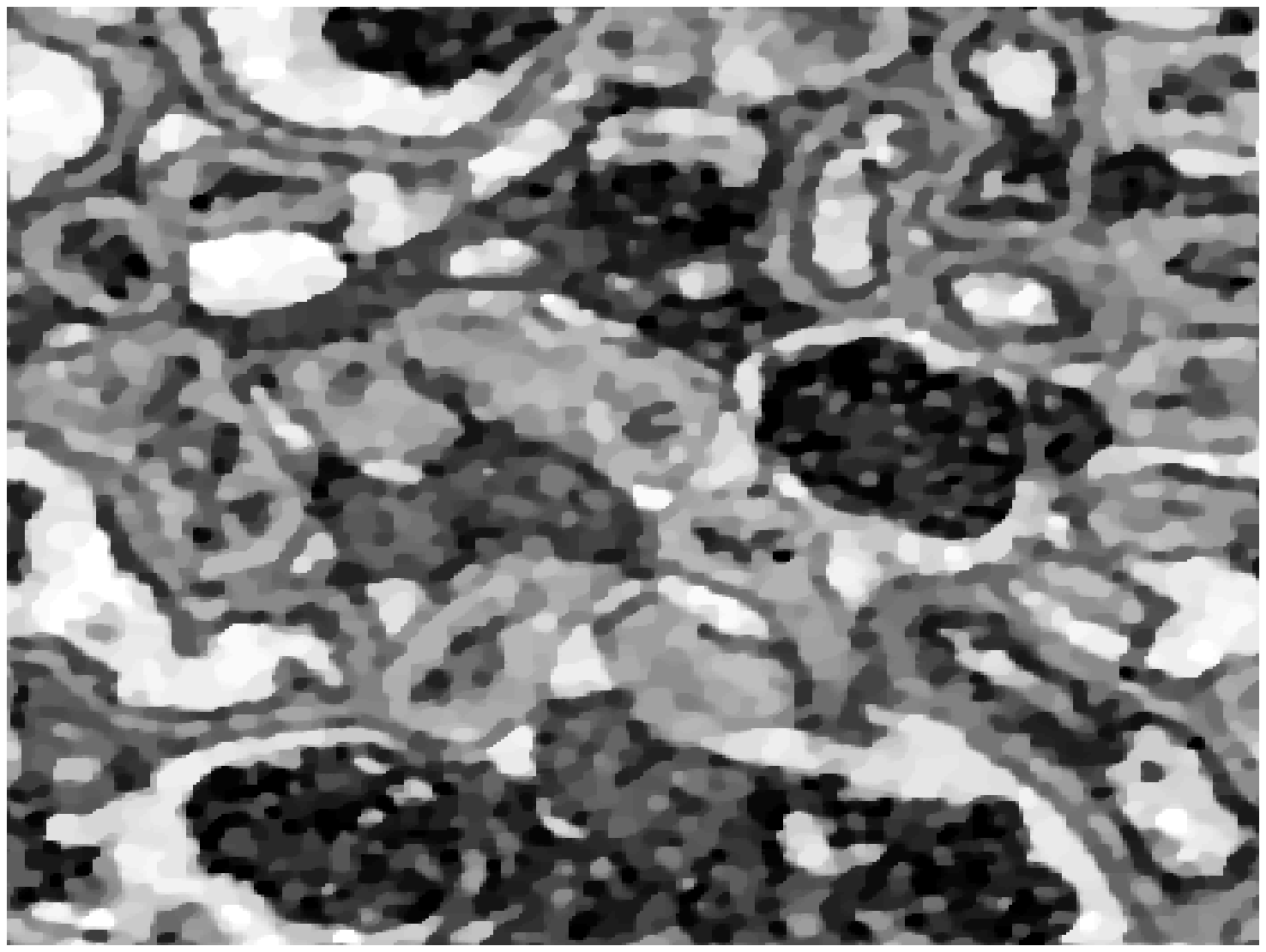} }
\end{tabular}
\caption{Restoration results for image $\bar{\mathbf{x}}_4$ using GAST approximation.}
\label{fig4}
\end{figure}

\begin{figure}[htb]
\centering
\begin{tabular}{cc}
\subfloat[][Degraded image with SNR=~8.55~dB (Gaussian kernel $7 \times 7$, std $1$. $x^+=20$ and $\sigma^2=9$).]{\includegraphics[scale=0.45]{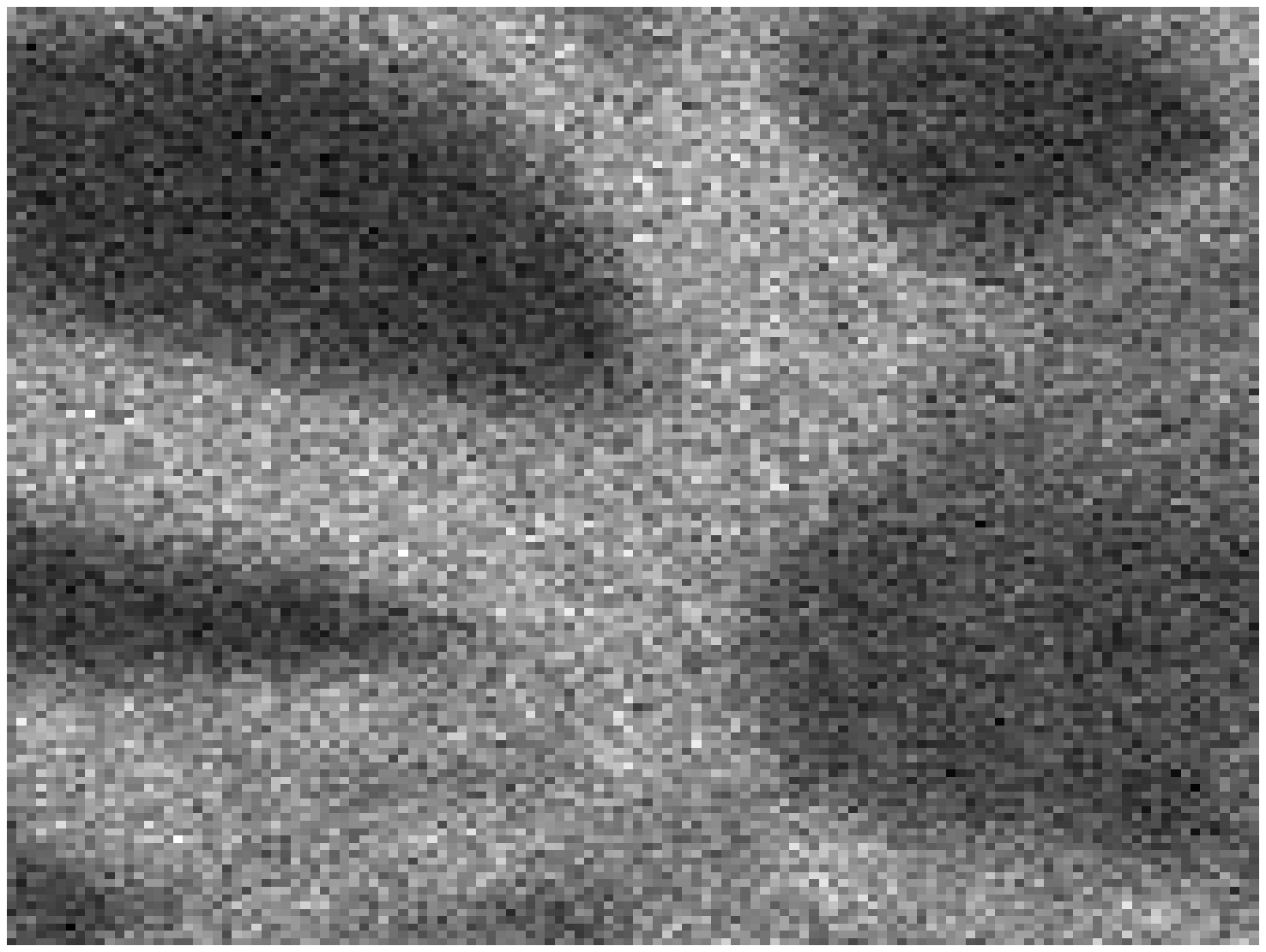} }&
 \subfloat[][Restored image with VBA approach using the diagonal approximation: SNR= 20.71 dB]{\includegraphics[scale=0.45]{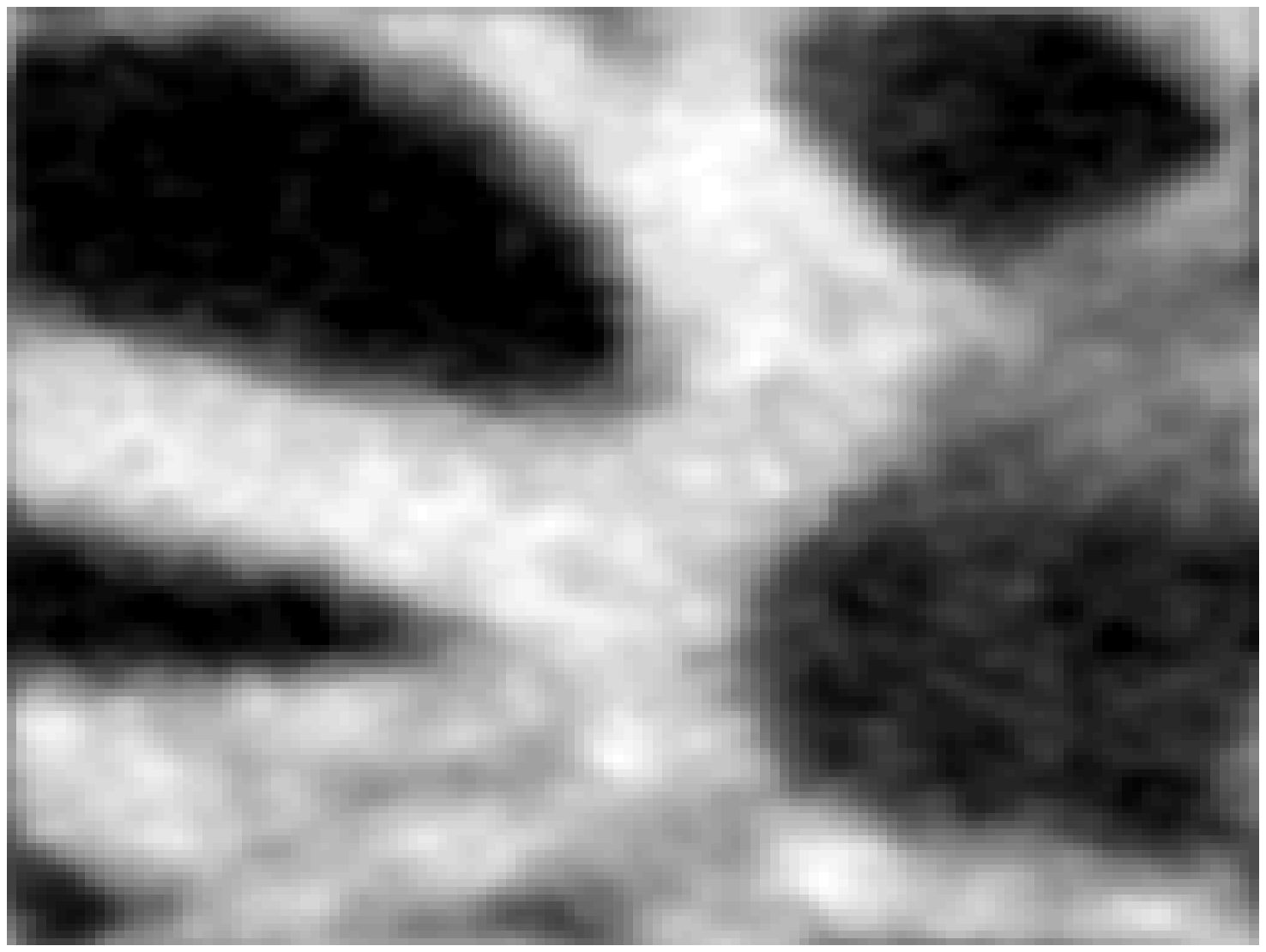} }\\
  \subfloat[][Restored image with discrepancy principle: SNR= 18.70 dB]{\includegraphics[scale=0.45]{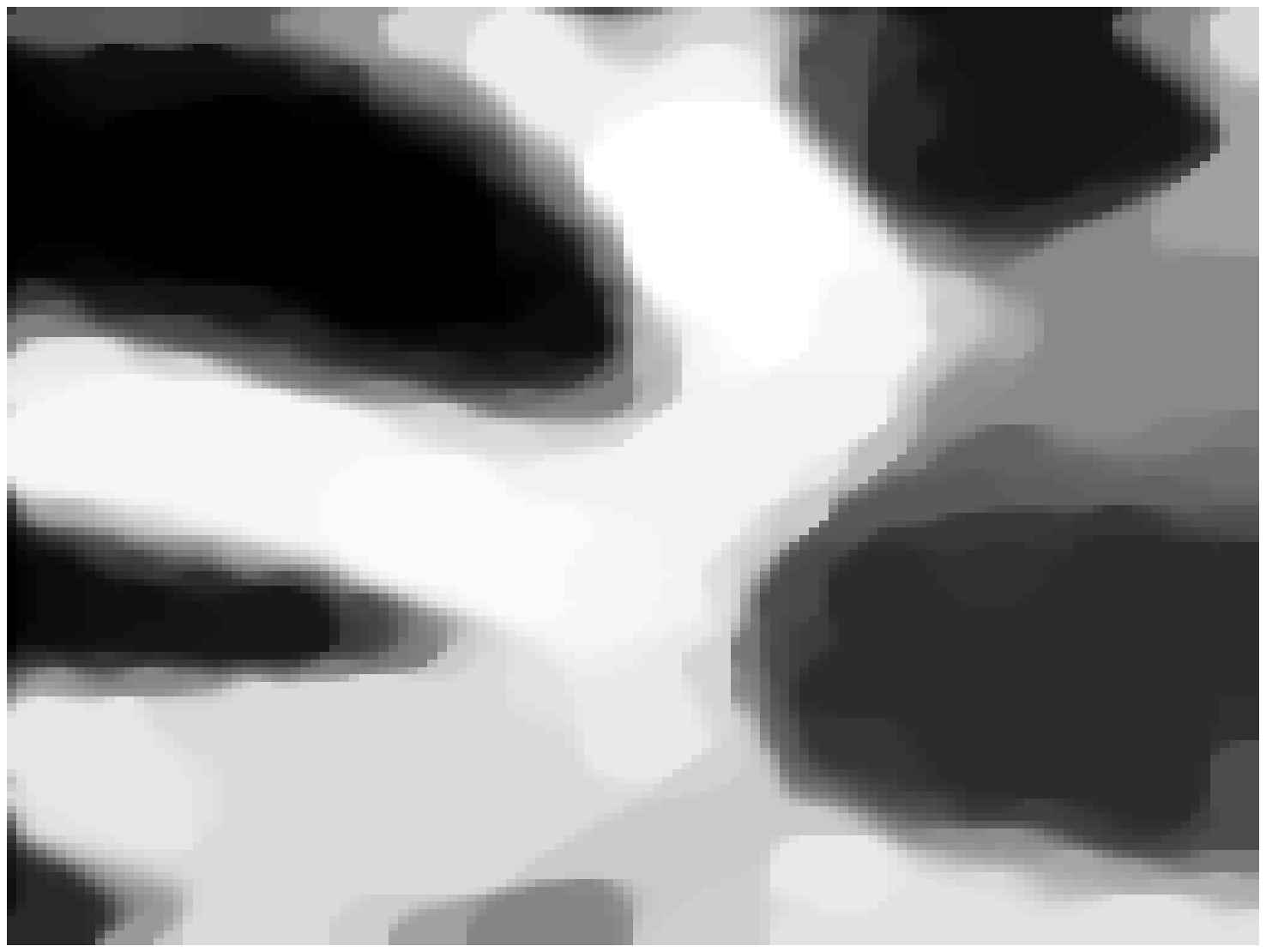} }&
   \subfloat[][Restored image with best parameter: SNR= 20.44 dB]{\includegraphics[scale=0.45]{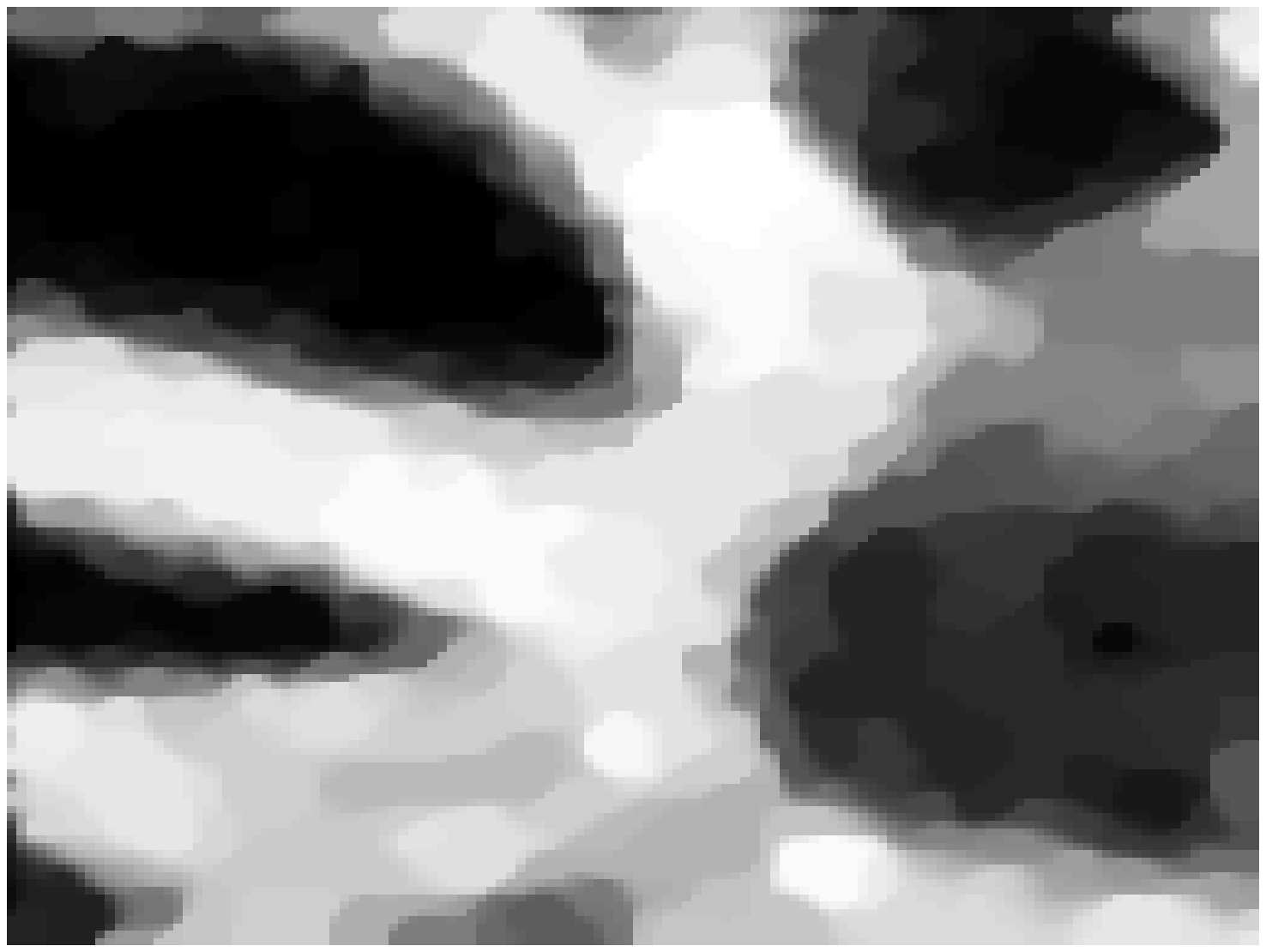} }
\end{tabular}
\caption{Restoration results for image $\bar{\mathbf{x}}_5$ using WL2 approximation.}
\label{fig5}
\end{figure}


\begin{figure}[htb]
\centering
\begin{tabular}{cc}
\subfloat[][Degraded image with SNR=~10.68~dB (Uniform kernel $3 \times 3$, $x^+=100$ and $\sigma^2=36$).]{\includegraphics[scale=0.45]{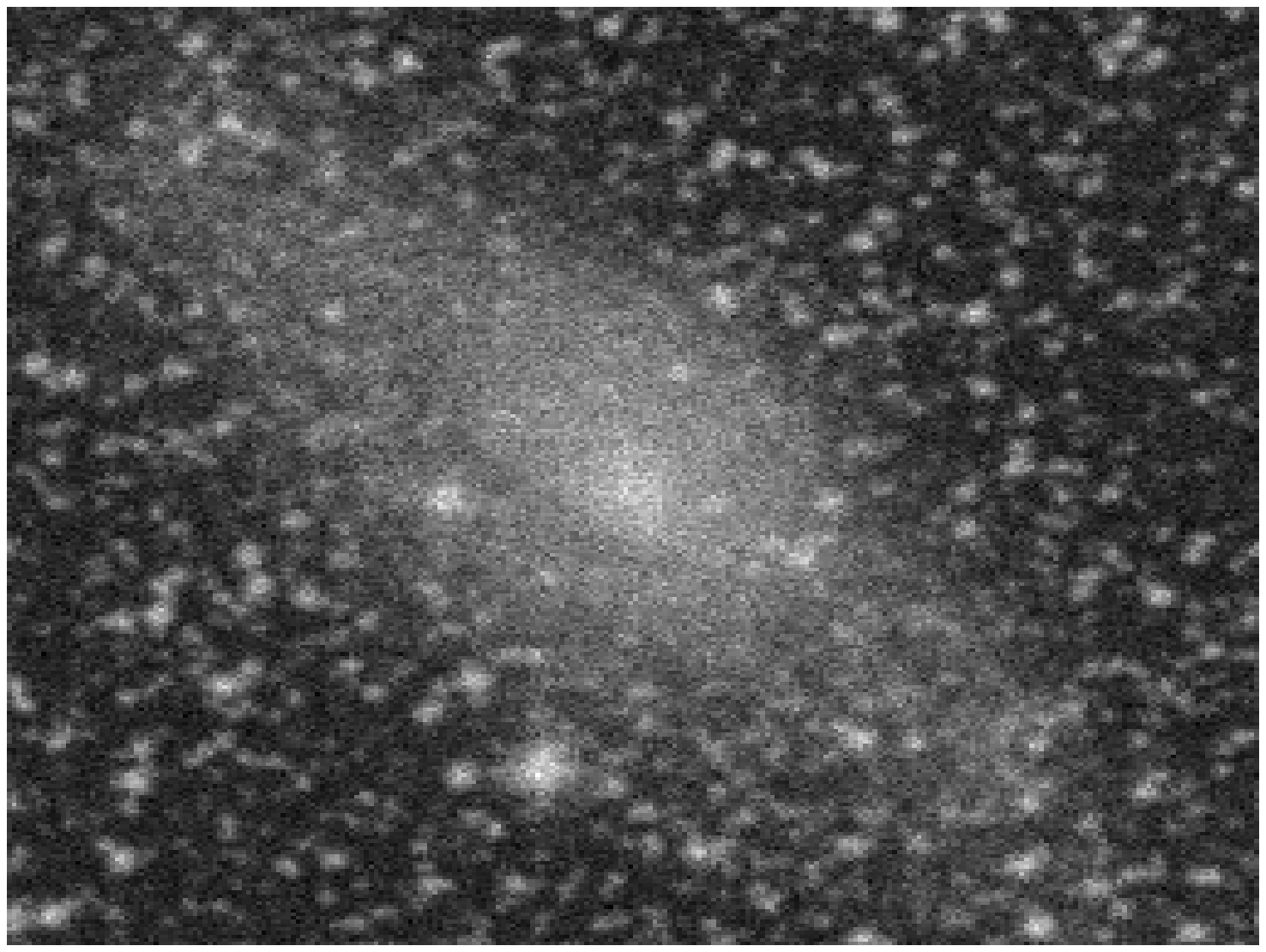} }&
 \subfloat[][Restored image with VBA approach using the Monte Carlo approximation with $640$ samples: SNR= 14.16 dB]{\includegraphics[scale=0.45]{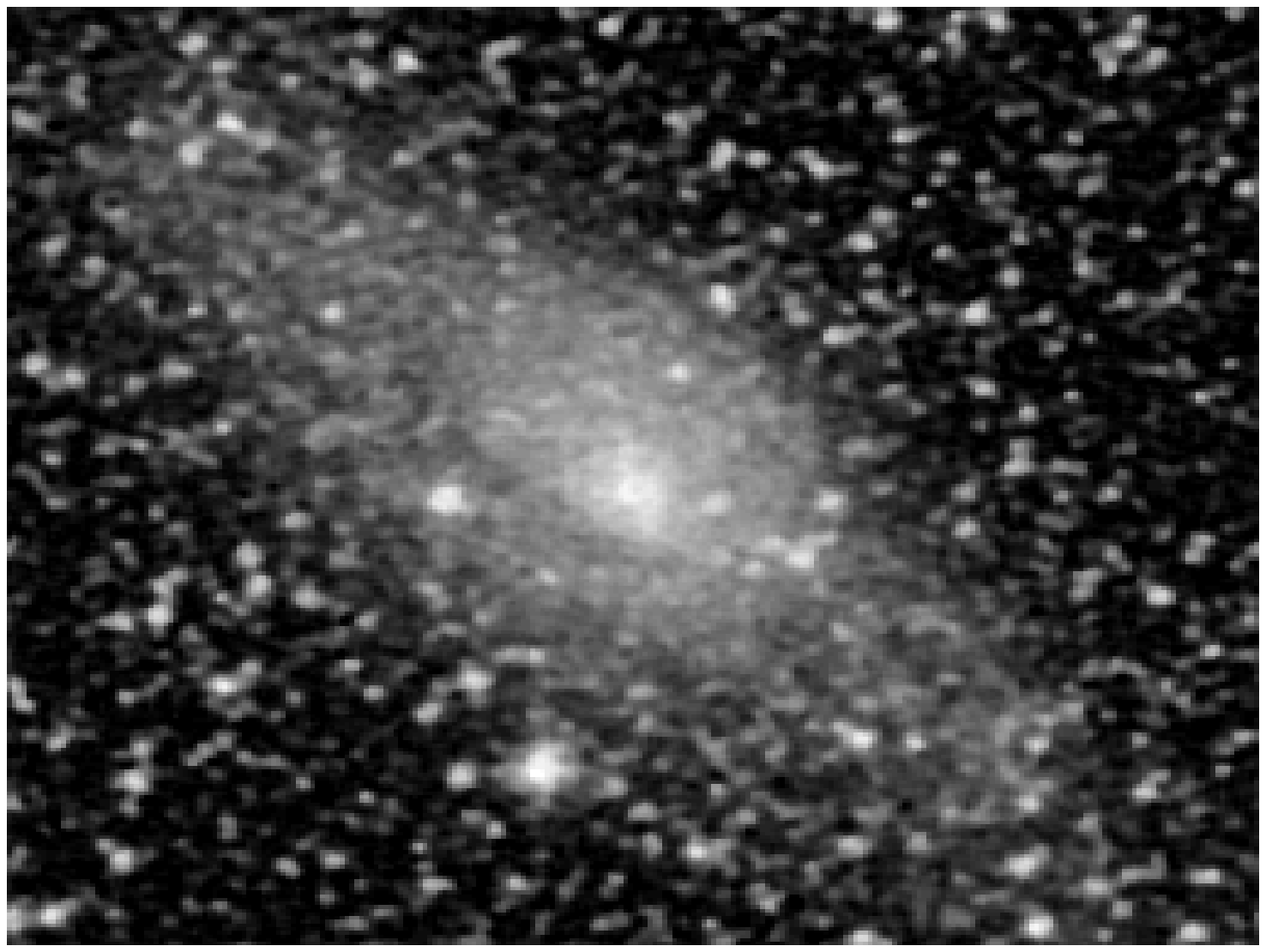} }\\
  \subfloat[][Restored image with discrepancy principle: SNR= 13.32 dB]{\includegraphics[scale=0.45]{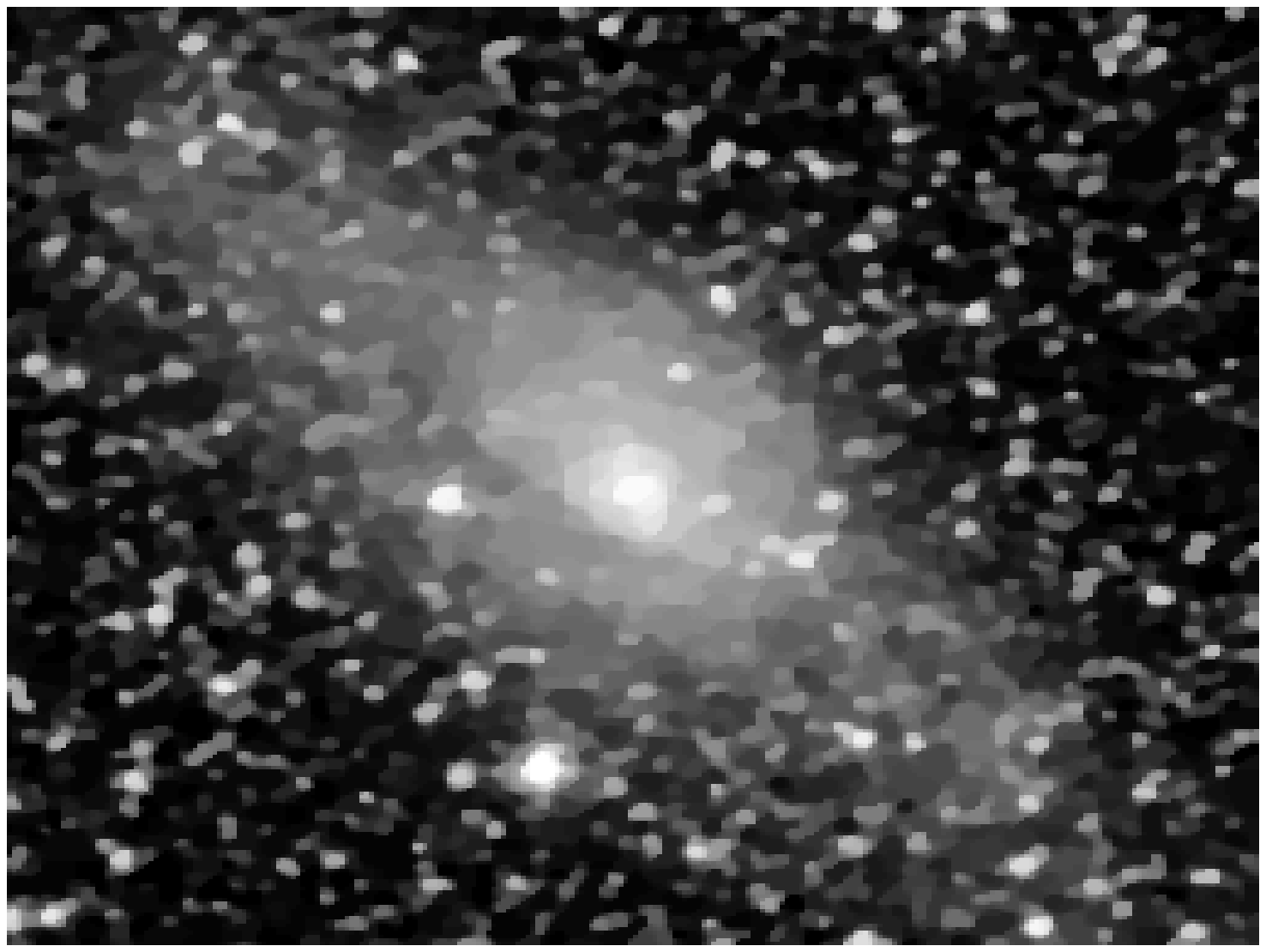} }&
   \subfloat[][Restored image with best parameter: SNR= 13.84 dB]{\includegraphics[scale=0.45]{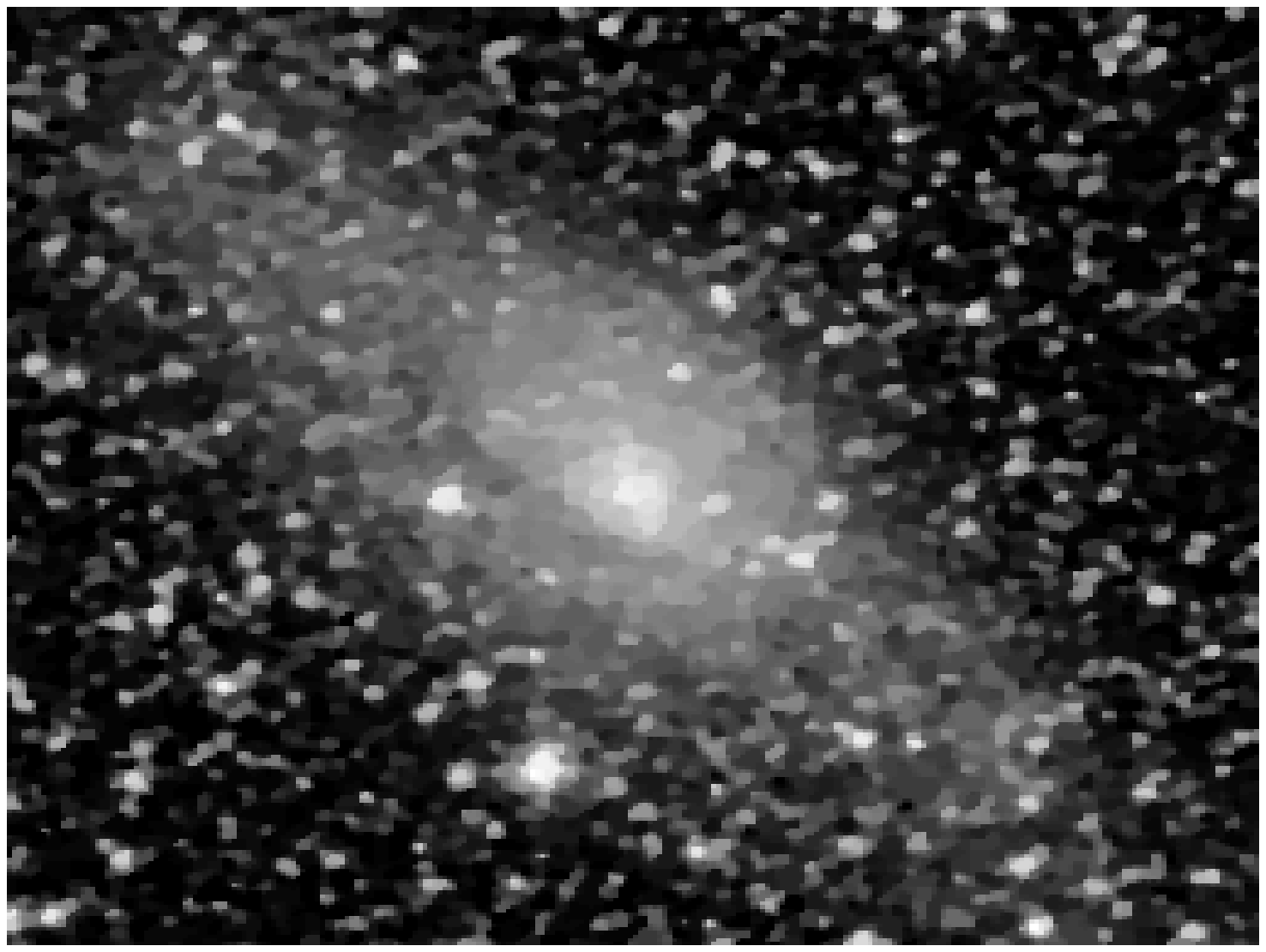} }
\end{tabular}
\caption{Restoration results for image $\bar{\mathbf{x}}_6$ using WL2 approximation.}
\label{fig6}
\end{figure}


In Figures \ref{fig1} - \ref{fig6}, we show some examples of visual results obtained with the different approaches, when the best approximation strategy for the covariance matrix is retained in the VBA method. It can be noticed that, unlike the other methods, the reconstructed images with the proposed VBA algorithm exhibit very few artifacts, without over-smoothed regions.

It should be emphasized that the problem of setting the regularization parameter for \f{MAP-based} algorithms must be carefully addressed as it highly impacts the quality of the restored image. The main advantage of our approach is that this parameter is tuned automatically without the need of the ground truth, while also often being the most competitive in terms of computation time. Furthermore, the performance of the proposed method could be further improved by using parallel implementation with more than 16 cores for the Monte Carlo approximation of the covariance matrix allowing
either generating a higher number of samples (i.e. an improved estimation error) or a reduction of the computation time. 

\f{Comparisons with image deblurring methods dedicated to a pure Poisson noise model have also been conducted. However, in our examples, they were observed to lead to poor results in terms of restoration quality, and to present a high computational time. For instance, the application of the proximal method from \cite{pustelnik2011parallel} using a TV prior and an empirical search for the regularization parameter, leads to an image with SNR equal to $12.88$ dB (computation time: $3489$ s.) on the test problem from Table \ref{t:5}, and a SNR of $18.37$ dB (computation time: $986$ s.) for the example from Table \ref{t:6}. The Plug and Play ADMM strategy from \cite{rond2015poisson} also leads to unsatisfactory results with a final SNR of $9.11$ dB (computation time: $1618$ s.) and $10.31$ dB (computation time: $204$ s.) for the examples from Table \ref{t:5} and Table \ref{t:6}, respectively. These numerical tests clearly highlight that image restoration in the presence of Poisson-Gaussian noise is challenging, and should be treated with specific methods that take into account the mixed noise model in an explicit manner.}

\subsubsection{Influence of the regularization term}
\f{The versatility of the proposed VBA method allows us to consider a large variety of regularization strategies, by defining appropriate prior operators $\mathbf{D}$. In the previous experiments, the TV prior has led to satisfactory results in terms of SNR, but a visual inspection of the restored versions of images $\mathbf{\bar{x}}_4$ and $\mathbf{\bar{x}}_6$ shows an undesirable starcasing effect. In this new set of experiments, we propose to compare these TV-based restoration results to those obtained with priors that have been recently shown to better preserve the natural features in images. Namely, we will consider the Hessian-based penalization~\cite{lefkimmiatis2012hessian}, the semi-local total variation (SLTV)~\cite{condat2014semi}, and the non-local total variation (NLTV) \cite{gilboa2008nonlocal, Chierchia_2014_nlst}. The Hessian prior operator is given, for every $j \in \left\lbrace 1,\ldots, N \right\rbrace$, by $\mathbf{D}_j\x  = \big[[\nabla^{hh} \x]_j, \sqrt{2} [\nabla^{hv}\x]_j , [\nabla^{vv} \x]_j\big]^\top \in \mathbb{R}^3$ where $\nabla^{hh}$, $\nabla^{hv}$ and $\nabla^{vv}$ model the second-order finite difference operators between neighbooring pixels, so that $S=3$ and $J=N$. The SLTV is based on differences of neighboring gradient values and is computed here using a $6$-pixels neighborhood, hence $S = 12$ and $J=N$. The NLTV prior operator is defined at every pixel position by a collection of weighted discrete gradient differences operators across a large set of directions, the weights being calculated according to a rough estimate of the target image. In our experiments, $49$ different directions are chosen and the corresponding weights are precomputed from the restored images using VBA with the TV prior and the diagonal approximation of the covariance matrix. As a result, $S=98$ and $J=N$ in that case. The SPoiss likelihood is chosen for the data fidelity term as it was observed to lead to the best tradeoff in terms of image quality and computational time in the previous set of tests. Table \ref{t:10} summarizes the obtained results for all the six test images, using the different considered priors. Complementary to these numerical results, Figures \ref{fig55} and \ref{fig56} show the visual improvements resulting from the different priors. One can observe that the NLTV prior gives in most experiments the best results in terms of SNR while the other priors perform quite similarly. In particular, for the image $\overline{\mathbf{x}}_5$, the gain in terms of SNR exceeds $2$ dB when using the NLTV prior, compared to the other regularization strategies. Note that despite small differences in SNR between the results obtained with the TV, SLTV and the Hessian regularizers, the Hessian and the SLTV appear to offer good alternatives in terms of visual quality to the TV prior for images that consist mostly of ridges and smooth transition of intensities. Indeed, it can be seen in Figure \ref{fig56} that the smooth piecewise constant areas are better reconstructed and the sharpness of edges is better maintained using these two priors. 
For textured images, Figure \ref{fig55} shows that the NLTV prior gives rise to less blurry images than the SLTV and Hessian priors and seems to reduce again the undesired staircase effect arising from TV regularization. However, as shown in Table \ref{t:10}, the approaches based on Hessian, SLTV and NLTV take much more computation time than the TV based approach in most test cases. Our suggestion would be to use the VBA approach with the TV prior and the diagonal approximation of the covariance matrix to obtain a satisfactory result in a low computational cost, and to use VBA with NLTV prior, using the former TV-based result to approximate the NLTV weights, in order to further improve the visual quality of the restored image. 
\begin{table}[t]
\tiny
\centering
\resizebox{0.9\textwidth}{!}{
\renewcommand{\arraystretch}{1.3}
\begin{tabular}{|l|l||l|l|l|l|l|l|}
\cline{3-8}
 \multicolumn{2}{c|}{ }& $\overline{\mathbf{x}}_1$& $\overline{\mathbf{x}}_2$ &$\overline{\mathbf{x}}_3$ &$\overline{\mathbf{x}}_4$ &$\overline{\mathbf{x}}_5$ &$\overline{\mathbf{x}}_6$ \\
\hhline{--|=|=|=|=|=|=|}
\multirow{3}{*}{\shortstack{ TV }}
&SNR & 10.20&19.12 & 12.36 & 13.90 & 20.67 & \textbf{14.19}\\
\cline{2-8}
&SSIM& 0.6088& 0.6930 & 0.4684 & 0.5769 & 0.7790 & \textbf{0.7650}\\
\cline{2-8}
&Time (s.)& 3507& 1828& 2051 & 34 & 184 & \textbf{479} \\
\hline
\hline
\multirow{3}{*}{\shortstack{ Hessian }}
&SNR & 10.17&\textbf{19.41} & 12.21& 13.56 & 20.57 & 14.05\\
\cline{2-8}
&SSIM& 0.6016 & \textbf{0.7300} & 0.4618 & 0.5501 & 0.8392&0.7643\\
\cline{2-8}
&Time (s.)& 8600& \textbf{5404}& 6974 & 5058 & 744 & 1332 \\
\hline
\hline
\multirow{3}{*}{\shortstack{ SLTV }}
&SNR & 10.32 & 19.26 & 12.26 & 13.53 & 20.62 & 13.93\\
\cline{3-8}
&SSIM&  0.6006& 0.7189 & 0.4656 &0.5478 & 0.8368&0.7578\\
\cline{2-8}
&Time (s.)& 6359& 2923 & 3497 & 1003 & 375 & 738 \\
\hline
\hline
\multirow{3}{*}{\shortstack{ NLTV }}
&SNR & \textbf{10.35}& 19.10& \textbf{12.46} & \textbf{14.09} & \textbf{22.89} & 13.95\\
\cline{2-8}
&SSIM&\textbf{0.4644} & 0.7075 & \textbf{0.4704} & \textbf{0.5812} & \textbf{0.7972} & 0.7530\\
\cline{2-8}
&Time (s.)& \textbf{7821}& 338 & \textbf{4602} & \textbf{8595} & \textbf{807} & 1547 \\
\hhline{|-|-------|}
\end{tabular}
}
\caption{Restoration results for the considered test images using the Spoiss likelihood and different regularization functions.}
\label{t:10}
\end{table}

\begin{figure}[htb]
\centering
\begin{tabular}{cc}
\subfloat[][Restored image with a TV prior: SNR= 13.90 dB]{\includegraphics[scale=0.45]{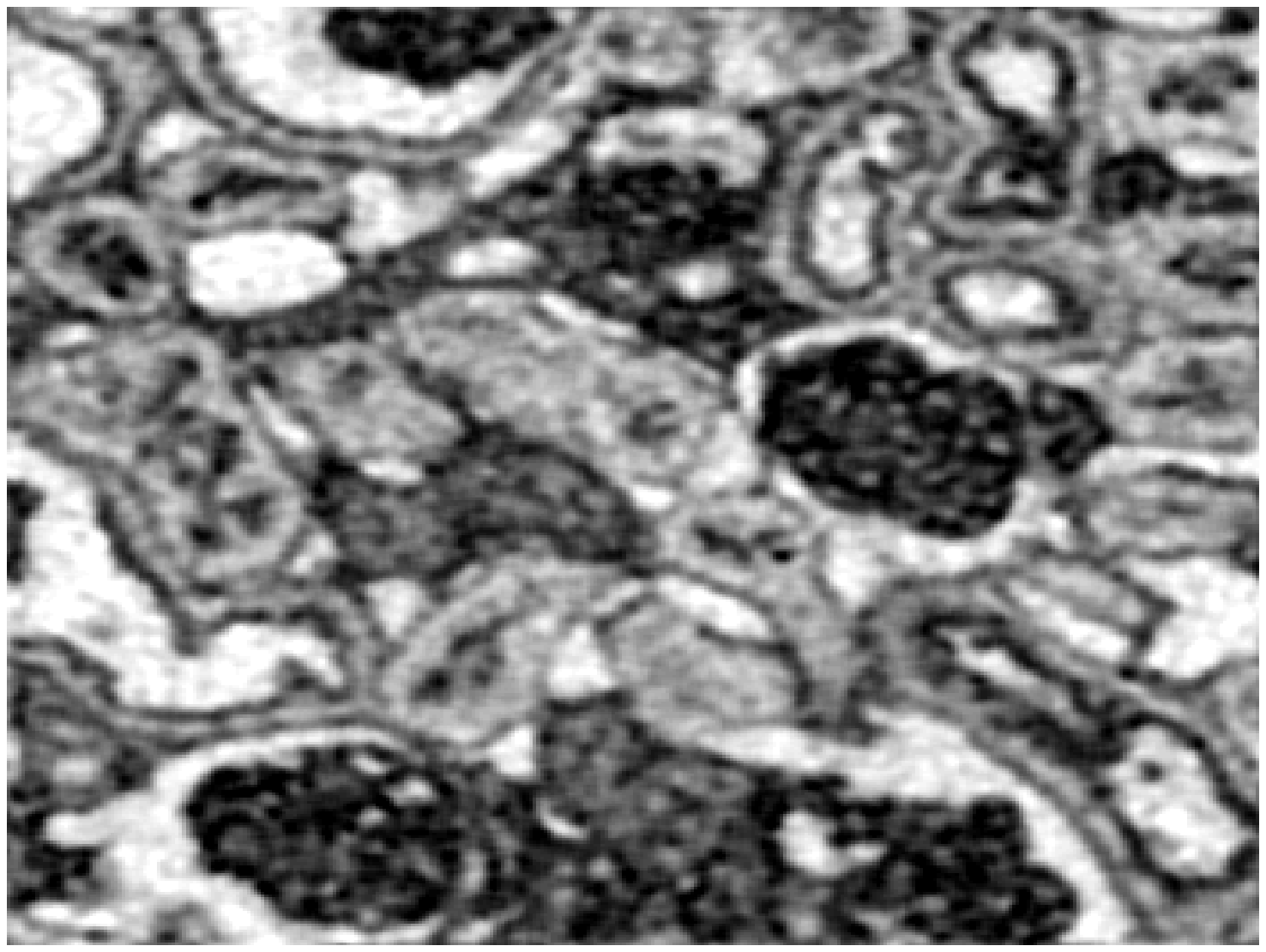} }&
 \subfloat[][Restored image with a Hessian prior: SNR= 13.56 dB]{\includegraphics[scale=0.45]{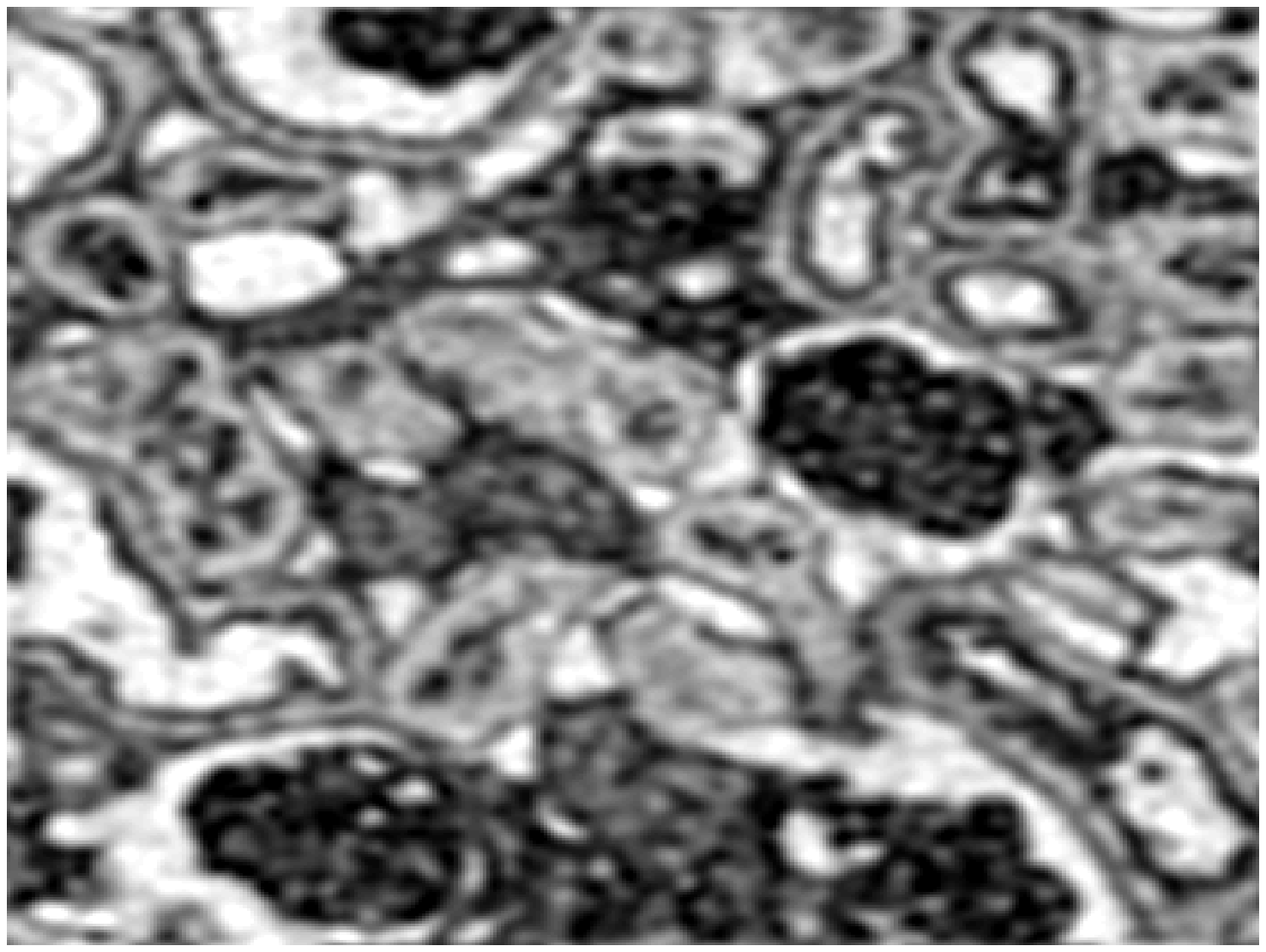} }\\
  \subfloat[][Restored image with a SLTV prior: SNR= 13.53 dB]{\includegraphics[scale=0.45]{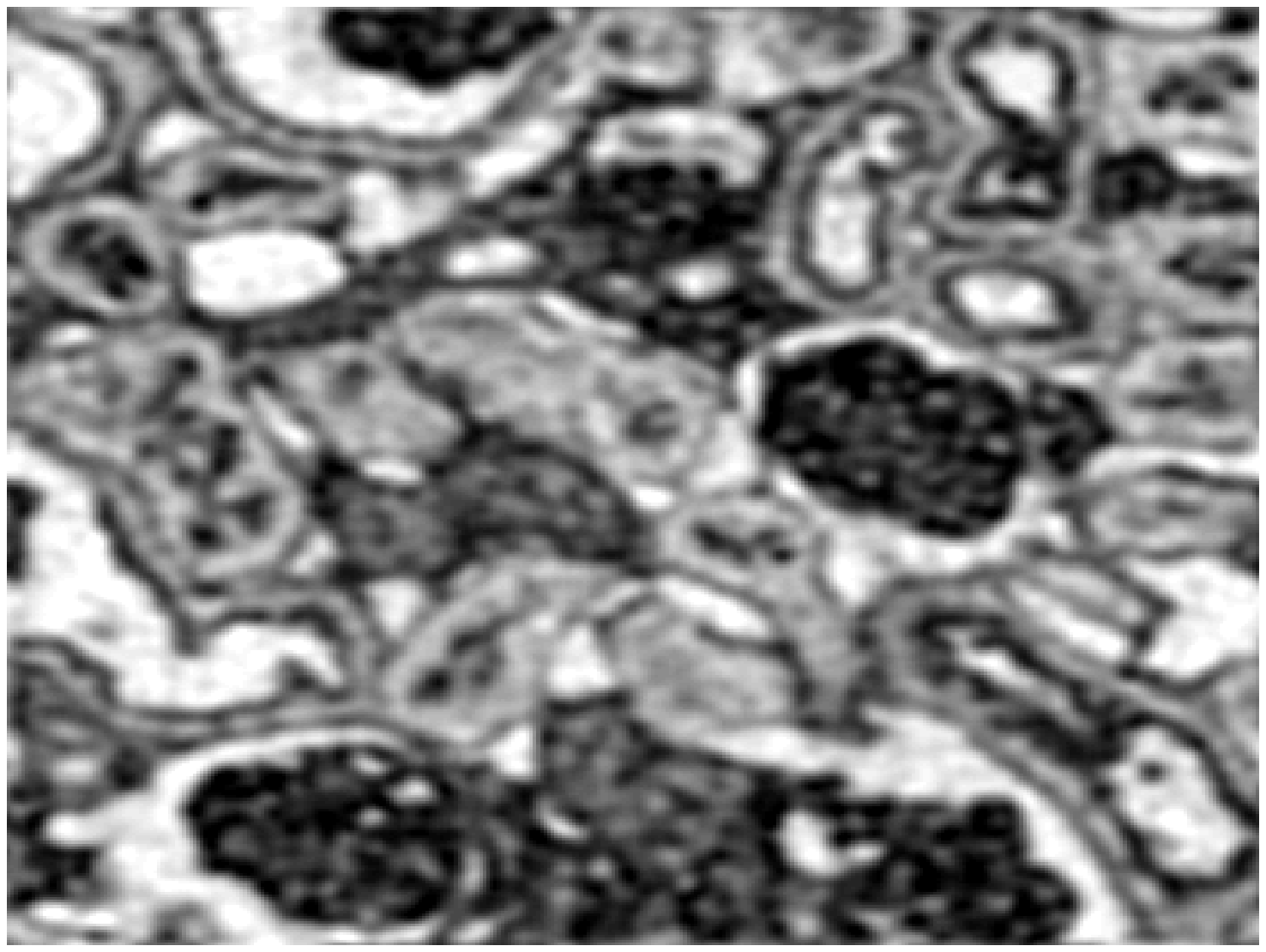} }&
   \subfloat[][Restored image with a NLTV prior: SNR= 14.09 dB]{\includegraphics[scale=0.45]{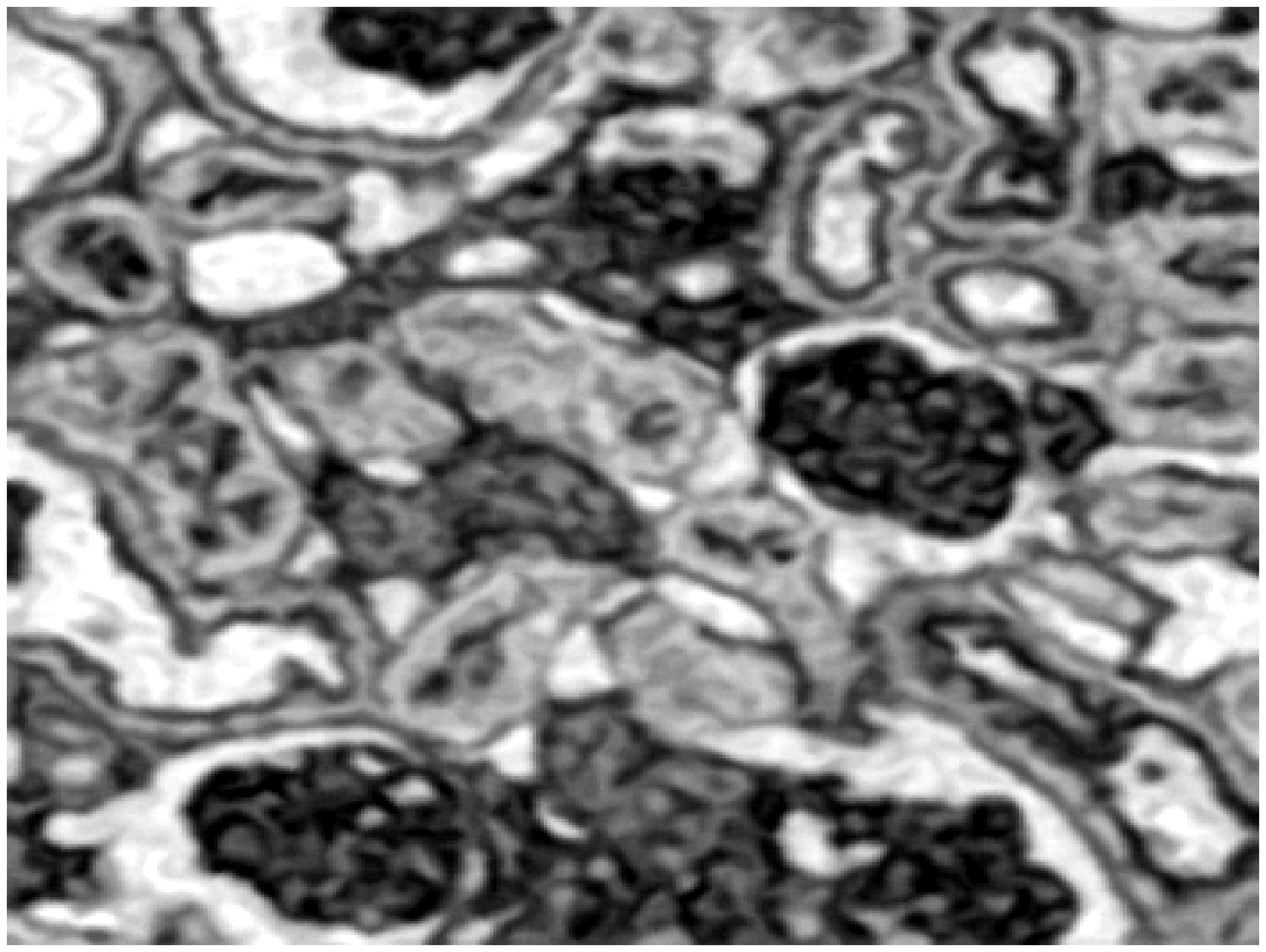} }
\end{tabular}
\caption{Restoration results for image $\bar{\mathbf{x}}_4$ using SPoiss likelihood and different regularization functions.}
\label{fig55}
\end{figure}

\begin{figure}[htb]
\centering
\begin{tabular}{cc}
\subfloat[][Restored image with a TV prior: SNR= 20.67 dB]{\includegraphics[scale=0.45]{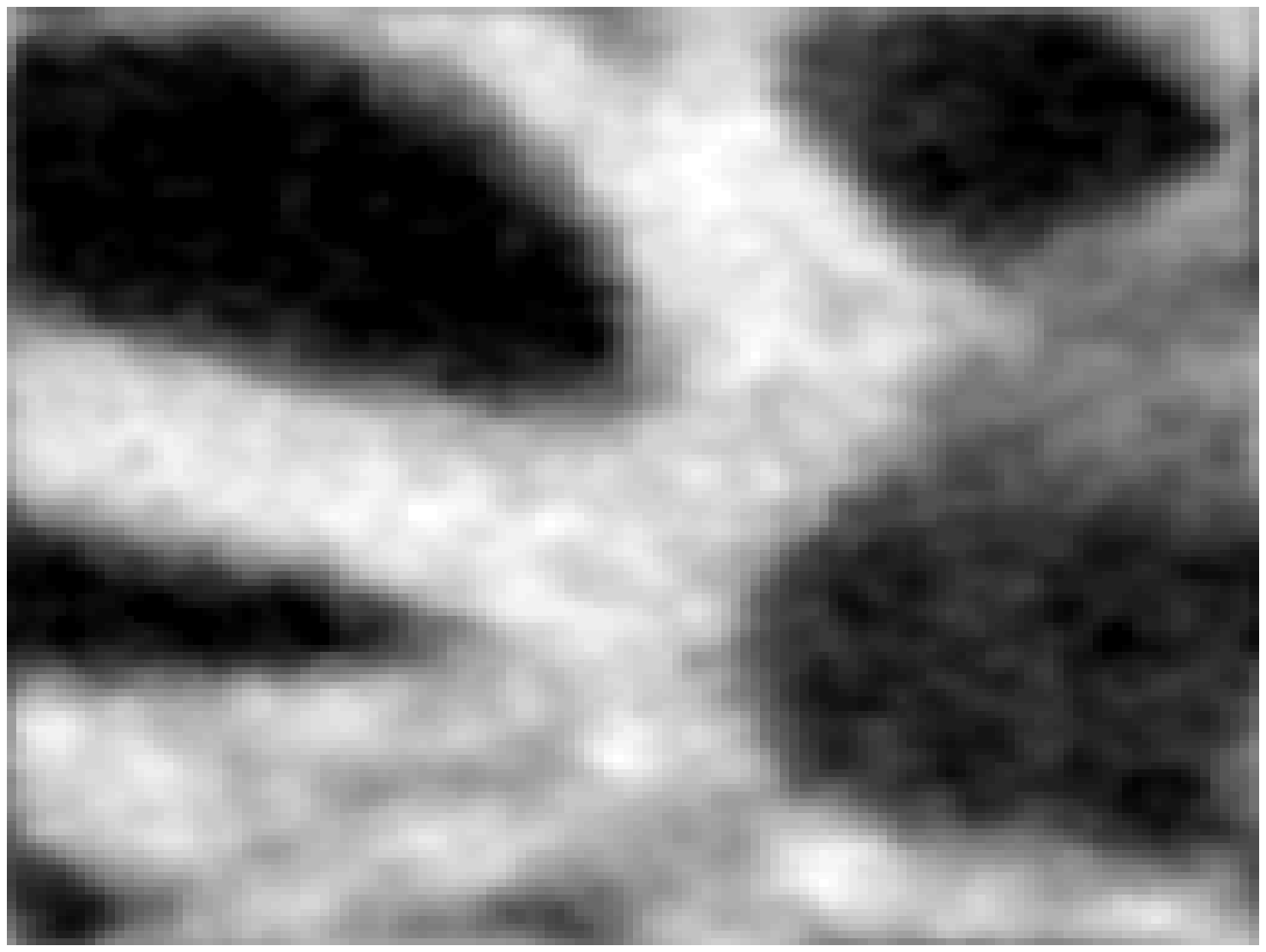} }&
 \subfloat[][Restored image with a Hessian prior: SNR= 20.57 dB]{\includegraphics[scale=0.45]{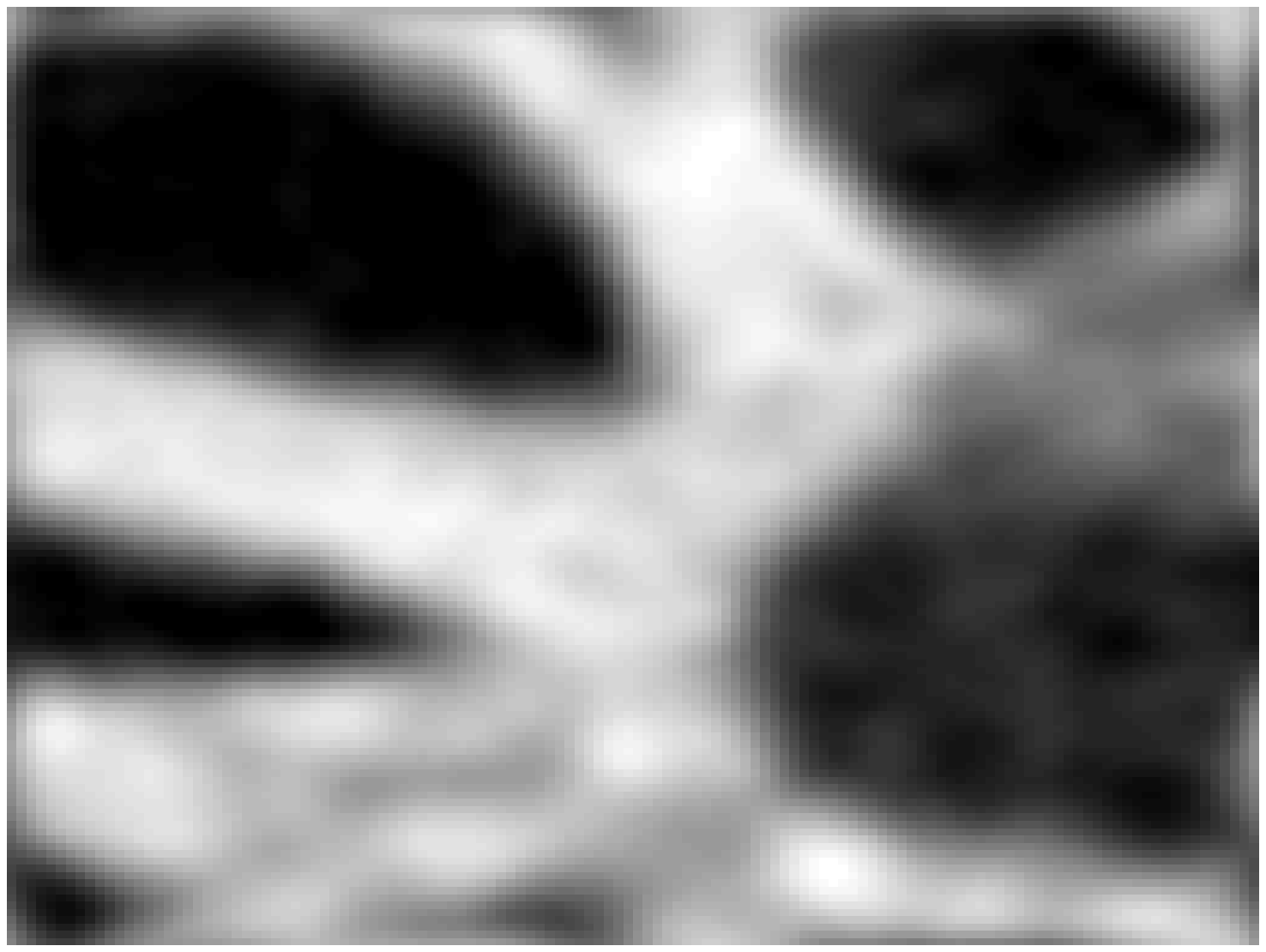} }\\
  \subfloat[][Restored image with a SLTV prior: SNR= 20.62 dB]{\includegraphics[scale=0.45]{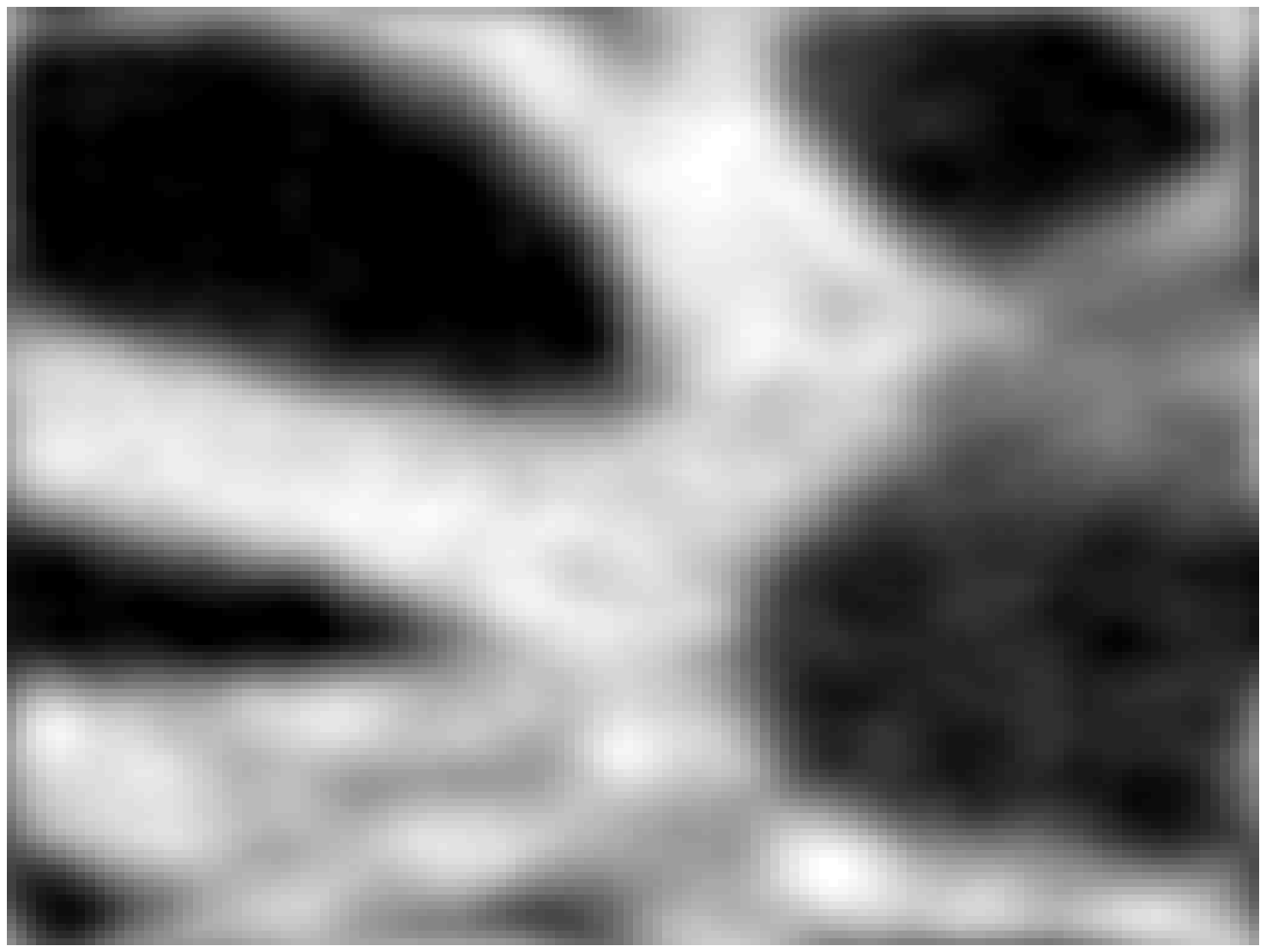} }&
   \subfloat[][Restored image with a NLTV prior: SNR= 22.89 dB]{\includegraphics[scale=0.45]{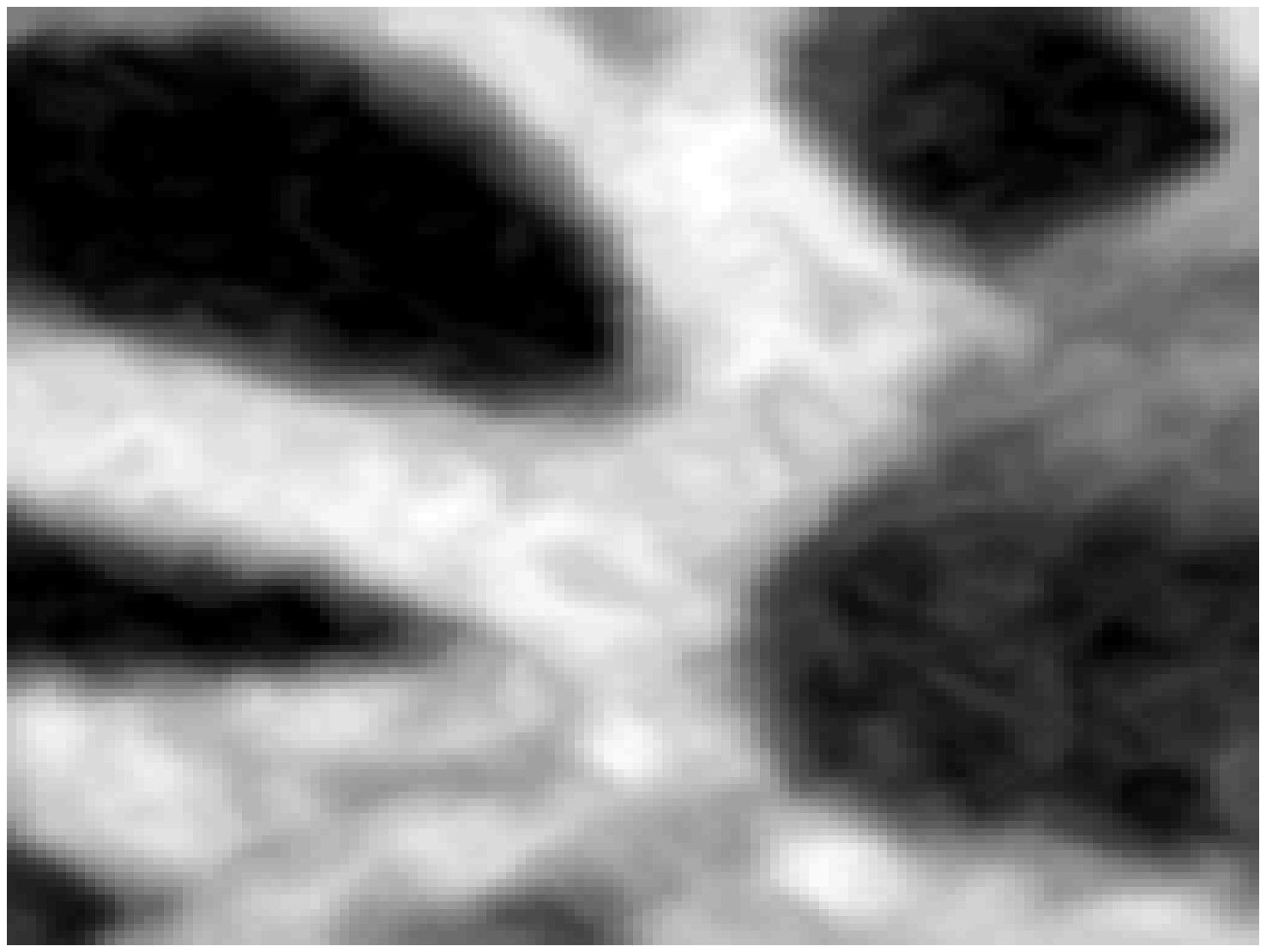} }
\end{tabular}
\caption{Restoration results for image $\bar{\mathbf{x}}_5$ using SPoiss likelihood and different regularization functions.}
\label{fig56}
\end{figure}

}

\section{Conclusion} \label{sec:Conclu}
In this paper, we have proposed a variational Bayesian approach for solving signal recovery problems in the presence of non-Gaussian noise. Our approach has two main advantages. First, 
the regularization
parameter 
is tuned automatically during the recovery process.
Second, the designed method is applicable to a wide range of prior distributions and data fidelity terms.
Simulations carried out on various images corrupted with mixed Poisson-Gaussian noise have shown that the proposed strategy constitutes a competitive solution for low computational and high-quality restoration of images compared with state-of-the art methods.

\appendix
\section{Proof of Proposition~\ref{prop_GY}}
\label{appA}

Let $i \in \{1, \ldots, M\}$. Let us define $g_i: \eR \rightarrow \eR$ such that
\begin{equation}
(\forall v \in \eR) \quad g_i(v) = \frac{v^2}{2} - \frac{\phi_i(v;y_i)}{\mu_i(y_i)}. \label{eq:gi}
\end{equation}
\f{According to Assumption \ref{assum1}, $g_i$ is convex, proper and lower semi-continuous (lsc)}. Its conjugate function \cite[Chapter 13]{bauschke2011convex} reads:
\begin{align}
(\forall w \in \eR) \quad g_i^*(w) 
& = \sup_{v \in \eR} \left(v w -g_i(v) \right) \label{eq:gistar}\\
& = \sup_{v \in \eR} \left(v w +\frac{\phi_i(v;y_i)}{\mu_i(y_i)} -\frac{v^2}{2} \right)\\
& = \sup_{v \in \eR} \left(- \frac{1}{2}(v-w)^2 + \frac{\phi_i(v;y_i)}{\mu_i(y_i)} \right) + \frac{w^2}{2}.
\end{align}
According to Definition \eqref{eq:varsigma},
\begin{equation}
(\forall w \in \eR) \quad g_i^*(w) = \varsigma_i(w;y_i) + \frac{w^2}{2}. \label{eq:gstar}
\end{equation}
The conjugate of $g_i^*$ is
\begin{align} \label{eq:gconj}
(\forall v \in \eR) \quad g_i^{**}(v) & = \sup_{w \in \eR} \big(vw -g_i^*(w) \big)\\
& = \sup_{w \in \eR} \left(vw - \frac{w^2}{2} -  \varsigma_i(w;y_i)  \right) \nonumber\\
& = \sup_{w \in \eR} \left( - \frac{1}{2}(v-w)^2   -  \varsigma_i(w;y_i) \right) + \frac{v^2}{2} \nonumber\\
& = - \inf_{w \in \eR} \left(\frac{1}{2}(v-w)^2 +  \varsigma_i(w;y_i) \right) +\frac{v^2}{2}.
\end{align}
\f{Since $g_i$ is convex, proper and lsc \cite[Theorem 13.32]{bauschke2011convex}}, $g_i = g_i^{**}$ so that
\begin{equation}
(\forall v \in \eR) \quad - \frac{\phi_i(v;y_i)}{\mu_i(y_i)} = - \inf_{w \in \eR} \left(\frac{1}{2}(v-w)^2 +\varsigma_i(w;y_i) \right)
\end{equation}
which is equivalent to
\begin{equation}
(\forall v \in \eR) \quad \phi_i(v;y_i) = \mu_i(y_i) \inf_{w \in \eR} \left(\frac{1}{2}(v-w)^2 + \varsigma_i(w;y_i)\right), \label{eq:psi_zeta}
\end{equation}
so that \eqref{eq:GY} holds. 

For every $v\in \mathbb{R}$, let
\begin{align}
\widehat{w}_i(v) = g_i'(v). \label{eq:what0}
\end{align}
The function \f{$g_i$ being convex, proper and lsc, according to \cite[Corollary 16.24]{bauschke2011convex}, 
the above relation can be reexpressed by making use of the subdifferential  $\partial g_i^*$ of the convex function $g_i^*$
(see \cite[Chapter 16]{bauschke2011convex} for more details).
More precisely, \eqref{eq:what0} is equivalent
to 
\begin{align}
v \in \partial g_i^*\big(\widehat{w}_i(v)\big). \label{eq:what2}
\end{align}
According to Fermat's rule \cite[Theorem 16.2]{bauschke2011convex}, \eqref{eq:what2} is a necessary and sufficient condition for $ \hat{w}_i(v)$ to be
a minimizer of the convex function $w \mapsto g_i^*(w) - v w$.}


This minimizer is unique since $\widehat{w}_i(v)$ is uniquely
defined by \eqref{eq:what0}. We have therefore established that
\begin{equation}
\widehat{w}_i(v) = \argmax{w \in \eR}{\big(vw -g_i^*(w) \big)}.
\end{equation}
The definition of $g_i$ in \eqref{eq:gi} shows that \eqref{eq:what0} also reads
\begin{equation}
\widehat{w}_i(v) = v - \frac{1}{\mu_i(y_i)} \phi_i'(v ; y_i).  \label{eq:what}
\end{equation}
According to \eqref{eq:gconj}, it is straightforward that $\widehat{w}_i(v)$ also reaches the infimum in \eqref{eq:psi_zeta}.
Hence the result by using \eqref{eq:what} and \eqref{e:defTi}.
 \bibliographystyle{IEEEbib}
 \bibliography{ref_icassp2015}

\end{document}